\documentclass{amsart}
\usepackage{wasysym}
\usepackage{amsmath}
\usepackage{amssymb}
\usepackage{amsthm}
\usepackage{tikz}
\usetikzlibrary{matrix,arrows,backgrounds,shapes.misc,shapes.geometric,patterns,calc,positioning}
\usetikzlibrary{calc,shapes}
\usepackage{wrapfig}
\usepackage{epsfig}  
\usepackage[margin=1.3in]{geometry}
\usepackage{color}	 %
\input{xy}
\xyoption{poly}
\xyoption{2cell}
\xyoption{all}
\newtheorem{thm}{Theorem}[section]
\newtheorem{prop}[thm]{Proposition}
\newtheorem{lemma}[thm]{Lemma}
\newtheorem{corollary}[thm]{Corollary}
\newtheorem{conjecture}[thm]{Conjecture}
    
\newtheorem*{thm*}{Theorem}

\usepackage{hyperref}
\theoremstyle{definition}
\newtheorem{definition}[thm]{Definition}
\theoremstyle{remark}
\newtheorem{remark}[thm]{Remark}
\newtheorem{example}[thm]{Example}
\numberwithin{equation}{section}
\newcommand{\cals}{\mathcal{S}}
\newcommand{\calp}{\mathcal{P}}
\newcommand{\calc}{\mathcal{C}}

\newcommand{\cald}{\mathcal{D}}

\newcommand{\calg}{\mathcal{G}}

\newcommand{\za}{\alpha}
\newcommand{\zb}{\beta}
\newcommand{\zd}{\delta}

\newcommand{\zg}{\gamma}
\newcommand{\zG}{\Gamma}

\newcommand{\zs}{\sigma}

\newcommand{\zL}{\Lambda}
\newcommand{\zO}{\Omega}
\newcommand{\kb}{\Bbbk}

\newcommand{\Hom}{\textup{Hom}}
\newcommand{\add}{\textup{add}}
\newcommand{\coker}{\textup{coker}}

\newcommand{\rad}{\textup{rad}\,}

\newcommand{\Ext}{\textup{Ext}}

\newcommand{\cmp}{\textup{CMP}}
\newcommand{\scmp}{\underline{\textup{CMP}}}

\newcommand{\scat}{\zO(\textup{mod}\,A)} 

\newcommand{\BM}{\Theta} 

\newcommand{\diag}{\textup{Diag}}
\newcommand{\diags}{\textup{Diag}(\cals)}

\newcommand{\arc}{\textup{Arc}}
\newcommand{\arcp}{\textup{Arc}(\calp)}

\newcommand{\zgg}{\tilde \zg}
\newcommand{\zdd}{\tilde \zd}
\begin{document}

\title[On Gorenstein algebras of finite Cohen-Macaulay type]{On Gorenstein algebras of finite Cohen-Macaulay type: dimer tree algebras and their skew group algebras} 
\author{Ralf Schiffler}
\thanks{The first author was supported by the NSF grant  DMS-2054561.}
\address{Department of Mathematics, University of Connecticut, Storrs, CT 06269-1009, USA}
\email{schiffler@math.uconn.edu}

\author{Khrystyna Serhiyenko}
\thanks{The second author was supported by the NSF grant DMS-2054255.}
\address{Department of Mathematics, University of Kentucky, Lexington, KY 40506-0027, USA}
\email{khrystyna.serhiyenko@uky.edu}


\maketitle
\setcounter{tocdepth}{2}

\begin{abstract} Dimer tree algebras are a class of non-commutative Gorenstein algebras of Gorenstein dimension 1.
In  previous work we showed that the stable category of Cohen-Macaulay modules of a dimer tree algebra $A$ is a 2-cluster category of Dynkin type $\mathbb{A}$. Here we show that, if $A$ has an admissible action by the group $G$ with two elements, then the stable Cohen-Macaulay category of the skew group algebra $AG$ is a 2-cluster category of Dynkin type $\mathbb{D}$. This result is reminiscent of and inspired by a result by Reiten and Riedtmann, who showed that for an admissible $G$-action on the path algebra of type $\mathbb{A}$ the resulting skew group algebra is of type $\mathbb{D}$.  Moreover, we provide a geometric model of the syzygy category of $AG$ in terms of a punctured polygon $\mathcal{P}$ with a checkerboard pattern in its interior, such that the 2-arcs in $\mathcal{P}$ correspond to indecomposable syzygies in $AG$ and 2-pivots correspond to morphisms. In particular, the dimer tree algebras and their skew group algebras are Gorenstein algebras of finite Cohen-Macaulay type $\mathbb{A}$ and $\mathbb{D}$ respectively. We also provide examples of types $\mathbb{E}_6,\mathbb{E}_7,$ and $\mathbb{E}_8$.
\end{abstract}
\tableofcontents

\section{Introduction}
We introduced the class of dimer tree algebras in \cite{SS3}. A \emph{dimer tree algebra} $A$ is the Jacobian algebra $\textup{Jac}(Q,W)$ of a quiver $Q$ without loops and 2-cycles together with a canonical potential $W$ satisfying the following two conditions. Every arrow of $Q$ lies in at least one oriented cycle, and the dual graph of $Q$ is a tree. The latter condition explains half of the terminology ``dimer tree''. The other half stems from the fact that every dimer tree algebra induces a dimer model on the disk, also known as Postnikov diagram. These dimer models appear in cluster structures on Grassmannians as well as in mathematical physics. For example, the Jacobian algebras arising from the coordinate rings of the Grassmannians $\textup{Gr}(3,n)$ are dimer tree algebras. Dimer models and their algebras have been studied extensively, see \cite{HK, Po,  JKS, BKM, Pr} and the references therein; for their connection to homological mirror symmetry, see \cite{Bocklandt}.

Dimer tree algebras are non-commutative Gorenstein algebras of Gorenstein dimension one \cite{KR}. Therefore the category of (maximal) Cohen-Macaulay modules $\cmp A$ of a dimer tree algebra $A$ is equivalent to the category of syzygies over $A$. Furthermore, the stable category $\scmp A$ is a triangulated 3-Calabi-Yau category \cite{KR} that is equivalent to the singularity category of $A$ \cite{Bu}. 

In our previous work \cite{SS3,SS4}, we introduce a derived invariant, the \emph{total weight} of a dimer tree algebra $A$. We show that the total weight is an even integer $2N$. We then construct a regular $(2N)$-gon $\cals$ equipped with a certain checkerboard pattern and show that there are equivalences of triangulated 3-Calabi-Yau categories
\begin{equation}\label{equiv}\scmp A \cong \diag\, \cals \cong \calc^2_{\mathbb{A}_{N-2}}\end{equation}
between the stable Cohen-Macaulay category $\scmp A$ of $A$, the combinatorial category of 2-diagonals $\diag \,\cals$ in $\cals$ and the 2-cluster category  $\calc^2_{\mathbb{A}_{N-2}}=\cald^b(\textup{mod}\,kQ)/\tau^{-1}[2]$ of Dynkin type $ {\mathbb{A}_{N-2}}$.  The latter equivalence was shown earlier in \cite{BM1}. Let us point out that we use the notation of Thomas \cite{T} for the higher cluster categories which is different from that of Iyama \cite{AIR}. In our notation the original cluster category of \cite{BMRRT} is the 1-cluster category.

In particular the number of indecomposable Cohen-Macaulay modules is finite, and equal to $N(N-2)$.

In the commutative case, the problem of classifying commutative Gorenstein rings $R$ of finite Cohen-Macaulay type has been studied in the 80s by several authors \cite{AV,Au,BGS, Es, Kn}.
In the case where $R$ has Krull dimension two, the classification is in terms of Dynkin diagrams of type $\mathbb{A,D,E}$. In fact, the stable Cohen-Macaulay category is equivalent to the 0-cluster category $\cald^b(\textup{mod}\,kQ)/\tau^{-1}[0]$, where $Q$ is a Dynkin quiver of type $\mathbb{A,D,E}$. Much more recently, in \cite{AIR}, it was shown that the stable Cohen-Macaulay category over a Gorenstein isolated singularity is equivalent to some higher cluster category, generally not of finite type. It is interesting to note that their result also applies to the centers of the Jacobian algebras arising from consistent dimer models on a torus; the Cohen-Macauley category of the center is shown to be equivalent to the 1-cluster category of an algebra obtained by cutting arrows of a perfect matching from the quiver and removing a source vertex of the resulting quiver.

For non-commutative Gorenstein rings, the question of finite Cohen-Macaulay type is wide open. Our result explained above shows that  the class of dimer tree algebras is of finite Cohen-Macaulay type $\mathbb{A}$, which suggest an analogy to the case of commutative Gorenstein rings of dimension two. 

In this paper, we construct a class of Jacobian algebras of Cohen-Macaulay type $\mathbb{D}$  and give examples for types $\mathbb{E}_6,\mathbb{E}_7,\mathbb{E}_8$ as well. These results provide further evidence for the analogy to the commutative case. An intriguing aspect of these examples  is that for all our algebras, the stable Cohen-Macaulay category is equivalent to the 2-cluster category of the respective Dynkin type. This seems to suggest a deeper and possibly broader connection between Cohen-Macaulay modules (or singularity categories) of Jacobian algebras on the one hand and 2-cluster categories on the other hand. 
We point out that, if the Cohen-Macaulay category of a Jacobian algebra is equivalent to a higher cluster category, then it should be a 2-cluster category, because it must have the 3-Calabi-Yau property.

Let us now explain the construction of our class of algebras that have Cohen-Macaulay type $\mathbb{D}$.  We obtain these algebras as skew group algebras $AG$ of a dimer tree algebra $A$ with respect to the action of a group $G$ of order two. This construction is inspired by the work of Reiten and Riedtmann in 1985 \cite{RR}, which in particular provided a way to realize path algebras of certain Dynkin quivers of type $\mathbb{D}$ as skew group algebras of path algebras of Dynkin quivers of type $\mathbb{A}$. In our situation, the skew group algebra is again a Jacobian algebra of a quiver with potential, thanks to a result of \cite{AP}. 

Our construction is very general. Indeed, \emph{every} dimer tree algebra gives rise to several of these skew group algebras --- one for each boundary arrow of the quiver of the dimer tree algebra, see Proposition~\ref{prop gluing 1}.
We show that the $G$-action on the dimer tree algebra $A$ induces a $G$-action on the 
Cohen-Macaulay category $\cmp A$ as well as on the associated checkerboard polygon $\cals$ in Proposition~\ref{prop G action on S}.
 We then prove in Theorems~\ref{thm equivalences} and \ref{thm Fequiv} that our equivalences (\ref{equiv}) carry over to equivalences of triangulated 3-Calabi-Yau categories

\begin{equation}\label{equiv2}\scmp AG \cong \arc\, \calp \cong \calc^2_{\mathbb{D}_{(N+1)/2}}\end{equation}
between the stable Cohen-Macaulay category of the skew group algebra $AG$, the combinatorial category of 2-arcs on the punctured $N$-gon and the 2-cluster category of Dynkin type ${\mathbb{D}_{(N+1)/2}}$. 
In particular, the number of indecomposable non-projective Cohen-Macaulay modules is equal to 
$N(N+1)/2$. 

Furthermore, our checkerboard pattern on the $2N$-gon $\cals$ induces a checkerboard pattern on the punctured $N$-gon $\calp$. As in the case of dimer tree algebras, the checkerboard pattern allows us to construct the complete Auslander-Reiten quiver of the Cohen-Macaulay category of the skew group algebra, see Theorem~\ref{thm checkerboard}. Indeed, for every 2-arc $\zg$ in $\arc\, \calp$, the crossing points of $\zg$ with the checkerboard pattern determines a projective presentation of the corresponding Cohen-Macaulay module $M_\zg$. The syzygy functor $\zO$, which is also the inverse shift in the triangulated category $\scmp\, AG$, is given geometrically as a rotation by $2\pi/N$ of $\calp$. Therefore the complete projective resolution of $M_\zg$ is determined by the rotation orbit of the 2-arc $ \zg$. In particular, the projective resolution is periodic of period $N$ or $2N$. 
In Definition~\ref{def 66}, we give a simple method to compute the total weight $N$ directly from the quiver of the (basic version of the) skew group algebra. 

As applications, we obtain in Corollary~\ref{cor 67} that the indecomposable Cohen-Macaulay modules are rigid and extensions between two indecomposables correspond to crossing points between the associated 2-arcs in the punctured disk. We further show in  Proposition~\ref{prop 68}, that the $A$-module $M\oplus \zs M$ is $\tau$-rigid if and only if the induced $AG$-module $M\otimes_A AG$ is $\tau$-rigid,  where $\sigma$ is the nontrivial element of $G$ of order 2.  We conjecture that every indecomposable Cohen-Macaulay module over $A$ and $AG$ is $\tau$-rigid.

Several examples of the construction are given in section \ref{sect examples}. We point out that not every Jacobian algebra of Cohen-Macaulay type $\mathbb{D}$ can be realized as the skew group algebra of a dimer tree algebra; an example is given in section~\ref{ex72}. This should not be surprising, because it is also the case that not every path algebra of type $\mathbb{D}$ can be realized as a skew group algebra of an algebra of type $\mathbb{A}$. 
In section~\ref{sect 74}, we illustrate the relation to dimer models in an example. 
Finally, we provide examples of Jacobian algebras of Cohen-Macaulay types $\mathbb{E}$ in section~\ref{sect 75}.

\section{Recollections}
Let $\kb$ be an algebraically closed field. If $\zL$ is a finite-dimensional $\kb$-algebra, we denote by $\textup{mod}\,\zL$ the category of finitely generated right $\zL$-modules. Let $D$ denote the standard duality $D=\Hom(-,\kb)$. If $Q_\zL$ is the ordinary quiver of the algebra $\zL$, and $i$ is a vertex of $Q_\zL$,  we denote by $P(i),I(i),S(i)$ the corresponding indecomposable projective, injective, simple $\zL$-module, respectively. 

Let $\rad\, \zL$ denote the Jacobson radical of $\zL$. If $M\in \textup{mod}\,\zL$ its radical is defined as $\rad \,M= M (\rad\,\zL)$ and its top
as $\textup{top}\,M= M/\rad\, M$. Thus in particular $\textup{top}\,P(i) =S(i)$.
Given a module $M$, we denote by $\add \,M$ 
the full subcategory of $\textup{mod}\,\zL$ whose objects are direct sums of summands of $M$.
For further information about representation theory and quivers we refer to \cite{ASS,S2}.

\subsection{Basic algebras}
For every finite dimensional algebra $\zL$ over $\kb$ there exists a complete set of primitive orthogonal idempotents $e_1,e_2, \ldots e_n$ such that the $\zL$-module  $\zL$ decomposes into a sum on indecomposable projective modules $\zL=\oplus_{i=1}^n e_i\zL$. The algebra $\zL$ is said to be \emph{basic} if the $e_iA$ are pairwise non-isomorphic. In that case, one can define the quiver of the algebra, whose vertex set is in bijection with the set of idempotents $e_1,\ldots, e_n$. 

If $\zL$ is not basic, we can choose a subset $e_{j_1},\ldots,e_{j_m}$ of idempotents such that the $e_{j_k}\zL$ are pairwise non-isomorphic and such that every $e_i \zL$ is isomorphic to one of the $e_{j_k}\zL$. Let $e=\sum_{k=1}^m e_{j_k}$. Then the algebra $\zL^b=e\zL e$ is a basic algebra that is Morita equivalent to $\zL$, which means that there is an equivalence of categories 
\[\textup{mod}\,\zL\cong \textup{mod}\,\zL^b,\] see for example \cite[Corollary 6.10]{ASS}.  

\subsection{Cohen-Macaulay modules over 2-Calabi-Yau tilted algebras}\label{sect CM} 

Now let $\zL$ be a \emph{2-Calabi-Yau tilted algebra}. Thus $\zL$ is the endomorphism algebra of a cluster-tilting object in a 2-Calabi-Yau category. 
A $\zL$-module $M$ is said to be \emph{projectively Cohen-Macaulay} if $\Ext^i_\zL(M,\zL)=0$ for all $i>0$. In other words, $M$ has no extensions with projective modules.

We denote by $\cmp\,\zL$ the full subcategory of $\textup{mod}\,\zL$ whose objects are the projectively Cohen-Macaulay modules. This  is a Frobenius category. The projective-injective objects in $\cmp\,\zL$ are are precisely the projective $\zL$-modules. The corresponding stable category $\scmp\,\zL$ is triangulated, and its inverse shift is given  by the syzygy operator $\zO$ in $\textup{mod}\,\zL$.

Moreover, by Buchweitz's theorem \cite[Theorem 4.4.1]{Bu}, there exists a triangle equivalence between $\scmp\,\zL$ and the singularity category $\cald^b(\zL)/\cald^b_{perf} (\zL)$ of $\zL$.
Keller and Reiten showed in \cite{KR} that the category $\scmp \,\zL$ is 3-Calabi-Yau.

It was shown in  \cite{GES} that if $M\in \text{mod}\,\zL$ is indecomposable then the following are equivalent. 
 \begin{itemize}
\item [(a)] $M$ is a non-projective syzygy;
\item [(b)] $M \in \textup{ind}\,\scmp\,\zL$; 
\item [(c)] $\zO^2_\zL \tau_\zL M \cong M$.
\end{itemize}

We may therefore use the terminology ``syzygy'' and ``Cohen-Macaulay module'' interchangeably.

Two algebras are said to be \emph{derived equivalent} if there exists a triangle equivalence between their bounded derived categories.
Two algebras are said to be \emph{singular equivalent} if there exists a triangle equivalence between their singularity categories.

\subsection{Quivers with potentials}\label{sect QP}
A quiver $Q=(Q_0,Q_1,s,t)$ consists of a finite set of vertices $Q_0$, a finite set of arrows $Q_1$ and two maps $s,t\colon Q_1\to Q_0$, where $s$ is the source and $t$ is the target of the arrow. Thus if $\za\in Q_1$ then $\za\colon s(\za)\to t(\za)$. 

 A \emph{potential} $W$ on a quiver $Q$ is a  linear combination of non-constant  cyclic paths. For every arrow $\za\in Q_1$, the cyclic derivative $\partial _\za$ is defined on a cyclic path $\za_1\za_2\dots\za_t$ as 
\[ \partial_\za(\za_1\za_2\dots\za_t)=\sum_{p\colon \za_p=\za} \za_{p+1}\dots\za_t\za_1\dots\za_{p-1}\]
and extended linearly to the potential $W$.

The \emph{Jacobian algebra} $\textup{Jac}(Q,W)$ of the quiver with potential is the quotient of the (completed) path algebra $\kb Q$ by  (the closure of) the 2-sided ideal generated by all partial derivatives $\partial_\za W$, with $\za\in Q_1$. 
Two parallel paths in the quiver are called \emph{equivalent} if they are equal in $\textup{Jac}(Q,W)$.

If $Q$ has no oriented 2-cycles then $\textup{Jac}(Q,W)$ is 2-Calabi-Yau tilted by \cite{Amiot}.


\subsection{Translation quivers and mesh categories} We review here the notions of translation quiver and mesh category from \cite{Ri, H}. These notions are often used in order to define a category from combinatorial data. 

A {\em  translation
 quiver} $(\zG,\tau)$ is a quiver $\zG=(\zG_0,\zG_1)$ without loops
 together with an injective map $\tau\colon \zG_0'\to\zG_0$  (the {\em translation}) from a subset $\zG_0'$ of $\zG_0$ to $\zG_0$ such that, for all vertices $x\in\zG_0'$, $y\in \zG_0$, 
the number of arrows from $y \to x$ is equal to the number of arrows
 from $\tau x\to y$. 
Given a  translation quiver $(\zG,\tau)$, a \emph{polarization of} $\zG$ is
 an injective map $\sigma:\zG_1'\to\zG_1$, where $\zG_1'$ is the set of all arrows $\za\colon y\to x$ 
 with $x \in \zG_0'$, such that 
$\sigma(\za)\colon \tau x\to y$  for every arrow $\za\colon y\to x\in \zG_1$.
From now on we assume that $\zG$ has no multiple arrows. In that case, there is a unique polarization of $\zG$.

The {\em path category } of a translation quiver $(\zG,\tau)$ is the category whose  objects are
the vertices $\zG_0$ of $\zG$, and, given $x,y\in\zG_0$, the $\kb$-vector space of
morphisms from $x$ to $y$ is given by the $\kb$-vector space with basis
the set of all paths from $x $ to $y$. The composition of morphisms is
induced from the usual composition of  paths.
The {\em mesh ideal} in the path category of $\zG$ is the ideal
generated by the {\em mesh relations}
\begin{equation}\nonumber
m_x =\sum_{\za:y\to x} \sigma(\za) \za 
\end{equation}  
for all $x \in \zG_0'$.

The {\em  mesh category } of the translation quiver $(\zG,\tau)$ is the
quotient of its path 
category by the mesh ideal.

\subsection{2-cluster categories}\label{sect 2-cluster} 
We recall the definition of 2-cluster categories and their geometric models in Dynkin types $\mathbb{A}$ and $\mathbb{D}$ given by Baur and Marsh \cite{BM1,BM2}.  

We use the notation of  Thomas \cite{T}. Thus the 2-cluster category $\calc^2$ is defined as \[\calc^2_H=\cald^b(\textup{mod}\,H)/\tau^{-1}[2],\] where $H$ is a hereditary algebra,  $\tau$ is the Auslander-Reiten translation in the bounded derived category $\cald^b(\textup{mod}\,H)$ and $[2]$ denotes the second shift. 

\begin{remark}
  We point out that this terminology is different from Iyama's, who would rather call this category a 3-cluster category in order to reflect the fact that it is 3-Calabi-Yau. Here we follow Thomas and call it a 2-cluster category, because the definition involves the second shift. 
\end{remark}

Baur and Marsh gave geometric models for the $\calc^2_H$ in the case where the algebra $H$ is of Dynkin type $\mathbb{A}$ \cite{BM1} or $\mathbb{D}$ \cite{BM2}. In type $\mathbb{A}$ the 2-cluster category is equivalent to the category $\diags$ of 2-diagonals in a regular polygon $\cals$, and in type $\mathbb{D}$ the 2-cluster category is equivalent to the category $\arcp$ of 2-arcs in a punctured polygon $\calp$.

We point out that the results by Baur-Marsh are more general than what we present here. Indeed they obtain geometric models for $m$-cluster categories for all $m\ge 2$. The case $m=1$ had been done in \cite{CCS,S}. In this paper, we shall only need the case $m=2$, and we adapt the presentation and notation accordingly.

\subsubsection{Dynkin type $\mathbb{A}$}\label{sect A}
Let $H$ be a path algebra of type $\mathbb{A}_{N-2}$ and $\calc^2_{\mathbb{A}_{N-2}}$ the 2-cluster category.  
Let $\cals=\cals_{2N}$ be a regular polygon with $2N$ vertices and label the vertices $1,2,\ldots, 2N$ in clockwise order around the boundary. For $i,j\in\{1,2,\ldots, 2N\}$, $i\ne j$,  we denote by $(i,j)$ the straight line connecting $i$ and $j$. We say $(i,j)$ is a \emph{boundary segment}  of $\cals $ if $i$ and $j$ are neighbors on the boundary of $\cals$, and otherwise, we call $(i,j)$ a \emph{diagonal} in $\cals$. 

A diagonal $(i,j)$ is called a \emph{2-diagonal} if cutting $\cals$ along $(i,j)$ will produce two polygons that have an even number of vertices. In other words,  $(i,j)$ is a 2-diagonal if and only if $|j-i| $ is odd. 

The 2-diagonals are the indecomposable objects in the category $\diags$. The irreducible morphisms are given by the 2-pivots defined below.

A 2-diagonal $(i,j)$ admits the following \emph{2-pivots},
\[ 
 \xymatrix@R5pt{ &(i,j+2)\\ (i,j)\ar[ru]\ar[rd] \\&(i+2,j)} 
\]
unless the target of the pivot is not a 2-diagonal but a boundary segment. The addition in the coordinates is modulo $2N$.
Thus a 2-pivot fixes one endpoint of the diagonal and moves the other endpoint two positions further along the boundary in clockwise direction.

The Auslander-Reiten translation $\tau$ of a 2-diagonal $(i,j)$ is given by $\tau(i,j)= (i-2,j-2)$.  Hence $\tau$ is given geometrically as a rotation by $2\pi/N$. 

This data defines a translation quiver $\zG_{\diags}$ with  2-diagonals as vertices, 2-pivots as arrows and translation $\tau$, see the top  picture in Figure~\ref{fig AR 2cluster BM} for the case $N=7$. The category $\diags $ is defined as the mesh category of $\zG$. 

\begin{thm}\label{thm BM1}
 \cite{BM1}
 There is an equivalence of categories \[\BM_\mathbb{A} \colon\diag (\cals_{2N})\to\calc^2_{\mathbb{A}_{N-2}}.\]
   This equivalence induces an isomorphism from the quiver $\zG_{\diags}$ to the  Auslander-Reiten quiver of the 2-cluster category.
\end{thm}

 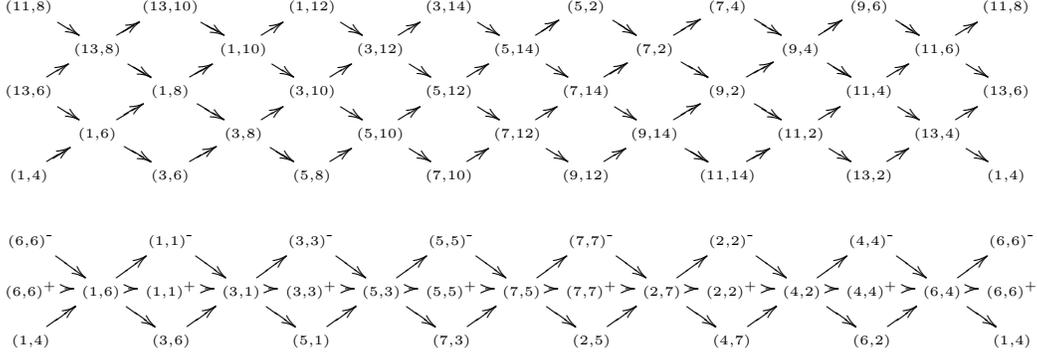
\begin{figure}
\begin{center}
\[
\begin{array}{l}
\xymatrix@R5pt@C2pt{
\scriptscriptstyle (11,8)\ar[dr]&&\scriptscriptstyle (13,10)\ar[dr]&&\scriptscriptstyle (1,12)\ar[dr]&&\scriptscriptstyle(3,14)\ar[dr]&&\scriptscriptstyle (5,2)\ar[dr]&&\scriptscriptstyle (7,4)\ar[dr]&&\scriptscriptstyle (9,6)\ar[dr]&&\scriptscriptstyle (11,8)
\\
&\scriptscriptstyle (13,8)\ar[dr]\ar[ur]&&\scriptscriptstyle (1,10)\ar[dr]\ar[ur]&&\scriptscriptstyle (3,12)\ar[dr]\ar[ur]&&\scriptscriptstyle (5,14)\ar[dr]\ar[ur]&&\scriptscriptstyle (7,2)\ar[dr]\ar[ur]&&\scriptscriptstyle (9,4)\ar[dr]\ar[ur]&&\scriptscriptstyle (11,6)\ar[ur]\ar[dr]
\\
\scriptscriptstyle (13,6)\ar[dr]\ar[ur]&&\scriptscriptstyle (1,8)\ar[dr]\ar[ur]&&\scriptscriptstyle (3,10)\ar[dr]\ar[ur]&&\scriptscriptstyle (5,12)\ar[dr]\ar[ur]&&\scriptscriptstyle (7,14)\ar[dr]\ar[ur]&&\scriptscriptstyle (9,2)\ar[dr]\ar[ur]&&\scriptscriptstyle (11,4)\ar[dr]\ar[ur]&&\scriptscriptstyle (13,6)
\\
&\scriptscriptstyle (1,6)\ar[dr]\ar[ur]&&\scriptscriptstyle (3,8)\ar[dr]\ar[ur]&&\scriptscriptstyle (5,10)\ar[dr]\ar[ur]&&\scriptscriptstyle (7,12)\ar[dr]\ar[ur]&&\scriptscriptstyle (9,14)\ar[dr]\ar[ur]&&\scriptscriptstyle (11,2)\ar[dr]\ar[ur]&&\scriptscriptstyle (13,4)\ar[dr]\ar[ur]&&
\\
\scriptscriptstyle (1,4)\ar[ur]&&\scriptscriptstyle (3,6)\ar[ur]&&\scriptscriptstyle (5,8)\ar[ur]&&\scriptscriptstyle (7,10)\ar[ur]&&\scriptscriptstyle (9,12)\ar[ur]&&\scriptscriptstyle (11,14)\ar[ur]&&\scriptscriptstyle (13,2)\ar[ur]&&\scriptscriptstyle (1,4)
}
\\ \\
\xymatrix@R7pt@C3.5pt{
\scriptscriptstyle(6,6)^{\mbox{\tiny -}}\ar[dr]&&\scriptscriptstyle(1,1)^{\mbox{\tiny -}}\ar[dr]&&\scriptscriptstyle(3,3)^{\mbox{\tiny -}}\ar[dr]&&\scriptscriptstyle(5,5)^{\mbox{\tiny -}}\ar[dr]&&\scriptscriptstyle(7,7)^{\mbox{\tiny -}}\ar[dr]&&\scriptscriptstyle(2,2)^{\mbox{\tiny -}}\ar[dr]&&\scriptscriptstyle(4,4)^{\mbox{\tiny -}}\ar[dr]&&\scriptscriptstyle(6,6)^{\mbox{\tiny -}}
\\
\scriptscriptstyle(6,6)^{\scalebox{0.8}{\tiny $+$}}\ar[r]&\scriptscriptstyle(1,6)\ar[dr]\ar[ur]\ar[r]&\scriptscriptstyle(1,1)^{\scalebox{0.8}{\tiny +}}\ar[r]&\scriptscriptstyle(3,1)\ar[dr]\ar[ur]\ar[r]&\scriptscriptstyle(3,3)^{\scalebox{0.8}{\tiny +}}\ar[r]&\scriptscriptstyle(5,3)\ar[dr]\ar[ur]\ar[r]&\scriptscriptstyle(5,5)^{\scalebox{0.8}{\tiny +}}\ar[r]&\scriptscriptstyle(7,5)\ar[dr]\ar[ur]\ar[r]&\scriptscriptstyle(7,7)^{\scalebox{0.8}{\tiny +}}\ar[r]&\scriptscriptstyle(2,7)\ar[dr]\ar[ur]\ar[r]&\scriptscriptstyle(2,2)^{\scalebox{0.8}{\tiny +}}\ar[r]&\scriptscriptstyle(4,2)\ar[dr]\ar[ur]\ar[r]&\scriptscriptstyle(4,4)^{\scalebox{0.8}{\tiny +}}\ar[r]&\scriptscriptstyle(6,4)\ar[ur]\ar[r]\ar[dr]&\scriptscriptstyle(6,6)^{\scalebox{0.8}{\tiny +}}
\\
\scriptscriptstyle(1,4) \ar[ur]&&\scriptscriptstyle(3,6) \ar[ur]&&\scriptscriptstyle(5,1) \ar[ur]&&\scriptscriptstyle(7,3) \ar[ur]&&\scriptscriptstyle(2,5)\ar[ur]&&\scriptscriptstyle(4,7) \ar[ur]&&\scriptscriptstyle(6,2) \ar[ur]&&\scriptscriptstyle(1,4)
\\
}
\end{array}\]

\caption{Both pictures illustrate the geometric realizations of  2-cluster categories. The top picture shows the AR-quiver of the 2-cluster category of  $\mathbb{A}_5$, where the labels are the 2-diagonals in the 14-gon in accordance with Theorem~\ref{thm BM1}. The bottom picture shows the AR-quiver in type $\mathbb{D}_4$, where the vertices are labeled by the 2-arcs in the punctured 7-gon as in Theorem~\ref{thm BM2}. Vertices with the same label have to be identified.}
\label{fig AR 2cluster BM}
\end{center}
\end{figure}

\subsubsection{Dynkin type $\mathbb{D}$}\label{sect BMD}
Now let $N\ge 7$ be an odd integer and $H$ be a path algebra of type $\mathbb{D}_{(N+1)/2}$. Let $\calc^2_{\mathbb{D}_{(N+1)/2}}$ be the 2-cluster category.  
Let $\calp=\calp_N$ denote a regular polygon with $N$ vertices and one puncture in the center. Label the vertices $1,2,\ldots, N$ in clockwise order around the boundary. For $i,j\in\{1,2,\ldots,N\}$, $i\ne j$, we denote by $(i,j)$ (the homotopy class of) a curve without selfcrossings from $i$ to $j$ that goes clockwise around the puncture. In other words, the puncture lies on the right  of the curve when traveling along from $i$ to $j$. 
Note that $(i,j)\ne (j,i)$. 
Furthermore, we denote by $(i,i)$ (the homotopy class of) the curve that starts at $i$ goes clockwise around the puncture once and ends at $i$. We call the curves $(i,i)$  \emph{loops} and in our figures we often draw them as a straight line from $i$ to the puncture. See Figure~\ref{fig 2arcs} for an illustration.

We say $(i,j)$, with $i\ne j$, is a \emph{boundary segment}  of $\calp $ if  $i$ and $j$ are neighbors on the boundary of $\cals$, and otherwise, we call $(i,j)$ an \emph{arc} in $\cals$. 
Furthermore, for each loop $(i,i)$, we define two arcs which we denote by $(i,i)^+$ and $(i,i)^-$. 

Every arc cuts the punctured polygon $\calp$ into two pieces, one of which is a polygon and the other a punctured polygon. 
An arc $(i,j)$ is called a \emph{2-arc} if the polygon piece obtained by cutting along $(i,j)$ has an even number of vertices.
\footnote{Note that since $N$ is odd, the punctured polygon piece will then have an odd number of vertices.}
Here we agree that the arcs $(i,i)^\pm$ are 2-arcs because the resulting polygon has $N+1$ vertices since the vertex $i$ splits into two vertices. 
Examples of 2-arcs are given in Figure~\ref{fig 2arcs}.
\begin{figure}
\begin{center}
\scalebox{0.8}{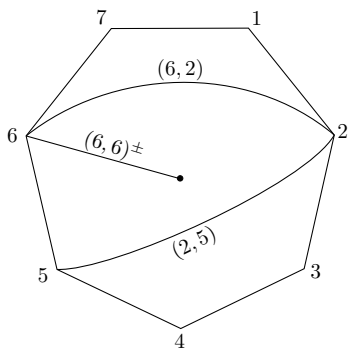}
\caption{Examples of 2-arcs in a punctured polygon $\calp$}
\label{fig 2arcs}
\end{center}
\end{figure}

For $i\ne j$, we can reformulate the condition by saying that an arc $(i,j)$ is a 2-arc if and only if the residue class of $j-i \pmod N$ is odd. For example the arc $(6,2)$ in Figure~\ref{fig 2arcs} is a 2-arc since $2-6=-4\equiv 3 \pmod 7$ is odd.
On the other hand, the curve $(2,6$) is not a 2-arc.

The 2-arcs are the indecomposable objects in the category $\arcp$. The irreducible morphisms are given by the 2-pivots defined below. 

A 2-arc $(i,j)$ admits the following \emph{2-pivots}.
\[\small
\begin{array}{ccccccc}
  \xymatrix@R5pt@C5pt{ &(i,j+2)\\ (i,j)\ar[ru]\ar[rd] \\&(i+2,j)} 
 & &
  \xymatrix@R5pt@C5pt{ &(i,i+5)\\ (i,i+3)\ar[ru]} 
  &&
    \xymatrix@R5pt@C5pt{ &(i,i)^+\\ (i,i-2)\ar[ru]\ar[rd]\ar[r] &(i,i)^-\\&(i+2,i-2)}
&&
\xymatrix@R5pt@C7pt{(i,i)^\pm\ar[r]&(i+2,i)}\\ \\
  \textup{if $j\ne i-2,i, i+3$ } & &\textup{if $j=i+3$}
  & &\textup{if $j=i-2$}& &\textup{if $j=i$}\\   \\
\end{array}
\]
The addition in the coordinates is  modulo $N$.
Thus a 2-pivot fixes one endpoint of the 2-arc and moves the other endpoint two positions further along the boundary in clockwise direction.

The Auslander-Reiten translation $\tau$ of a 2-arc $(i,j)$ is given by
$\tau(i,j)=(i-2,j-2)$ if $i\ne j$, and $\tau(i,i)^+=(i-2,i-2)^+$, $\tau(i,i)^-=(i-2,i-2)^-$. Hence $\tau$ is given geometrically as a rotation by $4\pi/N$. In particular, $\tau $ has period $N$, because $N$ is odd.

Similarly to type $\mathbb{A}$, this data defines a translation quiver $\zG_{\arcp}$ and the category $\arcp$ is defined as its mesh category,  see the bottom picture in Figure~\ref{fig AR 2cluster BM} for the case $N=7$.
\begin{thm}
 \label{thm BM2} \cite{BM2} There is an equivalence of categories 
 \[\BM_\mathbb{D}\colon \arc (\calp_N) \to \calc^2_{\mathbb{D}_{(N+1)/2}}.\] 
 This equivalence induces an isomorphism from the quiver $\zG_{\arcp}$ to the  Auslander-Reiten quiver of the 2-cluster category.
\end{thm}

\subsection{Skew group algebras}
In this section, we recall results on skew group algebras from \cite{RR}. Here we shall focus on the case where the group has order two. For further results on the order two case see \cite{AP}.
All modules will be finitely generated.

Let $A$ be a finite dimensional basic algebra over an algebraically closed field $\kb$ of characteristic different from 2. Let $Q$ be the ordinary quiver of $A$ and $I$ an admissible ideal such that $A\cong \kb Q/I$.  

Let $G=\{1,\zs\}$ be the group of order 2 and 
assume that $A$ admits a $G$-action  
such that  $\zs$  acts as an automorphism of $A$.
\begin{definition} We say the $G$-action on $A$ is \emph{admissible} if
\begin{enumerate}
\item  $\zs$ maps vertices to vertices and arrows to arrows, and 
\item $\zs$ fixes at least one vertex.
\end{enumerate}
\end{definition}
In this paper we will be only interested in admissible $G$-actions. For group actions that do not require this condition, we refer to \cite{RR}.

Let $AG$ be the skew group algebra, thus 
\[AG=A\otimes_\kb \kb G\] as $\kb$-vector spaces, and the multiplication is given on simple tensors by the formula
\[(a\otimes g) (a'\otimes g') 
= ag(a')\otimes gg'\]
for all $a,a'\in A, \ g,g'\in G$,
and extended to all of $AG$ by distributivity. For example
\[
\begin{array}{rcl}
(a\otimes(g_1+g_2))(a'\otimes g')
&=& (a\otimes g_1)(a'\otimes g') +(a\otimes g_2)(a'\otimes g')\\
&=& ag_1(a')\otimes g_1g' + ag_2(a')\otimes g_2g' 
\end{array}
\]

By definition, the dimension of $AG$ is twice the dimension of $A$.

\begin{remark}
 In general $AG$ is  not a basic algebra. 
\end{remark}

\subsubsection{$G$-action as automorphisms of $\textup{mod}\,A$ and of $\scat $}\label{Gmod}
 If $M$ is an $A$-module, define the module $\zs M$ to be the same vector space as $M$ but with twisted $\zs$-action given by 
$m\cdot_\zs a = m\cdot(\zs a)$, for $m\in M, a\in A$, where the right hand side is the action of $\zs a$ on $M$. If $f\colon M\to N$ is a morphism in $\textup{mod}\,A$, define $\zs f\colon \zs M\to \zs N$ by $(\zs f)(m)=f(m)$. The following computation shows that $\zs f$ is a morphism in $\textup{mod}\,A$.
\[(\zs f)(m\cdot_\zs a)= (\zs f)(m\cdot (\zs a)) = f(m\cdot (\zs a)) =
f(m) \cdot (\zs a) =(\zs f)(m)\cdot_\zs a \]
This defines a functor $\zs\colon \textup{mod}\,A\to \textup{mod}\,A$.

A \emph{syzygy} is a submodule of a projective module. If $M\in \textup{mod}\,A$, its syzygy $\zO M$ is defined to be the kernel of a projective cover of $M$. We thus have a short exact sequence
\begin{equation}
 \label{eq ses}
\xymatrix{0\ar[r]&\zO M\ar[r] &P\ar[r]^\pi&M\ar[r]&0}
\end{equation}
where $\pi$ is a projective cover. Every  non-projective syzygy  is of the form $\zO M$ for some $M\in \textup{mod}\,A$. 
 
\begin{lemma}\label{lem Gmod}

(a) The functor $\zs$ is an automorphism of $\textup{mod}\,A$ of order two. In particular $\zs$ is an exact functor.

(b) $\zs P(i)=P(\zs i)$, $\zs I(i)=I(\zs i)$, and $\zs S(i)=S(\zs i)$. 

(c) The restriction of $\zs$ to the  syzygy category and to the stable syzygy category are automorphisms.

(d) $\zs\, \rad\, P(i) =\rad\,P(\zs i)$.

\end{lemma}

\begin{proof}
Clearly $\zs^2$ is the identity. So $\zs$ is an automorphism of order 2.  Every equivalence between abelian categories is exact. This proves (a). 
 
(b) For  a path $w$ in $A$, we have $e_i \cdot_\zs w = e_i  \zs(w) $. The last expression is nonzero if and only if $\zs(w)$ starts at vertex $i$ which holds if and only if $w$ starts at $\zs(i)$. Thus $\zs P(i)=P(\zs i)$.  The equation for the injective modules is the dual statement. 
To show the equation for the simple modules let, $m\in \zs S(i)$ and consider the action of a constant path $e_j$ on $m$. We have $m\cdot_\zs e_j = m\cdot e_{\zs j}$ which is nonzero if and only if $\zs j =i$ which holds if and only if $j=\zs i$. Thus $\zs S(i)=S(\zs i)$. 

(c) From (b) we know that $\zs$ maps projective syzygies to projective syzygies. Now let $M$ be a non-projective syzygy and let $P$ be a projective module containing $M$ as a submodule. By part~(a), $\zs M$ is a submodule of $\zs P$, and by part (b), $\zs P$ is projective. Thus $\zs M$ is a syzygy. Moreover $\zs M$ is non-projective, because otherwise $M=\zs\zs M$ would be projective as well.

(d)  Because of (a) and (b), the short exact sequence 
\[\xymatrix{0\ar[r]&\rad \, P(i) \ar[r]& P(i)\ar[r]& S(i)\ar[r]&0}\]
 is mapped by $\zs $ to the short exact sequence
\[\xymatrix{0\ar[r]&\zs\,\rad \, P(i) \ar[r]& P(\zs i)\ar[r]& S(\zs i)\ar[r]&0}.\]
Thus  $\zs\, \rad\, P(i) =\rad\,P(\zs i)$.
\end{proof}

\subsubsection{Induction and restriction functors}

Let $F\colon \textup{mod}\,A\to \textup{mod}\,AG$ denote the induction functor, thus $F=-\otimes_A AG$, and let $H\colon\textup{mod}\,AG\to \textup{mod}\,A$ denote the restriction functor. Thus if $N$ is an $AG$-module then $H(N)=N$ as $\kb$-vector spaces, but we restrict the scalars from $AG$ to $A$. If $f$ is a morphism in $\textup{mod}\,AG$ then $H(f)=f$. 
Both $F$ and $H$ are exact functors and map projective modules to projective modules.

The following is shown in \cite[Proposition 1.8]{RR} in a more general setting. 

 \begin{prop}
  \label{prop RR} Let $X,Y$ be indecomposable $A$-modules.
  
  (a) $HF(X)\cong X\oplus \zs X$.
  
  (b) $F(X)\cong F(Y) $ if and only if $X\cong Y$ or $X\cong \zs Y$.
  
  (c) If $X\cong \zs X$ then $F(X) $ has exactly two indecomposable summands. 
  
  (d) If $X\not\cong \zs X$ then $F(X)$ is indecomposable. 
\end{prop}
We also need the next result.
\begin{prop}
 \cite[Theorem 3.8]{RR} \label{prop RR2}
(a)  The functor $F$  maps an almost split sequence in $\textup{mod}\,A$  to a direct sum of almost split sequences in $\textup{mod}\,AG$.

(b)  The functor $H$  maps an almost split sequence in $\textup{mod}\,AG$  to a direct sum of almost split sequences in $\textup{mod}\,A$.
\end{prop}

The action of the functors $F$ and $H$ on syzygies is described in the following result.
\begin{prop}
 \label{prop Fsyz}
 {\rm (a)} The functors $F,H$ map syzygies to syzygies.
%

{\rm (b)} 
Let $\zO_A$ and $ \zO_{AG}$ denote the syzygy functors over $A$ and $AG$ respectively. Then
\[ \zO_{AG} F = F\zO_A \quad\textup{and}\quad \zO_AH=H\zO_{AG}.\]
\end{prop}
\begin{proof}
 (a)  Let $M$ be a syzygy in $\textup{mod}\,A$. Then there exists a projective $A$-module $P$ that contains $M$ as a submodule. Since the functor $F$ is exact it follows that $F(M)$ is a submodule of $F(P)$. Moreover, since $F(P)$ is projective, $F(M)$ is a syzygy.
The proof for $H$ is analogous.

(b) Let $P$ be a projective module, then we have the first identity $\zO_{AG} F (P)= 0=F\zO_A (P)$.  Now, let $M$ be an indecomposable, non-projective syzygy in $\textup{mod}\,A$. Applying the exact functor $F$ to the short exact sequence (\ref{eq ses}) yields the short exact sequence
\[\xymatrix{0\ar[r]&F(\zO_A M)\ar[r] &F(P)\ar[r]^{F(\pi)}&F(M)\ar[r]&0}
.\]
  If $F(\pi)$ is not a projective cover then there exists  a direct summand $P'$ of $F(P)$ such that the short exact sequence $0\to P' \to P' \to 0\to 0$ is a summand of the short exact sequence above.  Then Proposition~\ref{prop RR} implies that the sequence~(\ref{eq ses}) also decomposes nontrivially, which is a contradiction to $\pi$ being a projective cover.
 It follows that  $F(\pi)$ is a projective cover and hence $F(\zO_A M)=\zO_{AG} F(M)$. This proves the first identity. 
 The identity that involves $H$ is proved analogously.
\end{proof}

\subsection{The basic algebra}
In \cite{AP}, Amiot and Plamondon specify a complete set of primitive orthogonal idempotents in $AG$ and explicitly construct an idempotent $\overline{e}$ such that $B=\overline{e}AG\overline{e}$ is a basic algebra. Furthermore, they provide a description of the quiver $Q_G$ of $B$ and an action of $G$ on $B$. 

If $A={\textup {Jac}}{(Q,W)}$ is the Jacobian algebra of a quiver $Q$ with potential $W$, they define a potential $W_G$ on $Q_G$ and prove the following. 
\begin{prop} \cite[Corollary 2.7]{AP}
 The algebra $B$ is the Jacobian algebra of the quiver $Q_G$ with potential $W_G$.
\end{prop}


\section{$G$-action on 2-cluster categories of type $\mathbb{A}$} \label{sect 31} 

Let $G=\{1,\zs\}$ be the group with two elements.
In this subsection, we describe a  $G$-action on the 2-cluster category $\calc^2=\cald^b(\textup{mod}\,\zL)/\tau^{-1}[2]$, with $\zL$ a path algebra of type $\mathbb{A}_r$, with $r$ odd. 
The $G$-action is inspired by an example in \cite{RR}. We will show later that this action is compatible with our $G$-action on dimer tree algebras.

In order to describe the $G$-action on $\calc^2$, we choose $\zL$ to  the path algebra of the following quiver
\[\small\xymatrix@R5pt{ &2\ar[r]^{\za_2} &3\ar[r]^{\za_3}&\cdots\ar[r]^{\za_{n-1}}&n\\
1\ar[ru]^{\za_1}\ar[rd]_{\za_1'}\\
 &2'\ar[r]_{\za_2'} &3'\ar[r]_{\za'_3}&\cdots\ar[r]_{\za_{n-1}'}&n'\\
}
\]
with $G$-action given by $\zs(1)=1$, $\zs(i)=i'$, for $i>1$ and $\zs(\za_i)=\za_i'$. This is Example 2.1 in \cite{RR}. The corresponding skew group algebra $\zL G$  is Morita equivalent to the path algebra of the following quiver. 
\[\small\xymatrix@R10pt
{ 1^+\ar[rd]^{\za_1^+}\\
&2\ar[r]^{\za_2}&3\ar[r]^{\za_3}&\cdots \ar[r]^{\za_{n-1}}&n\\
1^-\ar[ru]_{\za_1^-}\\}
\]
Note that $\zL$ is of Dynkin type $\mathbb{A}_{2n-1}$ and $\zL G$ is of Dynkin type $\mathbb{D}_{n+1}$.

The Auslander-Reiten quivers of $\textup{mod}\,\zL$ and $\textup{mod}\,\zL G$ are illustrated in Figure~\ref{fig AR} in the case $n=4.$

\begin{figure}
 \[\tiny\xymatrix@!@R-30pt@!@C-30pt{
P(4)\ar[rd] &&\cdot\ar[rd]&&\cdot\ar[rd]&&I(4')\ar[rd]
\\&P(3)\ar[ru]\ar[rd]&&\cdot\ar[ru]\ar[rd]&&\cdot\ar[rd]\ar[ru]&&I(3')\ar[rd]
\\&&P(2)\ar[ru]\ar[rd]&&\cdot\ar[ru]\ar[rd]&&\cdot\ar[rd]\ar[ru]&&I(2')\ar[rd]
\\&&&P(1)\ar[ru]\ar[rd]&&\cdot\ar[ru]\ar[rd]
&&\cdot\ar[rd]\ar[ru]&&I(1)
\\&&P(2')\ar[rd]\ar[ru]&&\cdot\ar[rd]\ar[ru]&&\cdot\ar[rd]\ar[ru]&&I(3)\ar[ru]
\\&P(3')\ar[ru]\ar[rd]&&\cdot\ar[rd]\ar[ru]&&\cdot\ar[rd]\ar[ru]&&I(4)\ar[ru]
\\
P(4')\ar[ru]&&\cdot\ar[ru]&&\cdot\ar[ru]&&I(1)
\ar[ru]}
\hfill \tiny
\xymatrix@!@R-30pt@!@C-25pt{
P(4)\ar[rd] &&\cdot\ar[rd]&&\cdot\ar[rd]&&I(4)\ar[rd]
\\&P(3)\ar[ru]\ar[rd]&&\cdot\ar[ru]\ar[rd]&&\cdot\ar[rd]\ar[ru]&&I(3)\ar[rd]
\\ &\!\!P(1^+)\ar[r]&P(2)\ar[ru]\ar[rd]&\cdot\ar@{<-}[l]\ar[r]&\cdot\ar[ru]\ar[rd]\ar[r]&\cdot\ar[r]&\cdot \ar[rd]\ar[ru]&\!\!I(1^-)\ar@{<-}[l]\ar[r]&I(2)
\\&P(1^-)\ar[ru]&&\cdot\ar[ru]
&&\cdot\ar[ru]&&I(1^+)\ar[ru]
}
\]
\caption{The Auslander-Reiten quivers of $\textup{mod}\,\zL$ and $\textup{mod}\,\zL G$.}\label{fig AR}
\end{figure}

The action of $\zs$ on $\textup{mod}\,\Lambda$ is given by the reflection along the horizontal line through the $\tau$-orbit of $P(1)$. This action induces a $G$-action on the derived category and on the 2-cluster category that are also given by the reflection at the horizontal line through the central $\tau$-orbit. Note that the Auslander-Reiten quiver of the 2-cluster category has the shape of a cylinder and therefore has $2n-1 $ different $\tau$-orbits, unlike the 1-cluster category, which has the shape of  a Moebius strip and therefore has only $n$ different $\tau$-orbits.

The two 2-cluster categories $\cals^2_{\mathbb{A}_{2n-1}}$ and $\cals^2_{\mathbb{D}_{n+1}}$ are illustrated in Figure~\ref{fig AR 2cluster BM} in the case $n=3$. In type $\mathbb{A}$, the vertices of the AR quiver in that figure are labeled by the corresponding 2-diagonals of the 14-gon as explained in section~\ref{sect 2-cluster}. 
Each 2-diagonal is denoted by the pair $(i,j)$ of its endpoints on the boundary of the 14-gon.  In this example, we see that
the action of $\zs$ maps the 2-diagonal $(i,j)$ to the 2-diagonal $(i+7, j+7)$. For example the first element in the bottom $\tau$-orbit is $(1,4)$ and it is mapped to the first element in the top $\tau$-orbit, thus $\zs((1,4))=(1+7,4+7)=(8,11)$. The elements in the center $\tau$-orbit are fixed by $\zs$. For example $\zs((1,8))=(1+7,8+7)=(8,1)$. 

For general $n$, and  setting $N=2n+1$, the 2-cluster category 
$\calc^2_{\mathbb{A}_{N-2}}$ corresponds to the category of 2-diagonals in a regular $2N$-gon, and the action of $\zs$ sends a 2-diagonal $(i,j)$ to the 2-diagonal $(i+N,j+N)$.  In other words $\zs$ acts on the polygon by a rotation by angle $\pi$. 

Denote by $F_2\colon \calc^2_{\mathbb{A}_{N-2}}\to  \calc^2_{\mathbb{D}_{(N+1)/2}}$  the functor induced by the induction functor $F$ and denote by $H_2\colon  \calc^2_{\mathbb{D}_{(N+1)/2}}\to \calc^2_{\mathbb{A}_{N-2}}$  the one induced by the restriction functor $H$. The following results follow from Propositions~\ref{prop RR} and \ref{prop RR2}.

 \begin{prop}
  \label{prop 2-clusters} Let $N\ge 7$ be an odd integer and $X,Y\in\calc^2_{\mathbb{A}_{N-2}}$ be indecomposable objects. Then
  
  (a) $H_2F_2(X)\cong X\oplus \zs X$.
  
  (b) $F_2(X)\cong F_2(Y) $ if and only if $X\cong Y$ or $X\cong \zs Y$.
  
  (c) If $X\cong \zs X$ then $F_2(X) $ has exactly two indecomposable summands. 
  
  (d) If $X\not\cong \zs X$ then $F_2(X)$ is indecomposable. 
\end{prop}
\begin{prop}
 \label{prop 2-clusters2} 
(a)  The functor $F_2$  maps an Auslander-Reiten triangle in $\calc^2_{\mathbb{A}_{N-2}}$  to a direct sum of Auslander-Reiten triangles in $\calc^2_{\mathbb{D}_{(N+1)/2}}$

(b)  The functor $H_2$  maps  Auslander-Reiten triangle in 
$\calc^2_{\mathbb{D}_{(N+1)/2}}$
 to a direct sum of Auslander-Reiten triangles in $\calc^2_{\mathbb{A}_{N-2}}$.
\end{prop}


\section{$G$-actions on dimer tree algebras}

The class of dimer tree algebras was introduced in \cite{SS3}. The terminology stems from the fact that the dual graph of the quiver of the algebra is a tree, and the algebra can be extended to a dimer model on a disk.

\subsection{Dimer tree algebras}
We recall the definition of dimer tree algebras following \cite{SS3,SS4}.
 A \emph{chordless cycle} in a quiver $Q$ is a cyclic path $C=x_0\to x_1 \to\dots\to x_t\to x_0$ such that $x_i\ne x_j$ if $i\ne j$ and the full subquiver on vertices $x_0,x_1,\dots, x_t$ is equal to $C$. 
The arrows that lie in exactly one chordless cycle will be called {\em boundary arrows} and those that lie in two or more chordless cycles {\em interior arrows} of $Q$. 

\begin{definition}
  The \emph{dual graph} $G$ of $Q$ is defined as follows. The set of vertices $G_0$ is the union of 
the set of chordless cycles of $Q$ and the set of boundary arrows of $Q$. 
The set of edges $G_1$ is the union of two sets
called the set of \emph{trunk edges} and the set of \emph{leaf branches}.  A trunk edge $\xymatrix{C\ar@{-}[r]^\za&C'}$ is drawn between any pair of chordless cycles  $(C,C')$ that share an arrow $\za$. A leaf branch $\xymatrix{C\ar@{-}[r]^\za &\za}$ is drawn between any pair $(C,\za)$ where $C$ is a chordless cycle and $\za$ is a boundary arrow such that $\za $ is contained in $C$.
\end{definition}

\begin{definition}
\label{def Q}  A  finite connected quiver $Q$ without loops and 2-cycles  is called a \emph{dimer tree quiver} if  it  satisfies the following conditions. 
\begin{itemize}
\item [(Q1)] Every arrow of $Q$ lies in at least one chordless cycle.
\item[(Q2)]\label{tree} The dual graph of $Q$ is a tree. 
\item[(Q3)] The boundary arrows of $Q$ form a simple (non-oriented) cycle.
\end{itemize}
\end{definition}

The following properties follow easily from the definition.
\begin{prop}\label{prop Q}\cite[Proposition 3.4]{SS3}  Let $Q$ be a dimer tree quiver. Then
 \begin{enumerate}
\item $Q$ has no parallel arrows.
\item $Q$ is planar.
\item For all arrows $\za $ of $Q$,
 \begin{enumerate}
\item[(a)] either $\za$ lies in exactly one chordless cycle, 
\item[(b)] or $\za$ lies in exactly two chordless cycles.\end{enumerate}
\item Any two chordless cycles in $Q$ share at most one arrow.
\end{enumerate}
\end{prop}

Since the dual graph $G$ of $Q$  is a tree, the quiver contains a chordless cycle $C_0$ that contains exactly one interior arrow. If $C$ is any chordless cycle in $Q$ we define the distance $d(C)$ of $C$ from $C_0$ to be  the length of the unique path from $C_0$ to $C$ in $G$. We define the \emph{potential} $W$ of the dimer tree quiver $Q$ as \[W=\sum_{C} (-1)^{d(C)} C,\]
 where the sum is taken over all chordless cycles of $Q$.

\begin{remark}Since the dual graph is a tree we can embed the quiver $Q$ into the plane. Then, up to sign, the potential is the sum of the clockwise chordless cycles minus the sum of the counterclockwise chordless cycles. 

 \end{remark}

\begin{definition}\label{def dimer tree} 
Let $Q$ be a dimer tree quiver and $W$ its potential. The Jacobian algebra \[A=\textup{Jac}(Q,W)\] is called a \emph{dimer tree algebra}.
\end{definition}


\subsection{A geometric model for the syzygy category of a dimer tree algebra} 
Let $A=\textup{Jac}(Q,W)$ be a dimer tree algebra. 
By definition, every dimer tree algebra is  2-Calabi-Yau tilted and (hence) Gorenstein of Gorenstein dimension 1. As usual for these algebras, we denote by $\cmp A$ the syzygy category of $A$ and by $\scmp\, A$ its stable category.

\begin{definition}
 (a) For every boundary arrow $\za$ in $Q$, we define its unique \emph{cycle path} (or \emph{zigzag path}) to be
 $\mathfrak{c}(\za)=\za_1\za_2\cdots\za_{\ell(\za)}$, where
 \begin{itemize}
\item [(i)] $\za_1=\za$ and $\za_{\ell(\za)}$  are boundary arrows, and 
 $\za_2,\ldots,\za_{\ell(\za)-1}$ are interior arrows,
\item [(ii)] every subpath of length two $\za_i\za_{i+1}$, 
 is a subpath of a chordless cycle $C_i$, and $C_i\ne C_j$ if $i\ne j$. 
\end{itemize}

 (b) The \emph{weight} $\text{w}(\za)$ of the boundary arrow $\za$ is defined as
\[\textup{w}(\za) =\left\{ 
\begin{array}
 {ll} 1&\textup{if the length of  $\mathfrak{c}(\za)$ is odd;}\\
 2&\textup{if the length of  $\mathfrak{c}(\za)$ is even.}\\
\end{array} \right.\]

(c) The \emph{total weight} of $A$ is defined as $\sum_\za \textup{w}(\za) $,
where the sum is over all boundary arrows of $Q$. 

\end{definition}
We showed in \cite{SS3} that the total weight of a dimer tree algebra is always an even number.
 We  then associate to every dimer tree algebra of weight $2N$ a polygon $\mathcal{S}$ with $2N$ vertices. Moreover $\cals$ carries the additional structure of a checkerboard pattern which is defined by a set of  radical lines $\rho(i)$ associated to the vertices $i$ of $Q$. In particular, each radical line is a 2-diagonal. The checkerboard pattern is constructed as the medial graph of the twisted completed dual graph of $Q$. We refer to \cite{SS3} for the details of this constructions and to section~\ref{sect examples} for examples.

By definition, a 2-diagonal $(i,j)$ connects two boundary points of $\cals$ of opposite parity. We orient each 2-diagonal from its odd endpoint towards its even endpoint. This allows us to talk about the direction of a crossing between 2-diagonals. 
  Then, for each 2-diagonal $\zg$ in $\mathcal{S}$, we define projective modules $P_0(\zg)=\oplus_i P(i)$ and $P_1(\zg)=\oplus_j P(j)$, where the first sum is over all $i$ such that $\rho(i)$ crosses $\zg$ from right to left and the second sum is over all $j$ such that $\rho(j)$ crosses $\zg$ from left to right. 
 
Let $\diags$ be the category of 2-diagonals of $\cals$ introduced in section~\ref{sect A} and $\scmp\, A$ the stable syzygy category of $A$.
  The main result of \cite{SS4} is the following.
\begin{thm}\cite[Theorem 1.1]{SS4}\label{thm SS4}
Let $A$ be a dimer tree algebra  and $\mathcal{S}$ the associated checkerboard polygon. 
For each 2-diagonal $\zg$ in $\cals$ there exists a morphism $f_\zg\colon P_1(\zg)\to P_0(\zg)$ such that the mapping $\zg\mapsto \coker f_\zg$ induces an equivalence of categories 
\[\Phi\colon\diags \to \scmp \, A.\]
 Under this equivalence, the radical line $\rho(i)$ corresponds to the radical of the indecomposable projective $P(i)$ for all $i\in Q_0$. The clockwise rotation $R$ of $\cals$ corresponds to the shift $\zO$ in $\scmp\,A$ and $R^2$ corresponds to the  inverse Auslander-Reiten translation $\tau^{-1}=\zO^2$. 
Thus
\[ 
\begin{array} {lcl}
 \Phi(\rho(i))&=&\rad P(i)\\
  \Phi\circ R&=& \zO\circ \Phi\\
   \Phi\circ R^2&=& \tau^{-1}\circ \Phi
\end{array}
 \]
Furthermore, $\Phi$ maps the 2-pivots in $\diags$ to the irreducible morphisms in $\scmp\, A$, and the  meshes in $\diags$ to the Auslander-Reiten triangles in $\scmp\,A$.
\end{thm}

This result has the following consequences.
\begin{corollary}\label{cor typeA}\cite[section 1]{SS3}
(a) The category $\scmp\, A$ is equivalent to the 2-cluster category of type $\mathbb{A}_{N-2}.$ In particular, the number of indecomposable syzygies is $N(N-2)$.

(b) The total weight of $A$ is a derived invariant.

(c) The projective resolution of any syzygy is periodic of period $N$ or $2N$. An indecomposable syzygy $M_\zg$ has period $N$ if and only if the corresponding 2-diagonal $\zg$ is a diameter in $\cals.$

(d) The indecomposable syzygies over $A$ are rigid $A$-modules.

(e)  Let $L,M$ be indecomposable syzygies over $A$. Then the dimension of $\Ext^1_A(L,M)\oplus \Ext^1_A(M,L)$ is equal to the number of crossing points between the corresponding 2-diagonals. In particular, the dimension is either 1 or 0.

(f) 
Let $\tau$ denote the Auslander-Reiten translation in $\textup{mod}\,A$ and $\nu$ the Nakayama functor. 
We denote the stable cosyzygy category by $\underline{\textup{CMI}}\,A$.
 Then the following diagram commutes. 
\[\xymatrix@C40pt{
\scmp\,A\ar[r]^{\tau}&\underline{\textup{CMI}}\,A\\
\diags\ar[u]^{\textup{cok}\,f_\zg} \ar[r]^{1}
&\diags \ar[u]_{\textup{ker}\,\nu\! f_\zg}  
}
\]
\end{corollary}

\subsection{Dimer tree algebras with $G$-action}
Let $A^0=\textup{Jac}(Q^0,W^0) $ be a dimer	tree algebra and $\za\colon i\to j$ a boundary arrow in $Q^0$. Define a new quiver $Q$ by taking two copies of $Q^0$ and gluing them together by identifying the two copies of the arrow $\za$, see the left picture in Figure~\ref{fig gluing}. Thus every vertex $x$ of $Q^0$, except for $i$ and $j$, produces two vertices $x,\zs x$ in $Q$ and 
every arrow $\zb$ of $Q^0$, except for the arrow $\za$, produces two arrows $\zb,\zs \zb$ in $Q$. By setting $\zs i=i, \zs j=j $, and $\zs \za =\za$, we obtain an admissible $G=\{1,\zs\}$-action on $Q$. 

We define a potential $W$ on $Q$ by $W=W^0-\zs W^0$, and let $A=\textup{Jac}(Q,W)$ be the associated Jacobian algebra. We call $A$ the \emph{fibered product of $A^0$ with itself along $\za$}.

To show that $A$ is again a dimer tree algebra, we need to check the two conditions in Definition~\ref{def Q}. The condition (Q1) that every arrow lies in at least one chordless cycle follows directly from the construction and the fact that it holds for $Q^0$. Condition (Q2) is that the dual graph $\calg$ of $Q$ is a tree. By construction, $\calg$ is given by taking two copies of the dual graph $\calg^0$ of $Q^0$ and gluing them together by identifying the two copies of the leaf edge corresponding to $\za$, see the right picture Figure~\ref{fig gluing}. 
\begin{figure}
\begin{center}
\scalebox{0.8}{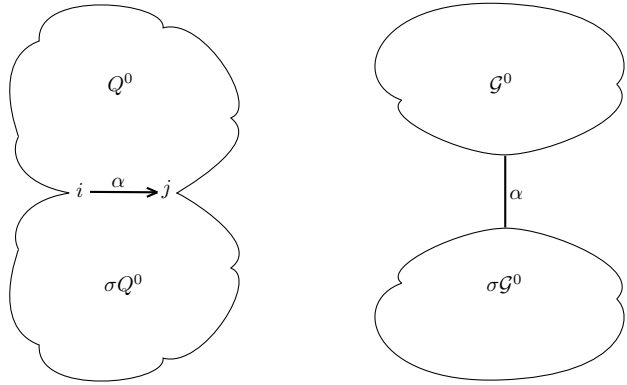}
\caption{Gluing two copies of a dimer tree algebra along a boundary arrow. The left picture shows the quiver and the right picture shows the dual graph.}
\label{fig gluing}
\end{center}
\end{figure}
We have shown the following.
\begin{prop}
 \label{prop gluing 1}
 For every dimer tree algebra $A^0$ and every boundary arrow $\za\colon i\to j$ in $A^0$, the fibered product of $A^0$ with itself along $\za$ is a dimer tree algebra $A$ with an admissible $G$-action. The element $\zs$ acts on the quiver $Q$ of $A$ by reflection at the line along the arrow $\za$ and $\zs$ acts on the dual graph $\calg$ by reflection at the line through the midpoint of the edge $\za$ and perpendicular to the edge $\za$. Moreover, the fixed vertex set of $\zs$ is $\{i,j\}$ and the fixed arrow set is $\{\za\}$. \qed 
\end{prop}

Conversely, suppose that $A$ is a dimer tree algebra with a nontrivial admissible $G$-action.
Suppose first that there is chordless cycle $C$ that is fixed by $\zs$, 
meaning that $\zs$ fixes every arrow and every vertex of $C$. 
Let $C'$ be another chordless cycle that shares an arrow $\zb$ with $C$. Then $\zs C'$ must be a chordless cycle that contains $\zs \zb=\zb$. But $\zb$ lies in at most two chordless cycles, and since $\zs C'\ne \zs C=C$, we must have $\zs C'=C'$.  So $\zs$ fixes every chordless cycle adjacent to $C$. Now using the fact that the dual graph is connected, we see that $\zs$ fixes every chordless cycle and thus the action of $G$ is trivial, a contradiction.
Thus $\zs$ cannot fix a chordless cycle.

Next suppose $\zs$ fixes an arrow $\za$. Then $\za$ must lie in two chordless cycles $C_\za, C_\za '$; otherwise $\zs$ would fix the unique chordless cycle containing $\za$. Thus we have $\zs C_\za=C_\za'$.  Assume now there exists another arrow $\zb$ that is fixed by $\zs$, and denote the two chordless cycles at $\zb$ by $C_\zb,C_\zb'$. Since the dual graph $\calg$ is connected, there exists a path 
$\xymatrix@C10pt{C_\za=C_1 \ar@{-}[r]&C_2\ar@{-}[r]&\cdots \ar@{-}[r]&C_t=C_\zb}$ in $\calg$. Applying $\zs$, we obtain the following cycle in $\calg$
\[\xymatrix@C10pt{C_\za \ar@{-}[d]_\za \ar@{-}[r]&C_2\ar@{-}[r]&\cdots \ar@{-}[r]&C_\zb \ar@{-}[d]^\zb\\
C_\za'  \ar@{-}[r]&\zs C_2\ar@{-}[r]&\cdots \ar@{-}[r]&C_\zb' }
\]
which is impossible, because $\calg$ is a tree. We have shown that $\zs $ fixes at most one arrow. 

Finally assume $\zs$ fixes no arrow at all.  Since our $G$-action is admissible, there exists a vertex $i$  that is fixed by $\zs$. 
 Since every arrow lies in a chordless cycle, there must be at least one incoming arrow $\za$ and one outgoing arrow $\zb$ at $i$ that both lie in the same chordless cycle $C$. Applying $\zs$, we obtain 
a second chordless cycle $\zs C$ at $i$ that contains the arrows $\zs \za $ and $\zs \zb$. 
Since $\calg$ is connected, there exists a path in $\calg$ from $C$ to $\zs C$ which we denote by 
$w\colon\xymatrix@C10pt{C\ar@{-}[r]^{\za_0}&C_1 \ar@{-}[r]^{\za_1}&\cdots \ar@{-}[r]&C_t \ar@{-}[r]^{\za_t}&\zs C}$. Applying $\zs$, we obtain a path 
$\zs w\colon\xymatrix@C10pt{\zs C\ar@{-}[r]^{\zs \za_0}&\zs C_1 \ar@{-}[r]^{\zs \za_1}&\cdots \ar@{-}[r]&\zs C_t \ar@{-}[r]^{\zs \za_t}& C}$.
Since $\calg$ is a tree, the path $\zs w$ must be the reverse of the path $w$.  Furthermore, since $\zs$ does not fix any arrows, the number of edges in $w$ must be even, otherwise the central edge would be fixed. However, since $\zs$ does not fix any chordless cycles, the number of edges in $w$ must be odd, otherwise the central chordless cycle would be fixed. This clearly is a contradiction. 

We have shown that $\zs$ fixes exactly one arrow $\za$, that $\za$ is contained in two chordless cycles $C_\za, C_\za'$ and that $\zs C_\za=C_\za'$.  Removing the edge $\za$ from the dual graph $\calg$ produces two connected components
\[ \calg\setminus \{\za\} = \calg^0\sqcup {\calg^0}'\]
one containing the point $C_\za$ and the other containing $C_\za'$. Since $\zs \calg^0$ is a connected subgraph of $\calg$ that contains $\zs C_\za$ and does not contain $\za$, we get $\zs \calg^0={\calg^0}'$.

Considering this information on the level of the quiver $Q$, we obtain two subquivers $Q^0, \zs Q^0$ of $Q$ that share the arrow $\za$ but no other arrows, and whose dual graphs are $\calg^0\cup\{\za\}$ and $\zs\calg^0\cup\{\za\}$.
 The quivers $Q^0$ and $\zs Q^0$ are isomorphic and the arrow $\za$ is a boundary arrow in both. Also both quivers satisfy the conditions (Q1) and (Q2) of Definition~\ref{def Q}. 
The restriction $W^0$ of the potential $W$ to $Q^0$ is a potential and the Jacobian algebra $A^0=\textup{Jac}(Q^0,W^0)$ is a dimer tree algebra. We have shown the following.

\begin{prop}
 \label{prop glue 2}
Let $G=\{1,\zs\}$ be the group with two elements.
 If $A=\textup{Jac}(Q,W)$ is a dimer tree algebra with a nontrivial admissible $G$-action then 
 the following holds.

 (a) $\zs$ fixes exactly one arrow $\za\colon i\to j$. 
 
 (b) There exists a subquiver $Q^0$ that is a dimer tree quiver such that  $Q=Q^0\cup \zs Q^0$ and $Q^0\cap \zs Q^0=\{\za\}$.
 
 (c) Let $e^0=\sum_{i} e_i$, where the sum is over all vertices in $Q^0$, and let $A^0=A/A(1-e^0)A$. Then $A$ is the fibered product $A^0$ with itself along $\za$.
 
(d) There is a planar embedding of $Q$ and $\calg$ such that $\zs$ acts on $Q$ by reflection at the line through the arrow  $\za$, and on $\calg$ by reflection at the line  through the midpoint of the edge $\za$ and perpendicular to $\za$.
\qed \end{prop}

\subsection{$G$-action on the checkerboard polygon}
Let $A=\textup{Jac}(Q,W)$ be a dimer tree algebra of total weight $2N$ with an admissible $G$-action. Let $\calg$ be the dual graph of $Q$. 

The checkerboard polygon $\cals$ is constructed by taking the medial graph of the twisted (completed) dual graph $\widetilde \calg$, see \cite{SS3} for details. 
Recall that the twisted dual graph is obtained from the dual graph $\calg$ by twisting along every edge of $\calg$. Since the edges in $\calg\setminus \{\za\}$ come in pairs $(e,\zs e)\in \calg^0\times\zs\calg^0$, we see that after performing the twist at all edges except $\za$, we still have a graph whose $\zs$-action is given by the reflection at the line through the midpoint of the edge $\za$ and perpendicular to $\za$.  This situation is illustrated in the left picture in Figure~\ref{fig twist}.
\begin{figure}
\begin{center}
\scalebox{0.8}{
\begingroup%
  \makeatletter%
  \providecommand\color[2][]{%
    \errmessage{(Inkscape) Color is used for the text in Inkscape, but the package 'color.sty' is not loaded}%
    \renewcommand\color[2][]{}%
  }%
  \providecommand\transparent[1]{%
    \errmessage{(Inkscape) Transparency is used (non-zero) for the text in Inkscape, but the package 'transparent.sty' is not loaded}%
    \renewcommand\transparent[1]{}%
  }%
  \providecommand\rotatebox[2]{#2}%
  \newcommand*\fsize{\dimexpr\f@size pt\relax}%
  \newcommand*\lineheight[1]{\fontsize{\fsize}{#1\fsize}\selectfont}%
  \ifx\svgwidth\undefined%
    \setlength{\unitlength}{412.55492275bp}%
    \ifx\svgscale\undefined%
      \relax%
    \else%
      \setlength{\unitlength}{\unitlength * \real{\svgscale}}%
    \fi%
  \else%
    \setlength{\unitlength}{\svgwidth}%
  \fi%
  \global\let\svgwidth\undefined%
  \global\let\svgscale\undefined%
  \makeatother%
  \begin{picture}(1,0.3920689)%
    \lineheight{1}%
    \setlength\tabcolsep{0pt}%
    \put(0,0){\includegraphics[width=\unitlength,page=1]{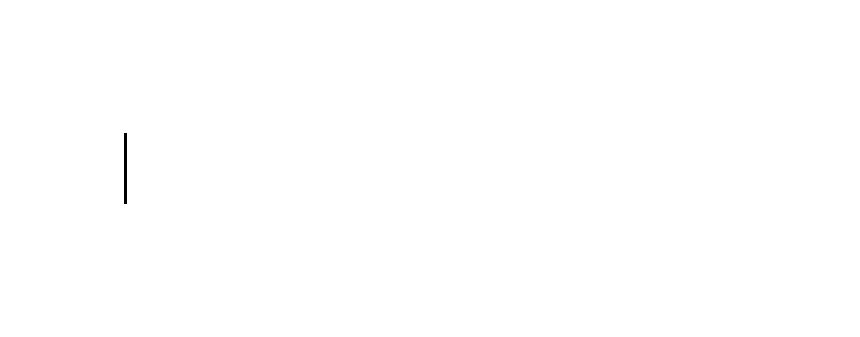}}%
    \put(0.15179103,0.18916211){\makebox(0,0)[lt]{\lineheight{1.25}\smash{\begin{tabular}[t]{l}$\za$\end{tabular}}}}%
    \put(0.21620357,0.3312188){\makebox(0,0)[lt]{\lineheight{1.25}\smash{\begin{tabular}[t]{l}$\overline\calg^0$\end{tabular}}}}%
    \put(0.21256764,0.05125599){\makebox(0,0)[lt]{\lineheight{1.25}\smash{\begin{tabular}[t]{l}$\zs \overline\calg^0$\end{tabular}}}}%
    \put(0,0){\includegraphics[width=\unitlength,page=2]{figtwist.pdf}}%
    \put(0.66081397,0.18916211){\makebox(0,0)[lt]{\lineheight{1.25}\smash{\begin{tabular}[t]{l}$\za$\end{tabular}}}}%
    \put(0.72522659,0.3312188){\makebox(0,0)[lt]{\lineheight{1.25}\smash{\begin{tabular}[t]{l}$\widetilde{\calg}^0$\end{tabular}}}}%
    \put(0.55070442,0.05125599){\makebox(0,0)[lt]{\lineheight{1.25}\smash{\begin{tabular}[t]{l}$\zs \widetilde\calg^0$\end{tabular}}}}%
    \put(0,0){\includegraphics[width=\unitlength,page=3]{figtwist.pdf}}%
  \end{picture}%
\endgroup%
}
\caption{The left picture illustrates the dual graph after twisting at every edge except at $\za$, and the right picture illustrates the resulting graph after twisting at $\za$ (this is the twisted dual graph). The action of $\zs$ on the left is by reflection, while the action on the right is by rotation.}
\label{fig twist}
\end{center}
\end{figure}
The twisted dual graph is now obtained by twisting along the edge $\za$; the result is illustrated in the right picture in Figure~\ref{fig twist}. Therefore the induced action of $\zs$ on the twisted dual graph is given by a rotation by angle $\pi$ with center the midpoint of the edge $\za$. 

On the medial graph $\cals$, the action of $\zs$ is also given by the rotation by $\pi$ with center the crossing point of the two radical lines $\rho(i)$ and $\rho(j)$, where $\za\colon i\to j$.  

Thus when we draw $\cals $ as a regular $2N$-gon with checkerboard pattern determined by the radical lines $\rho(x)$ of the medial graph, the action of $\zs$ on $Q$ induces and action on $\cals$ given by the rotation by $\pi$ at the center point of the polygon $\cals$. 
This is made precise in the following result.
\begin{prop}
 \label{prop G action on S}
 With the above notation, label the vertices of the polygon $\cals$ by $1,2,\ldots, 2N$ in clockwise order around the boundary.
 The $G$-action on the dimer tree algebra $A$ induces a $G$-action on the category of 2-diagonals $\diags$ of the checkerboard $2N$-gon $\cals$, where the action of $\zs$ is given by the rotation by angle $\pi$. If $\zg\in\diags$ is a 2-diagonal with endpoints 
$k,l$,  and $1\le k<l\le 2N$ then
 
 (a) $\zs\zg$ has endpoints $k+N, l+N\ (\textup{mod}\, 2N)$.
 
 (b) $\zs\zg=\zg$ if and only if $\zg $ is a diameter, that is $l=k+N$.
 
 (c) For the radical line $\rho(x)$, we have $\zs \rho(x)=\rho(\zs x)$.
 
 (d) Let $\za\colon i\to j$ be the unique arrow in the quiver of $A$ that is fixed by $\zs$.
  Then the radical lines $\rho(i)$ and $\rho(j)$ are fixed by $\zs$. 
 
 (e) $N$ is odd. 
 
 (f) Let $\Phi\colon \diags\to \scmp\, A$ be the equivalence of Theorem~\ref{thm SS4}. Then
 \[ \Phi(\zs\zg)=\zs\Phi(\zg),\]
 where the $\zs$ action on the right hand side is the one defined in section~\ref{Gmod}.
\end{prop}

\begin{proof}
 (a) and (b) follow directly from the discussion above.  (c) follows from Lemma~\ref{lem Gmod}(b) and the fact that the $G$-action on $\cals$ is induced by the $G$-action on $A$. 
(d) is a direct consequence of (c). 

(b) and (d) together imply that $\rho(i)$ and $\rho(j)$ are diameters in the $2N$-gon $\cals$. So their endpoints are of the form $h,h+N$ for some $h$. But these diameters are also 2-diagonals, so the two endpoints must have opposite parity. Thus their difference $N$ must be odd. This shows (e).
 
It remains to prove (f). The functor $\Phi$ sends a 2-diagonal $\zg$ to a syzygy $M_\zg$ that is uniquely determined by the crossing pattern of $\zg$ with the radical lines in $\cals$. Indeed, let   $P_0(\zg)=\oplus_i P(i)$ and $P_1(\zg)=\oplus_j P(j)$, where the first sum is over all $i$ such that $\rho(i)$ crosses $\zg$ from right to left and the second sum is over all $j$ such that $\rho(j)$ crosses $\zg$ from left to right. 
 Then $M_\zg$ is the cokernel of a morphism $f_\zg\colon P_1(\zg)\to P_0(\zg)$. 
 
 Since $\zg$ acts on $\cals $ by rotation, the radical line $\rho (x)$ crossed $\zg$ if and only if $\zs\rho(x)$ crosses $\zs \zg$, and both crossings are in the same direction.  Hence
 \[ P_0(\zs\zg)=\oplus_i P(\zs i)=\zs P_0(\zg) \qquad 
  P_1(\zs\zg)=\oplus_j P(\zs j)=\zs P_1(\zg),
 \]
 and therefore $\Phi(\zs\zg)=M_{\zs \zg}=\zs M_\zg=\zs\Phi(\zg)$.
\end{proof}


\section{A geometric model for the syzygy category of the skew group algebra} \label{sect 5}
Throughout this section, we use the following setting. Let $G=\{1,\zs\}$ denote the group of order two, and let $A$ be a dimer tree algebra of total weight $2N$ with an admissible $G$-action. The corresponding checkerboard polygon $\cals$ has $2N$ vertices. Under the Baur-Marsh equivalence the category of 2-diagonals $\diags$ is equivalent to the 2-cluster category $\calc^2_{\mathbb A_{N-2}}$. 
In this section, we give a geometric model for the syzygy category of the skew group algebra $AG$.

\subsection{An equivalence of categories}

We have shown in section~\ref{sect 31}
that the $\zs$-action on the 2-cluster category corresponds under the Baur-Marsh equivalence $\BM_{\mathbb{A}}$ to the rotation of the polygon by angle $\pi$. On the other hand, Proposition~\ref{prop G action on S} states that this $\zs$-action on the polygon also corresponds to the $\zs$-action on the syzygy category of the dimer tree algebra under the equivalence $\Phi$. Thus there is a commutative diagram as follows.
\[\xymatrix@C50pt{
\calc^2_{\mathbb{A}_{N-2}} \ar@{<-}[r]^{\BM_\mathbb{A}}_\cong\ar[d]_\zs & \diags \ar[d]^\zs\ar[r]^\Phi_\cong
&
\scmp\, A \ar[d]^\zs 
\\
\calc^2_{\mathbb{A}_{N-2}}\ar@{<-}[r]_{\BM_\mathbb{A}}^\cong 
&\diags\ar[r]_\Phi^\cong 
&\scmp\, A
}
\] 
On each level of this diagram, we can consider the induction functor $F$ to the the skew group category following Reiten and Riedtmann \cite{RR}. Starting from the 2-cluster category of type $\mathbb{A}_{N-2}$, we obtain  the 2-cluster category of type $\mathbb{D}_{(N+1)/2}$ as seen in  Proposition~\ref{prop 2-clusters2}. On the level of $\diags$, we obtain the category $\arcp$ of 2-arcs in a punctured $N$-gon described in section~\ref{sect BMD}. 
And by  Proposition \ref{prop Fsyz},
from the syzygy category of the dimer tree algebra $A$, we obtain the syzygy category of the skew group algebra $AG$. 
All these categories are equivalent, because of the commutativity of the diagram above. 
We thus have the following result.
\begin{thm}\label{thm equivalences}
 Let $A$ be a dimer tree algebra with admissible $G$-action, and let $2N$ denote the total weight of $A$. Then there are equivalences of categories 
 \[\xymatrix@C50pt{\calc^2_{\mathbb{D}_{(N+1)/2}} &\ar[l]_{\BM_\mathbb{D}}^\cong \arc(\calp_N) \ar[r]^\Psi_\cong &\scmp \,AG}.\]
\end{thm}

The first equivalence is  described in Theorem~\ref{thm BM2}. We shall now give a description of the second equivalence.

We define a functor $\Psi\colon \arcp\to \scmp \,AG$ as follows.
If $\zgg=(i,j)$ is a 2-arc in the punctured $N$-gon $\calp$, then define $\zg\in \diags$  to be the diameter $\zg=(i,i+N) \textup{ if the 2-arc } \zgg=(i,i)^\pm$ is a loop, and otherwise define $\zg$ to be the 2-diagonal 
\[ \zg=\left\{ 
\begin{array}
 {ll}
(i,j) & \textup{ if } i<j;\\
(i,j+N) & \textup{ if } i>j.
\end{array}
\right.
\]

Consider the image $\Phi(\zg)\in \scmp\, A$. According to Proposition~\ref{prop G action on S} (f), we have
$ \zs \Phi(\zg)=\Phi(\zs \zg)$.
Thus if $\zg=(i,i+N)$ is a diameter then $\Phi(\zg)$ is fixed by $\zs$, 
and if $\zg=(i,j)$ is not a diameter then $\Phi(\zg)$ is not fixed by $\zs$. Applying the induction functor $F= - \otimes_A AG$ and using Proposition~\ref{prop RR}, we have
\begin{enumerate}
\item If $\zgg=(i,i)^\pm$ then $F\Phi(\zg)$ is the direct sum of two indecomposable syzygies which we denote by $\tilde M_{\zgg}^+$ and $\tilde M_{\zgg}^-$.
\item If $\zgg=(i,j)$ with $i\ne j$ then $F\Phi(\zg)$ is the direct sum of two copies of the same indecomposable syzygy which we denote by  $\tilde M_{\zgg}$.
\end{enumerate}
In case (2) above, we define $\Psi(\zgg)=\tilde M_{\zgg}$. In case (1), we need to make a choice that is consistent with the Auslander-Reiten translations. To achieve this, we label all modules in the $\tau$-orbit of $\tilde M_{\zgg}^+$ in $\scmp\, AG$ by a $+$ sign. Thus $\tau^t (\tilde M_{\zgg}^+)= (\tau^t \tilde M_{\zgg})^+ $, for all $t$. Similarly, we label the modules in  the $\tau$-orbit of $\tilde M_{\zgg}^-$ such that $\tau^t (\tilde M_{\zgg}^-)= (\tau^t \tilde M_{\zgg})^- $, for all $t$.
Since the AR translation in $\arcp$ also preserves the sign, we obtain
\[\Psi \circ\tau_{\scalebox{0.7}{$\arcp$}} =\tau_{\scalebox{0.6}{$\scmp \,AG$}}\circ \Psi.\]

Next we define $\Psi$ on irreducible morphisms. Let $\tilde g\colon \zgg\to\zdd$ be a 2-pivot in $\arcp$. We distinguish three cases depending on whether or not $\zgg,\zdd$ are loops.
\begin{enumerate}
\item Suppose first that none of the two 2-arcs is a loop. Let $\zgg=(i,j)$, then $\zdd=(i,j+2)$ or $\zdd=(i+2,j)$. 
In this situation $\tilde g$ corresponds to two 2-pivots $g\colon \zg\to \zd$ and $\zs \zg\colon \zs\zg\to\zs\zd$ in $\diags$. Thus $F\Phi(g)=F\Phi(\zs g)$ is an irreducible morphism in $\scmp \,AG$ and we let $\Psi(\tilde g) = F\Phi (g)$.
\item Let $\zgg=(i,i)^\pm$. Then $\zdd=(i+2,i)$. Thus $\zgg$ corresponds to the diameter $(i,i+N)$ in $\diags$ and $\tilde g$ corresponds to two 2-pivots, one is  $g\colon\zg\to\zd$ with $\zd=(i,i+N+2)$ and the other is $\zs g\colon\zg\to\zs\zd$ with $\zs\zd=(i+2,i+N)$.

Since $\zg$ is fixed by $\zs$ in $\cals$, we obtain that $F\Phi(g)$ is a sum of two morphisms
\[F\Phi(g)\colon \Psi\left((i,i)^+\right)\oplus \Psi\left((i,i)^-\right) \to \Psi((i+2,i)).\]
We define $\Psi (\tilde g)$ to be the first component if $\zgg=(i,i)^+$ and the second component if $\zgg=(i,i)^-$.

\item Let $\zdd=(i,i)^\pm$. Then $\zgg=(i,i-2)$. This case is dual to the previous one. Again $\tilde g\colon \zgg\to \zdd$ corresponds to two 2-pivots $g\colon \zg\to \zd$ and $\zs g\colon\zs \zg\to \zd$ in $\diags$. Since $\zd$ is fixed by $\zs$, we have  that  $F\Phi(g)$ is a sum of two morphisms
\[F\Phi(g)\colon \Psi(i,i-2)\to  \Psi\left((i,i)^+\right)\oplus \Psi\left((i,i)^-\right).\]
We define $\Psi (\tilde g)$ to be the first component if $\tilde \zd=(i,i)^+$ and the second component if $\tilde\zd=(i,i)^-$.
\end{enumerate}

\begin{thm}\label{thm Fequiv}
 The functor $\Psi$ is an equivalence of categories which makes the following diagram commute.
 \[ \xymatrix@R30pt@C50pt{\diags\ar[d]_\Phi^\cong\ar[r]^F &\arcp\ar[d]^\Psi_\cong \\
 \scmp\, A \ar[r]^F&\scmp\, AG}
 \]
\end{thm}
\begin{proof}
The commutativity of the diagram follows from our construction of $\Psi$.

 Let
 $\tilde\zeta:\xymatrix@C10pt{0\ar[r] &\tau \zgg\ar[r]^{\tilde u}& \tilde E\ar[r]^{\tilde v}&\zgg\ar[r]&0}$ be a mesh in the Auslander-Reiten quiver of $\arcp$. From the definition of $\Psi$ together with Propositions~\ref{prop RR2} and \ref{prop 2-clusters2}, we see that $\Psi(\tilde \zeta)$ is an Auslander-Reiten triangle in $\scmp \,AG$. Thus $\Psi $ is well-defined.
 
 Conversely, the same propositions show that every Auslander-Reiten triangle of $\scmp\, AG$ is the image under $\Psi$ of a mesh in $\arcp$. Since  the morphisms in $\arcp$ are generated by the 2-pivots modulo the mesh relations, this implies that $\Psi$ is faithful. 
 To show that $\Psi$ is full and dense, let $\tilde f\colon \tilde M\to \tilde N$ be a morphisms in $\scmp\, AG$. The restriction functor maps $\tilde f$ to a sum of morphism  in $\scmp\, A$ which we can lift along $\Phi$ to $\diags$. Applying $F$, we obtain a sum of morphisms in $\arcp$ and, by definition, $\Psi$ maps one of these morphisms to $\tilde f$. Thus $\Psi $ is full and dense.
\end{proof}

\section{Checkerboard pattern for the punctured polygon $\calp$}\label{sect 6}
In section~\ref{sect 5}, we have constructed an equivalence of categories $\Psi\colon\arcp\to\scmp\,AG$. In this section, we give a combinatorial construction of $\Psi$ by introducing a checkerboard pattern on the punctured polygon $\calp$. The input data for our construction is the quiver of the skew group algebra $AG$. Here we will work in the basic algebra and we start by recalling its definition from \cite{AP}.

\subsection{The basic algebra \mbox{$B=\overline{e}AG\overline{e}$}} The skew group algebra $AG$ in Theorem~\ref{thm equivalences} is not basic. Indeed, for every vertex $i$ in the quiver $Q$ of $A$ that is not fixed by $\zs$, we have $(e_i\otimes 1) AG \cong (e_{\zs i}\otimes 1)AG$. 
In \cite{AP}, Amiot and Plamondon construct in great detail a complete set of primitive orthogonal idempotents $\{e_i^{\pm}
\mid i\in Q_0\}$ for $AG$. They then specify an idempotent $\overline{e}$ such that $B=\overline{e}AG\overline{e} $ is basic. 

In order to define $\overline{e}$, they decompose the set of vertices of $Q$ as 
$Q_0=Q_0^G\sqcup I\sqcup \zs I$, where $Q_0^G$ is the set of vertices that are fixed by $\zs$ and $I$ is a complete set of representatives of the nontrivial $\zs$-orbits. Then $\overline {e} = \sum_{i\in Q_0^G} e_i^++e_i^-  + \sum_{i\in I} e_i^+ $.

Consequently the quiver $Q_B$ of $B$ can be described as follows.
The vertices of $Q_B$ are
\begin{enumerate}
\item[(i)] two vertices $i^+$ and $i^-$ for each $i\in Q_0^G$;
\item[(ii)] one vertex $i$  for each $i\in I$.
\end{enumerate}
The arrows of $Q_B$ are
\begin{enumerate}
\item[(i)] two arrows $\za^+\colon i^+\to j^+$ and $\za^-\colon i^-\to j^-$ for every arrow $\za\colon i\to j$ such that $\zs \za=\za$;
\item[(ii)] two arrows $\za^+\colon i^+\to j$, and $\za^-\colon i^-\to j$ for every arrow $\za\colon i \to j$ with $i\in Q_0^G$, $j\notin Q_0^G$;
\item[(iii)] two arrows $\za^+\colon i\to j^+$, and $\za^-\colon i\to j^-$ for every arrow $\za\colon i \to j$ with $i\notin Q_0^G$, $j\in Q_0^G$;
\item[(iv)] one arrow $\za\colon i\to j$ for each arrow $\za\colon i\to j$ with $i,j\notin Q_0^G$. 
\end{enumerate}

\bigskip
Now let $A=\textup{Jac}(Q,W)$ be a dimer tree algebra with admissible $G$-action. Let $A^0=\textup{Jac}(Q^0,W^0)$ be the dimer tree algebra with boundary arrow $\za\colon i\to j$ such that $A$ is the fibered product of $A^0$ with itself along $\za$ as in Proposition~\ref{prop glue 2}. 
Then $Q_B$ is the quiver obtained form $Q^0$ by replacing the vertices $i$ and $j$ by the vertices $i^+,i^-$ and $j^+,j^-$, 
the arrow $\za\colon i\to j$ by two arrows $\za^+\colon i^+\to j^+$ and $\za^-\colon i^-\to j^-$, and each other arrow $\zb$ starting or ending at $i$ or $j$ by two arrows $\zb^+$ and $\zb^-$ starting or ending at $i^+,i^-$ or $j^+,j^-$, see Figure~\ref{figQ0QB}.
\begin{figure}
\begin{center}
\scalebox{0.6}{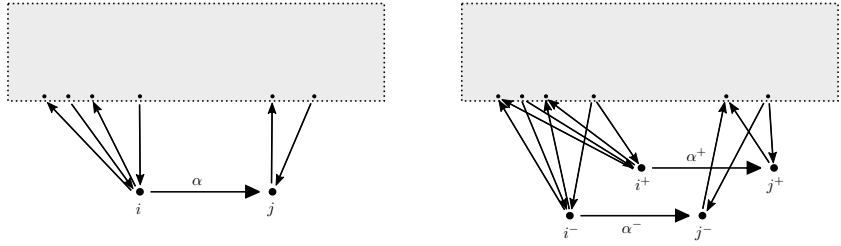}
\caption{The picture on the left illustrates the quiver $Q^0$ and the picture on the right the quiver $Q_B$.}
\label{figQ0QB}
\end{center}
\end{figure}

\subsection{Orientation and crossing of 2-arcs}
Recall that in the $2N$-gon $\cals$ the 2-diagonals $(i,j)$ are oriented from the odd endpoint towards the even endpoint. 
According to Corollary~\ref{cor typeA}(e), 
the 2-diagonals $\zg,\zd$ cross in $\cals$ if and only if $\Ext^1_{\scmp\,A} (\Phi(\zg),\Phi(\zd))\oplus \Ext^1_{\scmp\,A} (\Phi(\zd),\Phi(\zg)) \ne 0 $. More precisely 

\begin{itemize}
\item $\zg$ crosses $\zd$ from left to right if and only if $\Ext^1_{\scmp\,A} (\Phi(\zg),\Phi(\zd))\ne 0$ and 

\item $\zg$ crosses $\zd$ from right to left if and only if $\Ext^1_{\scmp\,A} (\Phi(\zd),\Phi(\zg))\ne 0$.

\end{itemize}

Given 2-diagonals $\zg,\zd$ that cross in $\cals$, we define the directed crossing number $e(\zg, \zd)$ to be 1 if $\zg$ crosses $\zd$ from left to right and 0 otherwise.  Similarly, we define $e(\zd, \zg)$ to be 1 if $\zg$ crosses $\zd$ from right to left and 0 otherwise.  Note that $e(\zg, \zd)= \textup{dim Ext}^1_{\scmp\,A}(\Phi(\zg),\Phi(\zd))$.

The corresponding notions of orientation and crossing is somewhat more complicated in the punctured polygon $\calp$. Also the crossing notion in $\calp$ is different form the one used in the theory of cluster algebras in \cite{FST}, because we are working here in the 2-cluster category and not in the ordinary cluster category. 

\subsubsection{Orientation} 
The $2N$-gon $\cals$ is a 2-fold branched cover of the punctured $N$-gon $\calp$, which is given by the rotation action of $\zs$ by angle $\pi$. \footnote{Strictly speaking we would need to replace the puncture by an orbifold point of order two. But in our combinatorial approach this will make no difference.}
The fixed point of $\zs$ is the center of $\cals$ and it corresponds to the puncture in $\calp$. We will now carefully choose a fundamental domain for the action of $\zs$ in $\cals$. 

As seen in Proposition~\ref{prop gluing 1}, 
 in the checkerboard pattern, the center of the polygon $\cals$ is given by the crossing point of the radical lines $\rho(i)$ and $\rho(j)$, that is to say, by the arrow $\za\colon i\to j$ that is fixed by $\zs$.  There are exactly two shaded regions $C$ and $\zs C$ incident to this point which correspond to the two chordless cycles that contain $\za$. Consequently, there are also exactly two white regions $W,\zs W$ incident to the point, see Figure~\ref{figfundom}  and the examples in section~\ref{sect examples}. Like any white region in a checkerboard polygon, $W$ contains either exactly one boundary point of $\cals$  or a boundary segment together with its two endpoints. In the former case, we label the sole boundary point in $W$ by 1 and in the latter case, we label the two boundary points in $W$ by $2N, 1$ in clockwise order. Then label the remaining boundary points in clockwise order. 
\begin{figure}
\begin{center}
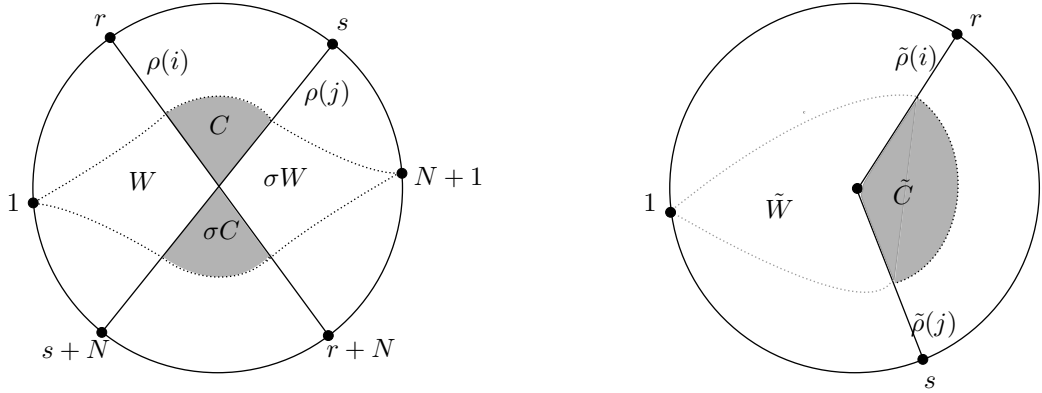
\caption{Construction of the fundamental domain.}
\label{figfundom}
\end{center}
\end{figure}

Then there exists $1<r,s< N+1$ such that the two radical lines are $\rho(i)=(r,r+N)$ and $\rho(j)=(s,s+N)$. Without loss of generality, we may assume that $r<s$ and the checkerboard regions at the center point $\za$ are $W,C,\zs W,\zs C$ in clockwise order. 
The diameter $(1,N+1)$ is not a radical line because it runs through the white region $W$. This diameter cuts the polygon into two pieces. We use the piece with boundary points $1,2,\ldots, N+1$ as  our fundamental domain for the $\zs$-action. 

This choice induces a labeling of the boundary vertices of the punctured $N$-gon by $1\cong N+1,2,\ldots, N$ in clockwise order. The radical lines $\tilde\rho(i)=(r,r)^+$ and $\tilde\rho(j)=(s,s)^+$ are both loops in $\calp$ and as usual we draw them as straight lines from their boundary point to the puncture. If the boundary point has an odd label, we orient the line towards the puncture, and otherwise, we orient the line towards the boundary point. 

There is exactly one shaded region $\tilde C$ and one white region $\tilde W$ incident to the puncture. Every other radical line $\tilde\rho(h)$ in $\calp$, lifts to  a unique radical line $\rho(h)$ in our fundamental domain in $\cals$, because no radical line can cross the white region $W$. 
We orient these radical lines in the same way as in the fundamental domain, that is, from the odd labeled endpoint towards the even labeled endpoint. 
Note that one vertex in $\calp$ carries two labels $1=N+1$. If $\rho(h)$ ends at the vertex $N+1$ in $\cals$ then $\rho(h)$ is oriented towards $N+1$ in $\cals$, because $N+1 $ is even. Consequently its image $\tilde\rho(h)\in \calp$ is oriented towards the vertex $1=N+1$ as well.
This defines an orientation on  all radical arcs in $\calp$. 

Now let $\zgg=(i,j)$ be an arbitrary 2-arc in $\calp$. Suppose first that $\zgg$ is not a loop. Let $\ell$ denote the straight line segment from the boundary point $1$ to the puncture. If $\zgg$ does not cross $\ell$, then $i<j$ and we orient $\zgg$ from its odd labeled endpoint towards its even labeled endpoint. 
If $\zgg$ crosses the line $\ell$, we label  the crossing point by $x$, and we orient the two subcurves $(i,x)$ and $(x,j)$  of $\zgg$ as follows. Note that in this case we have $i>j$ and both $i,j$ have the same parity since $N$ is odd.  We orient the ends $(i,x),(x,j)$ towards the point $x$  if $i$ and $j$ are odd, and towards $i $ and $j$ if $i$ and $j$ are even, see Example~\ref{ex71}. 

Now suppose that $\zgg=(i,i)^\pm$ is a loop. Then we orient $\zgg$ towards the puncture if $i$ is odd and towards the boundary if $i$ is even. The case where $\zgg=(1,1)^\pm$ is slightly different. If we choose a representative of $\zgg$ that lies to the left of the line segment $\ell$ then we orient it towards the puncture. This is consistent with the boundary point being labeled by the odd integer 1. If on the other hand, we choose a representative of $\zgg$ that lies to the right of $\ell$ then we orient it towards the boundary. Here we think of the endpoint being labeled by the even integer $N+1$.

\subsubsection{Crossing}
Now that we have orientated 2-arcs, we can define their directed crossings. Throughout this section we use the notation $\underline \zgg \in \zgg$ to express that the curve $\underline\zgg$ is a representative of the 2-arc $\zgg$.  Thus $\zgg$ is the homotopy class of $\underline{\zgg}$.

Let $\zgg,\zdd$ be two 2-arcs in $\calp$ and suppose first that at least one of them is not a loop. Define the directed crossing number $e(\zgg,\zdd)$ to be the minimum of the number of crossing points between $\underline \zgg\in\zgg$ and $\underline \zdd\in\zdd$ such that $\underline \zgg $ crosses $\underline \zdd$  from left to right. 
Similarly, define $e(\zdd,\zgg)$ to be the minimum of the number of crossing points between $\underline \zgg\in\zgg$ and $\underline \zdd\in\zdd$ such that $\underline \zgg $ crosses $\underline \zdd$  from  right to left.

We give several examples in Figure~\ref{fig crossing number}.
\begin{figure}
\begin{center}
\scriptsize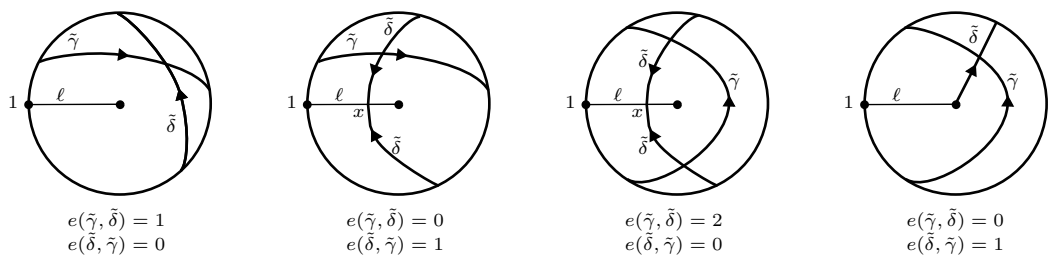
\caption{Examples of crossing numbers between 2-arcs in the punctured polygon.}
\label{fig crossing number}
\end{center}
\end{figure}

 Now suppose that both $\zgg=(x,x)^\pm$ and $\zdd=(y,y)^\pm$ are loops in $\cals$. Then we define the crossing numbers as follows.
 \[
\begin{array}
 {rcl}
 e\left( (x,x)^+, \zdd\right) &=&\left\{
\begin{array}
 {ll}
 1 &\textup{if $\zdd\in\{(x-2,x-2)^+,(x-4,x-4)^-,(x-6,x-6)^+,(x-8,x-8)^-,$}\\
 &\ldots, (x+1,x+1)^\pm\} ; \\[5pt]
 0 & \textup{if $\zdd$ is any other loop.}
\end{array}\right. \\
\\
e\left( (x,x)^-, \zdd\right) &=&\left\{
\begin{array}
 {ll}
 1 &\textup{if $\zdd\in\{(x-2,x-2)^-,(x-4,x-4)^+,(x-6,x-6)^-,(x-8,x-8)^+,$}\\
 &\ldots, (x+1,x+1)^\mp\} ; \\[5pt]
 0 & \textup{if $\zdd$ is any other loop.}
\end{array}\right. \\
\\
e\left(\zdd, (x,x)^+\right) &=&\left\{
\begin{array}
 {ll}
 1 &\textup{if $\zdd\in\{(x+2,x+2)^+,(x+4,x+4)^-,(x+6,x+6)^+,(x+8,x+8)^-,$}\\
 &\ldots, (x-1,x-1)^\pm\} ; \\[5pt]
 0 & \textup{if $\zdd$ is any other loop.}
\end{array}\right. \\
\\
e\left(\zdd, (x,x)^-\right) &=&\left\{
\begin{array}
 {ll}
 1 &\textup{if $\zdd\in\{(x+2,x+2)^-,(x+4,x+4)^+,(x+6,x+6)^-,(x+8,x+8)^+,$}\\
 &\ldots, (x-1,x-1)^\mp\} ; \\[5pt]
 0 & \textup{if $\zdd$ is any other loop.}
\end{array}\right. \\
\end{array}
 \]

\begin{example}\label{ex62}
 In Example~\ref{ex71}, the radical line $\tilde \rho(1^+)$ is given by the loop $(3,3)^+$. If $\zg$ is a loop, then we have
 \[e((3,3)^+,\zg) = 1 \quad \textup{ if and only if } \quad \zg=(4,4)^+,(6,6)^-,(1,1)^+;\]
  \[e(\zg,(3,3)^+) = 1\quad  \textup{ if and only if }\quad  \zg=(5,5)^+,(7,7)^-,(2,2)^+.\]
\end{example}

\begin{prop}
 \label{lem crossing ext}
 Let $\zgg,\zdd$ be two 2-arcs  $\arcp$ and $\BM_\mathbb{D}\colon \arcp\to\calc^2_{\mathbb{D}}$ the equivalence of Theorem~\ref{thm BM2}. Then
 \[\dim \textup{Ext}^1_{\calc^2_{\mathbb{D}}} (\BM_\mathbb{D}(\zgg),\BM_\mathbb{D}(\zdd)) = e(\zgg,\zdd).\]
\end{prop}
\begin{proof}
 This follows by direct inspection of the support of $\Ext^1$ in the AR-quiver of the 2-cluster category using the labels by the 2-arcs as in Figure~\ref{fig AR 2cluster BM}.
\end{proof}

The group $G$ acts on the set of 2-arcs in $\calp$ by the formula
$\zs \tilde \zg=\tilde \zg$ if $\tilde \zg$ is not a loop, and $\zs (i,i)^\pm =(i,i)^\mp$ if $\tilde \zg=(i,i)^\pm$ is a loop. 
If $\za$ is a 2-arc in $\calp$ or a 2-diagonal in $\cals$ we let $G\za=\{\za,\zs\za\}$ denote the $G$-orbit of  $\za$. The following result describes the relation between the intersection numbers in $\calp$ and $\cals$.

\begin{prop}\label{prop intersection number}
 Let $\tilde \zg,\tilde \zd$ be two 2-arcs in $\calp$ with lifts $\zg,\zd$ in $\cals$. Then
 \begin{equation}
\label{eq intersection number}
 \sum_{\tilde\zg'\in G\tilde \zg} e(\tilde \zg',\tilde\zd) = \sum_{\zd'\in G\zd} e(\zg, \zd').
\end{equation}
That is

(a) $e(\tilde\zg,\tilde\zd) = e(\zg,\zd)+ e(\zg,\zs \zd)$ if  $\tilde \zg $ is not a loop and $\tilde \zd$ is not a loop. 

(b) $e(\tilde\zg,\tilde\zd) = e(\zg,\zd)$ if  $\tilde \zg $ is not a loop and $\tilde \zd$ is a loop.

(c) $e(\tilde\zg,\tilde\zd) +e(\zs \tilde\zg,\tilde\zd)  = e(\zg,\zd)+ e(\zg,\zs \zd)$ if  $\tilde \zg $ is  a loop and $\tilde \zd$ is not a loop.  

(d) $e(\tilde\zg,\tilde\zd) +e(\zs \tilde\zg,\tilde\zd)  = e(\zg,\zd)$ if  $\tilde \zg $ is  a loop and $\tilde \zd$ is  a loop.  
\end{prop}
\begin{proof} First note that if $\tilde\zd$ is a loop in $\calp$ then its lift $\zd$ is a diameter in $\cals$ and hence $G\zd=\{\zd\}$. Therefore, equation (\ref{eq intersection number}) is equivalent to the four statements (a)-(d).
Parts (a) and (b) follow directly from the fact that the $2N$-gon $\cals$ is a 2-fold branched cover of the punctured $N$-gon $\calp$ under the action of the rotation $\zs$ by angle $\pi$.  

To show (c), suppose that $\tilde\zg$ is a loop in $\calp$. Thus $\zg$ is a diameter in $\cals$. If $\tilde \zd$ is not a loop, then $\tilde \zg$ crosses $\tilde\zd$ if and only if $\zg$ crosses $\zd$. Moreover, in that case, $\zg $ also crosses $\zs\zd$. Thus $e(\tilde \zg,\tilde\zd) = e(\zg,\zd)=e(\zg,\zs\zd)$. Since $\zs\tilde\zg$ is the same loop as $\tilde \zg$ but with the opposite sign, the same equation holds when replacing $\tilde \zg$ with $\zs\tilde\zg$.  The two equations together prove part (c). 

Finally, assume both $\tilde\zg,\tilde \zd$ are loops. 
Then the lifts $\zg,\zd$ are both diameters. Thus $\zg$ and $\zd$ cross unless they are equal. The crossing number of $\tilde \zg$ and $\tilde\zd$ depends on the signs of the loops. As we see from Example~\ref{ex62}, $\tilde\zg$ will cross exactly one of $\tilde\zd,\zs\tilde \zd$, unless $\zg=\zd$, and $\zs\tilde\zg$ will cross exactly the other. This shows part (d).
\end{proof}

We also obtain a dual statement of Proposition~\ref{prop intersection number}.

\begin{prop}\label{prop intersection number2}
 Let $\tilde \zg,\tilde \zd$ be two 2-arcs in $\calp$ with lifts $\zg,\zd$ in $\cals$. Then
 \begin{equation}
\label{eq intersection number2}
 \sum_{\tilde\zd'\in G\tilde \zd} e(\tilde \zg,\tilde\zd') = \sum_{\zg'\in G\zg} e(\zg', \zd).
\end{equation}
\end{prop}

\begin{proof}
Consider the following sequence of equations.
\begin{align*}
e(\tilde\zg, \tilde\zd) & =\dim \textup{Ext}^1_{\calc^2_{\mathbb{D}}} (\BM_\mathbb{D}(\zgg),\BM_\mathbb{D}(\zdd)) \\
&= \dim \textup{Ext}^2_{\calc^2_{\mathbb{D}}} (\BM_\mathbb{D}(\zdd), \BM_\mathbb{D}(\zgg))\\
&=\dim \textup{Ext}^1_{\calc^2_{\mathbb{D}}} (\BM_\mathbb{D}(R\zdd), \BM_\mathbb{D}(\zgg))\\
&= e(R\zdd, \zgg)
\end{align*}
The first and the last equality follow from Lemma~\ref{lem crossing ext}, while the second equality follows from the 3-Calabi-Yau property of $\calc^2_{\mathbb{D}}$.  Finally, the third equality holds since $\BM_\mathbb{D}$ is an equivalence of categories where the rotation $R$ in $\calp$ corresponds to an inverse shift in  $\calc^2_{\mathbb{D}}$.  Now the result follows from equation~(\ref{eq intersection number}) by replacing $R\zdd, \zgg$ with $\zgg, \zdd$ respectively. 
\end{proof}

\subsection{Combinatorial interpretation of the equivalence $\Psi$.}  We are now ready for our main result.

First, we define the rotation operation $R$ on the category $\arcp$ where the punctured polygon $\calp$ has size $N$, which is the analog of the clockwise rotation $R$ on $\diag (\cals)$.  Let $\tilde\zg$ be a 2-arc in $\calp$.  If $\tilde\zg=(i,j)$ is not a loop then define $R\tilde\zg = (i+1, j+1)$.  If $\tilde\zg=(i,i)^\pm$ is a loop then define $R\zg=(i+1,i+1)^\pm$ whenever $(N+1)/2$ is even and otherwise let $R\zg=(i+1,i+1)^\mp$.  Moreover, we can extend the definition of $R$ to 2-pivots $\tilde\zg \to \tilde \zd$ by setting $R(\tilde\zg \to \tilde \zd) = R \tilde\zg \to R\tilde \zd$.

Let $\tilde\rho(i)\in\arcp$ denote the radical line of vertex $i$.
For every 2-arc $\tilde \zg$ in $\calp$ we define projective modules 
\[\tilde P_1(\tilde \zg)=\bigoplus_{j} e(\tilde\rho(j),\tilde\zg)\, \tilde P(j)\qquad \textup{and}\qquad
 \tilde P_0(\tilde \zg)=\bigoplus_{i} e(\tilde\zg,\tilde\rho(i))\, \tilde P(i) .\]

\begin{thm}
 \label{thm checkerboard}
Let $B=\overline e AG \overline e$ be the basic algebra of the skew group algebra of a dimer tree algebra $A$ with respect to the action of a group $G$ of order 2. Let $\calp$ be the associated punctured polygon with checkerboard pattern. 
For each 2-arc $\tilde \zg$ in $\calp$ there exists a morphism $\tilde f_{\tilde \zg}\colon \tilde P_1(\tilde \zg)\to \tilde P_0(\tilde \zg)$ such that the mapping $\tilde \zg\mapsto \coker \tilde f_{\tilde \zg}$ induces an equivalence of categories
\[\Psi\colon\arcp\to\scmp \,B.\]
Under this equivalence, the radical line $\tilde \rho(i)$ corresponds to the radical of the indecomposable projective $\tilde P(i)$ for all vertices $i\in Q_B$. The clockwise rotation $R$ of $\calp$ corresponds to the shift $\zO$ in $\scmp\,B$ and $R^2 $ corresponds to the inverse Auslander-Reiten translation $\tau^{-1}=\zO^2$. Thus
\[
\begin{array}{rcl}
\Psi (\tilde \rho(i))  & =    & \rad \tilde P(i)   \\
 \Psi \circ R & =  & \zO\circ \Psi   \\
 \Psi\circ R^2& =  & \tau^{-1}\circ \Psi  
\end{array}
\]
Furthermore, $\Psi$ maps the 2-pivots in $\arcp$ to the irreducible morphisms in $\scmp\,B$, and the meshes in $\arcp$ to the Auslander-Reiten triangles in $\scmp\,B$.
\end{thm}

\begin{proof}
This result follows from Theorem~\ref{thm SS4} and properties of the induction functor $F$. The main part of the proof below is to check that our definition of the intersection numbers is correct. 

Let $\tilde\zg$ be a 2-arc in $\calp$ and let $\zg$ be a lift in $\cals$. 
 Theorem~\ref{thm SS4} shows the existence of a morphism $f_\zg\colon
 P_1(\zg) \to P_0(\zg) $
that defines a functor $\Phi\colon\diags \to\scmp \,A,\  \Phi(\zg)= \coker\, M_\zg$. The projective modules above are defined as 
\[P_0(\zg)=\bigoplus_i e(\zg, \rho(i))\, P(i) \quad\textup{ and }\quad P_1(\zg)=\bigoplus_j e(\rho(j),\zg)\, P(j),\]
where both sums run over all vertices of the quiver of $A$. 
We can rewrite $P_0(\zg)$ by grouping its summands according to the $G$-orbits  on the vertex set of $Q_A$ as follows. Let $I$ be a set of representatives of the $G$-orbits in the vertex set  of $Q_A$, and let us write $I=I_1\sqcup I_2$, where $I_1$ contains the vertices that are fixed by $\zs$ and $I_2$ contains the vertices that are not fixed by $\zs$. In the polygon $\cals$, the set $I_1$ corresponds to  the diameters and $I_2$ to the non-diameters. Then
\[
 P_0(\zg) = \oplus_{i\in I_1}   e(\zg, \rho(i))\, P(i)  \ \oplus\  \oplus_{i\in I_2} \Big( e(\zg, \rho(i))\, P(i) \oplus e(\zg, \rho(\zs i))\, P(\zs i)\Big).
 \]

We now apply the induction functor $F$.  Proposition~\ref{prop RR} yields $F(P(i))=\tilde P(\tilde i^+)\oplus \tilde P(\tilde i^-)$ if $i\in I_1$ and   $F(P(i))=F(P(\zs i))=\tilde P(\tilde i)$ if $i\in I_2$, where $\tilde i^\pm$ are the two vertices of $Q_B$ that correspond to the two copies $\tilde\rho(i^\pm)$  of a loop  in $\arcp$. Thus we obtain
\[ F(P_0(\zg)) = \oplus_{i\in I_1} \,  e(\zg, \rho(i))\, \Big( \tilde P(\tilde i ^+) \oplus \tilde P(\tilde i ^-) \Big)\ \oplus\  \oplus_{i\in I_2} \,\Big(e(\zg, \rho(i))+ e(\zg, \rho(\zs i))\Big)\, \tilde P(\tilde i). \]

Note that the coefficient $ e(\zg, \rho(i))$ in the first direct sum is equal to  $\sum_{i'\in Gi}  e(\zg, \rho(i'))$, since $\zs i=i$ in this case, and the coefficient $ e(\zg, \rho(i))+ e(\zg, \rho(\zs i))$ in the second direct sum is equal to  $\sum_{i'\in Gi}  e(\zg, \rho(i'))$. Therefore  equation~(\ref{eq intersection number}) implies
\begin{equation}\label{eq 63} F(P_0(\zg)) = \oplus_{\tilde i\in \tilde I_1}\,   \sum_{\tilde \zg'\in G\tilde \zg}e(\tilde \zg', \tilde \rho(\tilde i))\, \Big( \tilde P(\tilde i ^+) \oplus \tilde P(\tilde i ^-) \Big)\ \oplus\  \oplus_{\tilde i\in \tilde I_2} \sum_{\tilde\zg'\in G\tilde\zg} \,e(\tilde\zg', \tilde\rho(\tilde i))\, \tilde P(\tilde i), \end{equation}
 where $\tilde I_1\sqcup \tilde I_2$ is a complete set of representatives of the $G$ orbits in the vertex set of $Q_B$ such that $\tilde I_1$ contains the vertices that are not fixed by $\sigma$ and $\tilde I_2$ contains the vertices that are fixed by $\sigma$. 

 
 Suppose first that $\tilde\zg$ is not a loop. Then its orbit $G\tilde \zg$ only contains $\tilde \zg$ and thus  equation~(\ref{eq 63}) becomes 
 \[ F(P_0(\zg))=
 \oplus_{\tilde i\in \tilde I_1}\, 
 e(\tilde \zg, \tilde \rho(\tilde i))\, \Big( \tilde P(\tilde i ^+) \oplus \tilde P(\tilde i ^-) \Big)\ \oplus\  \oplus_{\tilde i\in \tilde I_2}  \,e(\tilde\zg, \tilde\rho(\tilde i))\, \tilde P(\tilde i),
 \]
 which can be written as
\begin{equation}\label{eq 631}F(P_0(\zg))=  \oplus_{i} \,e(\tilde\zg,\tilde\rho( \tilde i))\, \tilde P(\tilde i) 
 =\tilde P_0(\tilde\zg),\end{equation}
  where the sum is over all vertices of the quiver of $B$.
  
 On the other hand, if $\tilde \zg=(h,h)$ is a loop, then $G\tilde \zg$ contains two arcs $(h,h)^+$ and $(h,h)^-$. Then
equation~(\ref{eq 63}) becomes 
 \[ F(P_0(\zg))=
 \oplus_{\tilde i\in \tilde I_1}\, 
 \Big(e((h,h)^+, \tilde \rho(\tilde i))  
 +e((h,h)^-, \tilde \rho(\tilde i)) \Big)  
 \, \Big( \tilde P(\tilde i ^+) \oplus \tilde P(\tilde i ^-) \Big)\ \oplus\  \oplus_{\tilde i\in \tilde I_2}  \,2e((h,h), \tilde\rho(\tilde i))\, \tilde P(\tilde i),
 \]
which yields
\begin{equation}\label{eq 632}F(P_0(\zg))= \tilde P_0((h,h)^+)\oplus
\tilde P_0((h,h)^-).\end{equation}

The analogous computation for $P_1(\zg)$  together with equation~(\ref{eq intersection number2}) yields
\begin{equation}\label{eq 633} F(P_1(\zg)) =\left\{
\begin{array}{cc}
 \tilde P_1(\tilde\zg) &\textup{if $\tilde \zg $ is not a loop};\\
  \tilde P_1((h,h)^+)\oplus
\tilde P_1((h,h)^-) &\textup{if $\tilde \zg =(h,h)$ is  a loop.}\\
\end{array}
\right. 
\end{equation}

To complete the proof of the theorem, we define $\tilde f_{\tilde \zg} = F(f_\zg)$. Then $\coker \tilde f_{\tilde\zg}= F(\coker f_\zg)$, since $F$ is an exact functor. In particular, $\coker \tilde f_{\tilde \zg}$ is a non-projective syzygy, by Proposition~\ref{prop Fsyz} (a), and
the projective presentation in Theorem~\ref{thm SS4} is mapped under $F$ to a projective presentation of $\coker \tilde f_{\tilde \zg}$.  Therefore the functor $\Psi $ is well-defined. 
The statement that $\Psi$ is an equivalence of categories follows from Theorem~\ref{thm Fequiv}, using the fact that the algebras $AG$ and $B$ are Morita equivalent. 

It remains to prove the three equations in the statement of the theorem.  The first equation follows from the corresponding equation in Theorem~\ref{thm SS4}, because the exact  functor $F$ maps the radical  of the projective $P(i)$ to the  radical   of the projective $F(P(i))$. The last two equations  follow from the corresponding equations in Theorem~\ref{thm SS4} together with 
 Proposition~\ref{prop Fsyz}, which says that the induction functor $F$ commutes with the syzygy operator $\zO$.  
 \end{proof}

\subsection{Applications}

We immediately obtain the following consequences of our results. The statements appearing below are analogous to the ones shown for dimer tree algebras in Corollary~\ref{cor typeA}.

We begin by defining the weight of the algebra $B$. Recall that the quiver $Q_B$ of $B$ is obtained from a quiver $Q^0$ of a dimer tree algebra by replacing a boundary arrow $\alpha\colon i\to j$ in $Q^0$ with two arrows $\alpha^-\colon  i^-\to j^-$ and $\alpha^+\colon  i^+\to j^+$ and each arrow $\beta$ in $Q^0$ starting or ending at $i$ or $j$ by two arrows $\beta^+$ and $\beta^-$ starting or ending at $i^+, i^-$ or $j^+, j^-$.

\begin{definition}\label{def 66}
Let $B$ be a skew group algebra of a dimer tree algebra.

 (a) For every boundary arrow $\zb\not=\alpha$ in $Q^0$, we define its unique \emph{skew cycle path} to be
 $\tilde{\mathfrak{c}}(\zb)=\zb_1\zb_2\cdots\zb_{\ell(\zb)}$ in $Q^0$, where
 \begin{itemize}
\item [(i)] $\zb_1=\zb$ and $\zb_{\ell(\zb)}$  are boundary arrows in $Q^0$ different from $\alpha$, and each of 
 $\zb_2,\ldots,\zb_{\ell(\zb)-1}$ is an interior arrow or the arrow $\alpha$,
\item [(ii)] every subpath of length two $\zb_i\zb_{i+1}$, 
 is a subpath of a chordless cycle $C_i$ in $Q^0$, and $C_i= C_j$ if and only if $j=i+1$ and $\zb_{i+1}=\alpha$.  
\end{itemize}

 (b) The \emph{weight} $\tilde{\text{w}}(\zb)$ of the boundary arrow $\zb$ is defined as
\[\tilde{\textup{w}}(\zb) =\left\{ 
\begin{array}
 {ll} 1&\textup{if the length of  $\tilde{\mathfrak{c}}(\zb)$ is odd;}\\
 2&\textup{if the length of  $\tilde{\mathfrak{c}}(\zb)$ is even.}\\
\end{array} \right.\]

(c) The \emph{total weight} of $B$ is defined as $\sum_\zb \tilde{\textup{w}}(\zb) $,
where the sum is over all boundary arrows  of $Q^0$ different from $\alpha$.
\end{definition}

Note that the skew cycle path $\tilde{\mathfrak{c}}(\zb)$ coincides with the cycle path ${\mathfrak{c}}(\zb)$ unless $\tilde{\mathfrak{c}}(\zb)$ ends in the arrow $\alpha$ in which case $\tilde{\mathfrak{c}}(\zb)={\mathfrak{c}}(\zb){\mathfrak{c}}(\za)$.  Recall that cycle paths for dimer tree algebras correspond to moving clockwise around the unshaded regions of the checkerboard polygon.  Similarly, the skew cycle paths correspond to moving clockwise around the unshaded regions in the punctured polygon. Moreover, the weight $\tilde{\text{w}}(\zb)$ equals the weight of the corresponding arrow in the quiver $Q_A$.  This implies that the total weight of the algebra $B$ is half the total weight of $A$.

\begin{corollary}\label{cor 67}
(a) The category $\scmp\, B$ is equivalent to the 2-cluster category of type $\mathbb{D}_{(N+1)/2}.$ In particular, the number of indecomposable syzygies is $N(N+1)/2$.

(b) The total weight of $B$ is a derived invariant. 

(c) The projective resolution of any syzygy is periodic of period $N$ or $2N$. An indecomposable syzygy $M_{\tilde\zg}$ has period $2N$ if and only if the corresponding 2-diagonal $\tilde\zg$ is a loop in $\calp$ and $(N+1)/2$ is odd.

(d) The indecomposable syzygies over $B$ are rigid $B$-modules.

(e)  Let $L,M$ be indecomposable syzygies over $B$. Then the dimension of $\Ext^1_B(L,M)\oplus \Ext^1_B(M,L)$ is equal to the sum of the two crossing numbers between the corresponding 2-arcs. In particular, the dimension is either 0, 1, or 2.

(f) 
Let $\tau$ denote the Auslander-Reiten translation in $\textup{mod}\,B$ and $\nu$ the Nakayama functor. 
 We denote the stable cosyzygy category by $\underline{\textup{CMI}}\,B$.
The following diagram commutes. 
\[\xymatrix@C40pt{
\scmp\,B\ar[r]^{\tau}&\underline{\textup{CMI}}\,B\\
\arcp\ar[u]^{\textup{cok}\,\tilde f_{\tilde\zg}} \ar[r]^{1}
&\arcp \ar[u]_{\textup{ker}\,\nu\! \tilde f_{\tilde\zg}}  
}
\]
\end{corollary}

\begin{proof}
(a) Since $B$ is the basic algebra of $AG$ then the  two categories are equivalent by Theorem~\ref{thm equivalences}.  
The Auslander-Reiten quiver of  $\calc^2_{\mathbb{D}_{(N+1)/2}}$ has  $(N+1)/2$  $\tau$-orbits each of cardinality $N$. Thus there are $N(N+1)/2$ indecomposable objects.

(b) The weight of $B$ equals $N$, the size of the punctured polygon. By part (a), the integer $N$ determines the syzygy category of $B$ up to equivalence.  
Now the statement follows because derived equivalent algebras have equivalent singularity categories.

(c) By definition of $R$ if a 2-arc $\tilde\zg$ is not a loop then $R^N \tilde\zg=\tilde\zg$. Moreover, if $\tilde\zg=(i,i)^\pm$ is a loop then $R^{N}\tilde\zg=\tilde\zg$ whenever $(N+1)/2$ even, and since $N$ is odd then $R^N\tilde\zg=(i,i)^\mp$ whenever  $(N+1)/2$ is odd.  Hence, only in the latter case the period of $\tilde\zg$ is $2N$ under $R$, while in all other cases it is $N$.  By Theorem~\ref{thm checkerboard} the rotation $R$ on the 2-arcs corresponds to the syzygy functor $\Omega$ on $\scmp\,B$, which shows (c).

(d) Let $M$ be an indecomposable syzygy over $B$. Under the equivalence of categories in Theorem~\ref{thm equivalences}, $M$ corresponds to a 2-arc in $\arcp$ which, because of Proposition~\ref{lem crossing ext}, is a rigid object in $\scmp\,B$. Thus $M$ is rigid in $\text{mod}\,B$, as $\scmp\,B$ is an extension closed subcategory of $\text{mod}\,B$ by \cite[Proposition 2.3]{SS3}. 

(e) This follows from the equivalences in Theorem~\ref{thm equivalences} and Proposition~\ref{lem crossing ext}, together with the  fact that $\scmp\,B$ is an extension closed subcategory of $\text{mod}\,B$ as above. 

(f) The proof here is analogous to the one for Corollary~\ref{cor typeA}(f) given in \cite{SS3}.
\end{proof}

A module $M\in\text{mod}\,A$ is said to be $\tau$-rigid if $\text{Hom}_A(M, \tau M)=0$.  Note that a $\tau$-rigid module is rigid, and we know that indecomposable syzygies over algebras of finite Cohen-Macaulay type are rigid.  Moreover, in \cite{SS3} we conjecture that indecomposable syzygies over dimer tree algebras are $\tau$-rigid and reachable.  In particular, this would imply that they correspond to cluster variables in the cluster algebra of $Q$.  It is natural to ask how this property behaves by passing to the skew group algebra of a dimer tree algebra, thus we make the following observation. 

\begin{prop}\label{prop 68}
The module $M\oplus \sigma M$ is $\tau$-rigid in $\textup{mod}\,A$ if and only if the induced module $FM$ is $\tau$-rigid in $\textup{mod}\,AG$. 
\end{prop}

\begin{proof}
Consider the following
\[\Hom_{AG}(FM,\tau FM)\cong\Hom_{AG}(FM,F (\tau M))\cong \Hom_A(M,HF(\tau M))\cong \Hom_A(M,\tau M\oplus \sigma \tau M)\] 
where the first isomorphism follows from Proposition~\ref{prop RR2}, the second isomorphism from the adjointness of $H$ and $F$, and the third from Proposition~\ref{prop RR}.  Since $\sigma$ is an automorphism of $A$ of order 2, we also get that the right hand side is isomorphic to $\Hom_A(\sigma M,\sigma\tau M\oplus \tau M)$.  This shows the proposition. 
\end{proof}

Note that for an indecomposable syzygy $M$, we know from the geometric model that $M\oplus \sigma M$ is rigid, so again it is natural to make the following conjecture. 

\begin{conjecture}
If $M$ is an indecomposable syzygy then  $M\oplus \sigma M$ is $\tau$-rigid.
\end{conjecture}

\section{Examples}\label{sect examples}
\subsection{Example with detailed computations}\label{ex71}
The left picture below shows the checkerboard polygon $\cals$ as well as the quiver of the corresponding dimer tree algebra $A$. The right picture shows the checkerboard punctured polygon $\calp$ and the quiver of the skew group algebra $B=\overline{e}\,AG\,\overline{e}$. The orientations of the radical lines are indicated by arrows in both pictures. 

In the left picture there are two red 2-diagonals labeled $\zg$ and $\zs\zg$. The 2-diagonal $\zg$ crosses the radical lines $\rho(4)$ and $\rho(5')$ from right to left and $\zg$ crosses $\rho(3')$ from left to right. Thus the corresponding syzygy $\Phi(\zg)\in \scmp\,A$ is the cokernel of a map $P(4)\oplus P(5')\to P(3')$. Hence $\Phi(\zg)=\begin{smallmatrix}3'\\1\,\end{smallmatrix}$.

On the right hand side, the 2-arc $\tilde \zg$ crosses the radical lines $\tilde \rho(4)$ and $\tilde \rho(5)$ from right to left and $\tilde \zg$ crosses $\tilde \rho(3)$ from left to right. Thus the corresponding syzygy $\Psi(\tilde \zg)\in \scmp\,B$ is the cokernel of a map $\tilde P(4)\oplus \tilde P(5)\to \tilde P(3)$.  Hence $\Psi(\tilde \zg)=\begin{smallmatrix}3\\1^{\scalebox{.5}{$+$}}\ 1^{\scalebox{.5}{$-$}}\end{smallmatrix}$.

\begin{center}
 \scalebox{0.67}{\Large
\begingroup%
  \makeatletter%
  \providecommand\color[2][]{%
    \errmessage{(Inkscape) Color is used for the text in Inkscape, but the package 'color.sty' is not loaded}%
    \renewcommand\color[2][]{}%
  }%
  \providecommand\transparent[1]{%
    \errmessage{(Inkscape) Transparency is used (non-zero) for the text in Inkscape, but the package 'transparent.sty' is not loaded}%
    \renewcommand\transparent[1]{}%
  }%
  \providecommand\rotatebox[2]{#2}%
  \newcommand*\fsize{\dimexpr\f@size pt\relax}%
  \newcommand*\lineheight[1]{\fontsize{\fsize}{#1\fsize}\selectfont}%
  \ifx\svgwidth\undefined%
    \setlength{\unitlength}{579.39258808bp}%
    \ifx\svgscale\undefined%
      \relax%
    \else%
      \setlength{\unitlength}{\unitlength * \real{\svgscale}}%
    \fi%
  \else%
    \setlength{\unitlength}{\svgwidth}%
  \fi%
  \global\let\svgwidth\undefined%
  \global\let\svgscale\undefined%
  \makeatother%
  \begin{picture}(1,0.44748669)%
    \lineheight{1}%
    \setlength\tabcolsep{0pt}%
    \put(0,0){\includegraphics[width=\unitlength,page=1]{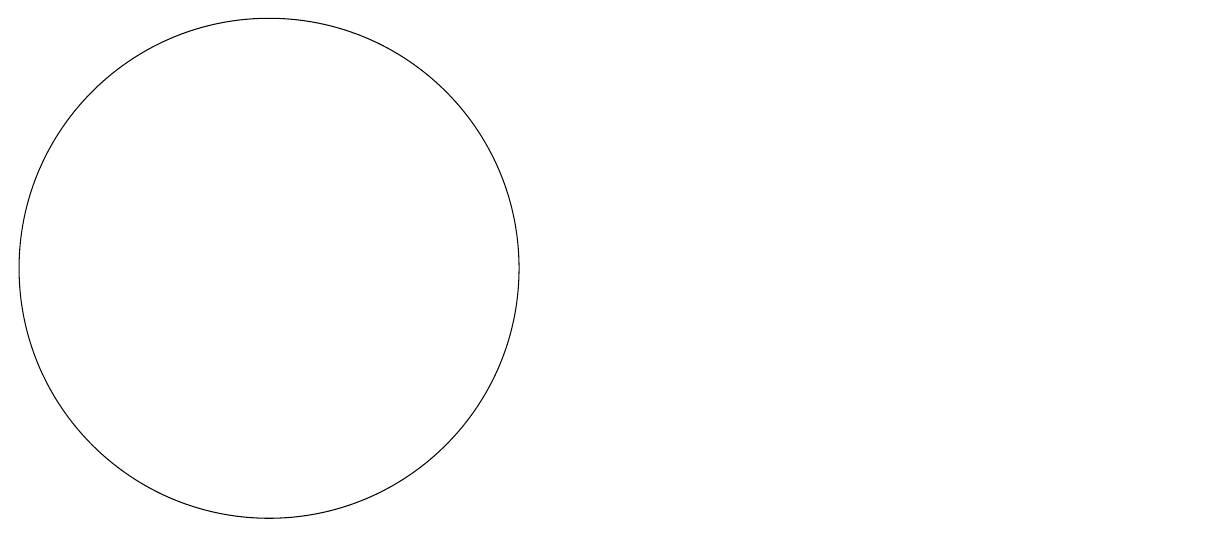}}%
    \put(-0.00234283,0.22264692){\makebox(0,0)[lt]{\lineheight{1.25}\smash{\begin{tabular}[t]{l}1\end{tabular}}}}%
    \put(0,0){\includegraphics[width=\unitlength,page=2]{example14.pdf}}%
    \put(0.01836852,0.31843689){\makebox(0,0)[lt]{\lineheight{1.25}\smash{\begin{tabular}[t]{l}2\end{tabular}}}}%
    \put(0.08050255,0.39351552){\makebox(0,0)[lt]{\lineheight{1.25}\smash{\begin{tabular}[t]{l}3\end{tabular}}}}%
    \put(0.16852573,0.43493821){\makebox(0,0)[lt]{\lineheight{1.25}\smash{\begin{tabular}[t]{l}4\end{tabular}}}}%
    \put(0.26949358,0.43234929){\makebox(0,0)[lt]{\lineheight{1.25}\smash{\begin{tabular}[t]{l}5\end{tabular}}}}%
    \put(0.354928,0.3909266){\makebox(0,0)[lt]{\lineheight{1.25}\smash{\begin{tabular}[t]{l}6\end{tabular}}}}%
    \put(0.41706191,0.31843689){\makebox(0,0)[lt]{\lineheight{1.25}\smash{\begin{tabular}[t]{l}7\end{tabular}}}}%
    \put(0.44295122,0.22005799){\makebox(0,0)[lt]{\lineheight{1.25}\smash{\begin{tabular}[t]{l}8\end{tabular}}}}%
    \put(0.41965086,0.12944585){\makebox(0,0)[lt]{\lineheight{1.25}\smash{\begin{tabular}[t]{l}9\end{tabular}}}}%
    \put(0.35751692,0.04918942){\makebox(0,0)[lt]{\lineheight{1.25}\smash{\begin{tabular}[t]{l}10\end{tabular}}}}%
    \put(0.2720825,0.00258892){\makebox(0,0)[lt]{\lineheight{1.25}\smash{\begin{tabular}[t]{l}11\end{tabular}}}}%
    \put(0.16334789,-0){\makebox(0,0)[lt]{\lineheight{1.25}\smash{\begin{tabular}[t]{l}12\end{tabular}}}}%
    \put(0.06755795,0.04401161){\makebox(0,0)[lt]{\lineheight{1.25}\smash{\begin{tabular}[t]{l}13\end{tabular}}}}%
    \put(0.00801283,0.12167912){\makebox(0,0)[lt]{\lineheight{1.25}\smash{\begin{tabular}[t]{l}14\end{tabular}}}}%
    \put(0,0){\includegraphics[width=\unitlength,page=3]{example14.pdf}}%
    \put(0.17003284,0.262851){\makebox(0,0)[lt]{\lineheight{1.25}\smash{\begin{tabular}[t]{l}1\end{tabular}}}}%
    \put(0.26796409,0.26492757){\makebox(0,0)[lt]{\lineheight{1.25}\smash{\begin{tabular}[t]{l}2\end{tabular}}}}%
    \put(0.21499795,0.32081504){\makebox(0,0)[lt]{\lineheight{1.25}\smash{\begin{tabular}[t]{l}3\end{tabular}}}}%
    \put(0.15306062,0.38190751){\makebox(0,0)[lt]{\lineheight{1.25}\smash{\begin{tabular}[t]{l}4\end{tabular}}}}%
    \put(0.27937165,0.38329189){\makebox(0,0)[lt]{\lineheight{1.25}\smash{\begin{tabular}[t]{l}5\end{tabular}}}}%
    \put(0.09092659,0.33530699){\makebox(0,0)[lt]{\lineheight{1.25}\smash{\begin{tabular}[t]{l}4\end{tabular}}}}%
    \put(0.07021524,0.28352862){\makebox(0,0)[lt]{\lineheight{1.25}\smash{\begin{tabular}[t]{l}4\end{tabular}}}}%
    \put(0.11403014,0.29751478){\makebox(0,0)[lt]{\lineheight{1.25}\smash{\begin{tabular}[t]{l}3\end{tabular}}}}%
    \put(0.06225178,0.32081504){\makebox(0,0)[lt]{\lineheight{1.25}\smash{\begin{tabular}[t]{l}3\end{tabular}}}}%
    \put(0.32114363,0.29751478){\makebox(0,0)[lt]{\lineheight{1.25}\smash{\begin{tabular}[t]{l}3\end{tabular}}}}%
    \put(0.38068873,0.30010368){\makebox(0,0)[lt]{\lineheight{1.25}\smash{\begin{tabular}[t]{l}3\end{tabular}}}}%
    \put(0.29903116,0.3374173){\makebox(0,0)[lt]{\lineheight{1.25}\smash{\begin{tabular}[t]{l}2\end{tabular}}}}%
    \put(0.33009822,0.37366219){\makebox(0,0)[lt]{\lineheight{1.25}\smash{\begin{tabular}[t]{l}2\end{tabular}}}}%
    \put(0.1721741,0.17690437){\makebox(0,0)[lt]{\lineheight{1.25}\smash{\begin{tabular}[t]{l}2\end{tabular}}}}%
    \put(0.13851814,0.1044147){\makebox(0,0)[lt]{\lineheight{1.25}\smash{\begin{tabular}[t]{l}2\end{tabular}}}}%
    \put(0.11004007,0.06816985){\makebox(0,0)[lt]{\lineheight{1.25}\smash{\begin{tabular}[t]{l}2\end{tabular}}}}%
    \put(0.1363769,0.33792964){\makebox(0,0)[lt]{\lineheight{1.25}\smash{\begin{tabular}[t]{l}1\end{tabular}}}}%
    \put(0.09236528,0.36381883){\makebox(0,0)[lt]{\lineheight{1.25}\smash{\begin{tabular}[t]{l}1\end{tabular}}}}%
    \put(0.26582294,0.17741671){\makebox(0,0)[lt]{\lineheight{1.25}\smash{\begin{tabular}[t]{l}1\end{tabular}}}}%
    \put(0.29947893,0.10233812){\makebox(0,0)[lt]{\lineheight{1.25}\smash{\begin{tabular}[t]{l}1\end{tabular}}}}%
    \put(0.34090156,0.07903783){\makebox(0,0)[lt]{\lineheight{1.25}\smash{\begin{tabular}[t]{l}1\end{tabular}}}}%
    \put(0.03397041,0.16961613){\makebox(0,0)[lt]{\lineheight{1.25}\smash{\begin{tabular}[t]{l}5'\end{tabular}}}}%
    \put(0.10904904,0.11007114){\makebox(0,0)[lt]{\lineheight{1.25}\smash{\begin{tabular}[t]{l}5'\end{tabular}}}}%
    \put(0.15306062,0.05570384){\makebox(0,0)[lt]{\lineheight{1.25}\smash{\begin{tabular}[t]{l}5'\end{tabular}}}}%
    \put(0.27937165,0.05449932){\makebox(0,0)[lt]{\lineheight{1.25}\smash{\begin{tabular}[t]{l}4'\end{tabular}}}}%
    \put(0.32856112,0.32892459){\makebox(0,0)[lt]{\lineheight{1.25}\smash{\begin{tabular}[t]{l}5\end{tabular}}}}%
    \put(0.34409465,0.10368874){\makebox(0,0)[lt]{\lineheight{1.25}\smash{\begin{tabular}[t]{l}4'\end{tabular}}}}%
    \put(0.36739493,0.16323391){\makebox(0,0)[lt]{\lineheight{1.25}\smash{\begin{tabular}[t]{l}4'\end{tabular}}}}%
    \put(0.39069522,0.2745575){\makebox(0,0)[lt]{\lineheight{1.25}\smash{\begin{tabular}[t]{l}5\end{tabular}}}}%
    \put(0.21499795,0.11887946){\makebox(0,0)[lt]{\lineheight{1.25}\smash{\begin{tabular}[t]{l}3'\end{tabular}}}}%
    \put(0.0493072,0.13959082){\makebox(0,0)[lt]{\lineheight{1.25}\smash{\begin{tabular}[t]{l}3'\end{tabular}}}}%
    \put(0.37292196,0.11887946){\makebox(0,0)[lt]{\lineheight{1.25}\smash{\begin{tabular}[t]{l}3'\end{tabular}}}}%
    \put(0.11403014,0.13959078){\makebox(0,0)[lt]{\lineheight{1.25}\smash{\begin{tabular}[t]{l}3'\end{tabular}}}}%
    \put(0.32114363,0.13959078){\makebox(0,0)[lt]{\lineheight{1.25}\smash{\begin{tabular}[t]{l}3'\end{tabular}}}}%
    \put(0,0){\includegraphics[width=\unitlength,page=4]{example14.pdf}}%
    \put(0.56721904,0.22264692){\makebox(0,0)[lt]{\lineheight{1.25}\smash{\begin{tabular}[t]{l}1\end{tabular}}}}%
    \put(0,0){\includegraphics[width=\unitlength,page=5]{example14.pdf}}%
    \put(0.65006447,0.39351552){\makebox(0,0)[lt]{\lineheight{1.25}\smash{\begin{tabular}[t]{l}2\end{tabular}}}}%
    \put(0.83905564,0.43234929){\makebox(0,0)[lt]{\lineheight{1.25}\smash{\begin{tabular}[t]{l}3\end{tabular}}}}%
    \put(0.98662423,0.31843689){\makebox(0,0)[lt]{\lineheight{1.25}\smash{\begin{tabular}[t]{l}4\end{tabular}}}}%
    \put(0.98921319,0.12944585){\makebox(0,0)[lt]{\lineheight{1.25}\smash{\begin{tabular}[t]{l}5\end{tabular}}}}%
    \put(0.63711983,0.04401161){\makebox(0,0)[lt]{\lineheight{1.25}\smash{\begin{tabular}[t]{l}7\end{tabular}}}}%
    \put(0.8446378,0.00680815){\makebox(0,0)[lt]{\lineheight{1.25}\smash{\begin{tabular}[t]{l}6\end{tabular}}}}%
    \put(0,0){\includegraphics[width=\unitlength,page=6]{example14.pdf}}%
    \put(0.70572632,0.36762634){\makebox(0,0)[lt]{\lineheight{1.25}\smash{\begin{tabular}[t]{l}3\end{tabular}}}}%
    \put(0.7639769,0.3029034){\makebox(0,0)[lt]{\lineheight{1.25}\smash{\begin{tabular}[t]{l}3\end{tabular}}}}%
    \put(0.90118958,0.22394138){\makebox(0,0)[lt]{\lineheight{1.25}\smash{\begin{tabular}[t]{l}3\end{tabular}}}}%
    \put(0.77562699,0.13332925){\makebox(0,0)[lt]{\lineheight{1.25}\smash{\begin{tabular}[t]{l}3\end{tabular}}}}%
    \put(0.70184303,0.0647229){\makebox(0,0)[lt]{\lineheight{1.25}\smash{\begin{tabular}[t]{l}3\end{tabular}}}}%
    \put(0.697959,0.29254774){\makebox(0,0)[lt]{\lineheight{1.25}\smash{\begin{tabular}[t]{l}4\end{tabular}}}}%
    \put(0.77821595,0.34820944){\makebox(0,0)[lt]{\lineheight{1.25}\smash{\begin{tabular}[t]{l}4\end{tabular}}}}%
    \put(0.88047822,0.32361473){\makebox(0,0)[lt]{\lineheight{1.25}\smash{\begin{tabular}[t]{l}4\end{tabular}}}}%
    \put(0.62288145,0.18899098){\makebox(0,0)[lt]{\lineheight{1.25}\smash{\begin{tabular}[t]{l}5\end{tabular}}}}%
    \put(0.78209931,0.08802319){\makebox(0,0)[lt]{\lineheight{1.25}\smash{\begin{tabular}[t]{l}5\end{tabular}}}}%
    \put(0.90507295,0.11391236){\makebox(0,0)[lt]{\lineheight{1.25}\smash{\begin{tabular}[t]{l}5\end{tabular}}}}%
    \put(0.80669419,0.39351551){\makebox(0,0)[lt]{\lineheight{1.25}\smash{\begin{tabular}[t]{l}$1^\pm$\end{tabular}}}}%
    \put(0.82222749,0.31843681){\makebox(0,0)[lt]{\lineheight{1.25}\smash{\begin{tabular}[t]{l}$1^\pm$\end{tabular}}}}%
    \put(0.7730387,0.2537138){\makebox(0,0)[lt]{\lineheight{1.25}\smash{\begin{tabular}[t]{l}$1^\pm$\end{tabular}}}}%
    \put(0.77562759,0.18122404){\makebox(0,0)[lt]{\lineheight{1.25}\smash{\begin{tabular}[t]{l}$2^\pm$\end{tabular}}}}%
    \put(0.82222861,0.11909001){\makebox(0,0)[lt]{\lineheight{1.25}\smash{\begin{tabular}[t]{l}$2^\pm$\end{tabular}}}}%
    \put(0.80928419,0.05436707){\makebox(0,0)[lt]{\lineheight{1.25}\smash{\begin{tabular}[t]{l}$2^\pm$\end{tabular}}}}%
    \put(0,0){\includegraphics[width=\unitlength,page=7]{example14.pdf}}%
    \put(0.69092185,0.16727323){\color[rgb]{1,0,0}\makebox(0,0)[lt]{\lineheight{1.25}\smash{\begin{tabular}[t]{l}$\tilde\zg$\end{tabular}}}}%
    \put(0.06742929,0.21306755){\color[rgb]{1,0,0}\makebox(0,0)[lt]{\lineheight{1.25}\smash{\begin{tabular}[t]{l}$\zg$\end{tabular}}}}%
    \put(0.35308366,0.22776823){\color[rgb]{1,0,0}\makebox(0,0)[lt]{\lineheight{1.25}\smash{\begin{tabular}[t]{l}$\zs\zg$\end{tabular}}}}%
  \end{picture}%
\endgroup%
}
\end{center}
\[\xymatrix{
4\ar[rr]&&3\ar[ld]\ar[rr]&&5\ar[ld]\\
&1\ar[lu]\ar[rr]^\za\ar[ld]&&2\ar[ld]\ar[lu]\\
4'\ar[rr]&&3'\ar[lu]\ar[rr]&&5'\ar[lu]\\
}
\qquad\qquad\qquad
\xymatrix{4\ar[rr]&&3\ar[ld]\ar[rr]\ar[ldd]&&5\ar[ld]\ar[ldd]\\
&1^+\ar[lu]\ar[rr]^{\za^+}&&2^+\ar[lu]\\
&1^-\ar[luu]\ar[rr]^{\za^-}&&2^-\ar[luu]\\} 
\]

To illustrate the correspondence between the projective resolutions over $A$ and $B$ we specify them below. Each morphism in these resolutions is of the form $f_\zg$ or $\tilde f_{\tilde \zg}$ for some 2-diagonal $\zg$ in $\cals$ or a 2-arc $\tilde \zg$ in $\calp$. The label on top of the arrows below indicate these $\zg$  and $\tilde \zg$. The projective resolutions are periodic and we illustrate one period. 

\subsubsection{Projective resolution of the radical of $P(5)$} \label{res5} 

Over the algebra $A$, the radical of $P(5)$ corresponds to the 2-diagonal $(5,8)$ in the 14-gon $\cals$. Over the algebra $B$, the radical of $P(5)$ corresponds to the 2-arc $(5,1)$ in the punctured 7-gon $\calp$. The projective resolutions are  given by the rotation orbit of $(5,8)$ in $\cals$ and of $(5,1)$ in $\calp$. In particular the period of the resolution is 14 over $A$ and $7$ over $B$.  Thus over $A$ we have $\zO_A^{14} (\rad P(5))=\rad P(5)$, and over $B$ we have  $\zO_B^{7}( \rad P(5))=\rad P(5)$. 
We point out that $\zO_A^7(\rad P(5))=\rad P(\zs 5).$ 
The top two rows below together form the resolution over $A$. The second row is the image of the first row under $\zs$. The third row is the corresponding resolution over $B$. It is obtained from the first row by replacing $P(i')$ by $P(i)$, for $i=3,4,5$,  replacing $P(i)$ by $P(i^+)\oplus P(i^-)$ for $i=1,2$, and reducing the arrow labels modulo 7. In order to keep it short we use the notation $P(i^\pm)=P(i^+)\oplus P(i^-)$, for $i=1,2$.

{\scriptsize \[
\xymatrix@C15pt@R10pt{
P(5')\ar[r]^(.5){(
10,13)}&P(4')\ar[r]^(.5){(
9,12)}&P(1)\ar[r]^(.45){(
8,11)}&P(3')\ar[r]^(.35){(
7,10)}&P(4')\oplus P(5)\ar[r]^(.6){(
6,9)}&P(3)
\ar[r]^(.45){(
5,8)}&P(2)\ar[r]^(.4){(
4,7)}&\rad P(5)\ar[r]&0
\\
P(5)\ar[r]^(.45){(
3,6)}&P(4)\ar[r]^(.45){(
2,5)}&P(1)\ar[r]^(.45){(
1,4)}&P(3)\ar[r]^(.34){(
14,3)}&P(4)\oplus P(5')\ar[r]^(.6){(
13,2)}&P(3')\ar[r]^(.5){(
12,1)}& P(2) \ar[r]^(.45){(
11,14)}&P(5')
\\ \\
P(5)\ar[r]^(.45){(
3,6)}&P(4)\ar[r]^(.45){(
2,5)}&P(1^\pm)\ar[r]^(.5){(
1,4)}&P(3)\ar[r]^(.34){(
7,3)}&P(4)\oplus P(5)\ar[r]^(.6){(
6,2)}&P(3)
\ar[r]^(.45){(
5,1)}&P(2^\pm)\ar[r]^(.45){(
4,7)}&\rad P(5)\ar[r]&0
}
\begin{array}
 {ccccccccccccccccccc}
\end{array}
\]
}

\subsubsection{Projective resolution of the radical of $P(3)$} \label{res3}The radical of $P(3)$ corresponds to the 2-diagonal $(2,7)$ in the 14-gon $\cals$ over the algebra $A$ and  to the 2-arc $(2,7)$ in the punctured 7-gon $\calp$ over the algebra $B$. The projective resolutions are  given by the rotation orbits of these curves in $\cals$ and $\calp$. Again, the period of the resolution is 14 over $A$ and $7$ over $B$.  The top two rows below is the resolution over $A$ and the bottom row is the corresponding resolution over $B$. Note that the correspondence between the resolutions is exactly as in  \ref{res5}.

{\scriptsize \[
\xymatrix@C10pt@R10pt{P(3')\ar[r]^(.4){(
7,12)}& P(1)\oplus P(5) \ar[r]^(.5){(
6,11)}&P(3)\oplus P(3')\ar[r]^(.5){(
5,10)}&P(4')\oplus P(2)\ar[r]^(.6){(
4,9)}&P(3)\ar[r]^(.35){(
3,8)}&P(4)\oplus P(2)\ar[r]^{(
2,7)}&P(1)\oplus P(5)\ar[r]^(.55){(
1,6)}&\rad P(3)
\ar[r]&0\\
P(3)\ar[r]^(.4){(
14,5)}& P(1)\oplus P(5') \ar[r]^(.5){(
13,4)}&P(3')\oplus P(3)\ar[r]^(.5){(
12,3)}&P(4)\oplus P(2)\ar[r]^(.6){(
11,2)}&P(3')\ar[r]^(.35){(
10,1)}&P(4')\oplus P(2)\ar[r]^{(
9,14)}&P(1)\oplus P(5')\ar[r]^(.55){(
8,13)}& P(3')
\\ \\
P(3)\ar[r]^(.37){(
7,5)}& P(1^\pm)\oplus  P(5) \ar[r]^(.51){(
6,4)}&P(3)\oplus P(3)\ar[r]^(.48){(
5,3)}&P(4)\oplus P(2^\pm)\ar[r]^(.65){(
4,2)}&P(3)\ar[r]^(.32){(
3,1)}&P(4)\oplus P(2^\pm)\ar[r]^(.5){(
2,7)}&P(1^\pm)\oplus P(5)\ar[r]^(.55){(
1,6)}&\rad P(3)
\ar[r]&0}
\begin{array}
 {ccccccccccccccccccc}
\end{array}
\]
}

\subsubsection{Projective resolution of the radicals of $P(1)$ and $P(1^+)$} \label{res1}The radical of $P(1)$ over the algebra $A$ corresponds to the 2-diagonal $(3,10)$ in the 14-gon $\cals$. This is different from the previous cases because $(3,10)$ is a diameter in $\cals$ and therefore it is fixed under $\zs=\zO^7$.  Therefore the projective resolution of $\rad P(1)$ has period 7. Over the algebra $B$, there are two corresponding vertices $1^+,1^-$ and the projective resolutions of the radicals of $P(1^+)$ and $P(1^-) $ are symmetric.  The radical of $P(1^+)$ corresponds  to the loop $(3,3)^+$ in the punctured 7-gon $\calp$. The syzygy functor $\zO$ acts on loops by moving the boundary endpoint to its clockwise neighbor (note that it doesn't change the sign since $(N+1)/2=4$ is even). Therefore the period of the projective resolution is $7$ over $B$.  The top row below is the resolution over $A$ and the bottom two rows are the corresponding resolution over $B$. The bottom row is obtained from the top row by removing $P(i')$ for $i=3,4,5$ and replacing $P(i)$ by $P(i^+)$ or $P(i^-)$, for $i=1,2$.

{\scriptsize \[
\xymatrix@C10pt@R10pt{
P(1)\ar[r]^(.5){(
8,1)}&P(2)\ar[r]^(.37){(
7,14)}& P(1)\oplus \atop P(5) \oplus P(5') \ar[r]^(.45){(
6,13)}&P(3)\oplus P(3')\ar[r]^(.5){(
5,12)}&P(1)\oplus P(2)\ar[r]^(.5){(
4,11)}&P(3)\oplus P(3')\ar[r]^(.52){(
3,10)}& P(2)\oplus \atop P(4)\oplus P(4')\ar[r]^{(
2,9)}&\rad P(1)
\ar[r]&0\\
\\
P(1^{{{\!\scalebox{.6}{+}}\!} })\ar[r]^(.5){(
1,1)^{{\!\scalebox{.6}{+}}\!}}&P(2^\textup{-} )\ar[r]^(.45){(
7,7)^{{{\!\scalebox{.6}{+}}\!} }}& P(1^\textup{-} )\oplus \atop P(5)  \ar[r]^(.5){(
6,6)^{{\!\scalebox{.6}{+}}\!}}&P(3)\ar[r]^(.35){(
5,5)^{{\!\scalebox{.6}{+}}\!} }&P(1^{{\!\scalebox{.6}{+}}\!})\oplus P(2^{\textup{-}} )\ar[r]^(.6){(
4,4)^{\!\scalebox{.6}{+}}\!}&P(3)\ar[r]^(.45){(
3,3)^{{\!\scalebox{.6}{+}}\!} }& P(2^{{\!\scalebox{.6}{+}}\!})\oplus \atop P(4)\ar[r]^
(.45){(
2,2)^{{\!\scalebox{.6}{+}}\!}}&\rad P(1^{{\!\scalebox{.6}{+}}\!})
\ar[r]&0\\
}
\]
}

\subsubsection{Auslander-Reiten quiver of $\cmp\,B$} 
To complete our study of Example~\ref{ex71}, we give the Auslander-Reiten quiver of the syzygy category $\cmp\,B$.

\newcommand{\km}{\scalebox{1}{-}}
\newcommand{\kp}{\scalebox{0.6}{+}}
{\scriptsize \[\xymatrix@R10pt@C4pt{
&{\begin{smallmatrix} 2^{\kp}\\3\ \, \\1^{\km}\end{smallmatrix}}\ar[ddr]&&&&
{\begin{smallmatrix} 1^{\km}\\4\ 2^{\km} \\3\ \end{smallmatrix}}\ar[ddr]
\\
&{\begin{smallmatrix} 2^{\km}\\3\ \, \\1^{\kp}\end{smallmatrix}}\ar[ddr]&&&&
{\begin{smallmatrix} 1^{\kp}\\4\ 2^{\kp} \\3\ \end{smallmatrix}}\ar[ddr]
\\
{\begin{smallmatrix} 3\ \, \\1^{\km}\end{smallmatrix}}\ar[dr]\ar[uur]
&&
{\begin{smallmatrix} 2^{\kp}\\3\ \end{smallmatrix}}\ar[dr]
&&
{\begin{smallmatrix} 2^{\km}4\\3\end{smallmatrix}}\ar[dr]\ar[uur]
&&
{\begin{smallmatrix} 1^{\km}\\42^{\km}2^{\kp}\\3\end{smallmatrix}}\ar[dr]
&&
{\begin{smallmatrix} 1^{\kp}5\\42^{\km}2^{\kp}\\3\end{smallmatrix}}\ar[dr]
&&
{\begin{smallmatrix} 1^{\km}\end{smallmatrix}}\ar[dr]
&&
{\begin{smallmatrix} 3\\1^{\kp}5\end{smallmatrix}}\ar[dr]
&&
{\begin{smallmatrix} 3\\1^{\km}\end{smallmatrix}}
&&
\\
{\begin{smallmatrix} 3\ \, \\1^{\kp}\end{smallmatrix}}\ar[r]\ar[uur]
&
{\begin{smallmatrix} 3 \end{smallmatrix}}\ar[r]\ar[ru]\ar[rd]
&
{\begin{smallmatrix} 2^{\km}\\3\ \end{smallmatrix}}\ar[r]
&
{\begin{smallmatrix} 2^{\kp}2^{\km}4\\33 \end{smallmatrix}}\ar[r]\ar[ru]\ar[rd]
&
{\begin{smallmatrix} 2^{\kp}4\\3\end{smallmatrix}}\ar[r]\ar[uur]
&
{\begin{smallmatrix} 2^{\kp}2^{\km}4\\3 \end{smallmatrix}}\ar[r]\ar[ru]\ar[rd]
&
{\begin{smallmatrix} 1^{\kp}\\42^{\kp}2^{\km}\\3\end{smallmatrix}}\ar[r]
&
{\begin{smallmatrix} 1^{\kp}1^{\km}5\\442^{\kp}2^{\kp}2^{\km}2^{\km}\\33 \end{smallmatrix}}\ar[r]\ar[ru]\ar[rd]
&
{\begin{smallmatrix} 1^{\km}5\\42^{\kp}2^{\km}\\3\end{smallmatrix}}\ar[r]
&
{\begin{smallmatrix} 1^{\kp}1^{\km}5\\42^{\kp}2^{\km}\\3 \end{smallmatrix}}\ar[r]\ar[ru]\ar[rd]
\ar[ddr]
&
{\begin{smallmatrix} 1^{\kp}\end{smallmatrix}}\ar[r]
&
{\begin{smallmatrix} 3\\1^{\kp}1^{\km}5\end{smallmatrix}}\ar[r]\ar[ru]\ar[rd]
&
{\begin{smallmatrix} 3\\1^{\km}5\end{smallmatrix}}\ar[r]
&
{\begin{smallmatrix} 33\\1^{\kp}1^{\km}5\end{smallmatrix}}\ar[r]\ar[ru]\ar[rd]
&
{\begin{smallmatrix} 3\\1^{\kp}\end{smallmatrix}}
\\
{\begin{smallmatrix} 3\\ 5\end{smallmatrix}}\ar[ru]\ar[rd]
&&
{\begin{smallmatrix} 4\\3\end{smallmatrix}}\ar[ru]
&&
{\begin{smallmatrix} 2^{\kp}2^{\km}\\3\end{smallmatrix}}\ar[ru]\ar[dr]
&&
{\begin{smallmatrix} 5\\42^{\kp}2^{\km}\\3 \end{smallmatrix}}\ar[ru]
&&
{\begin{smallmatrix} 1^{\kp}1^{\km}\\4 2^{\kp}2^{\km}\\3 \end{smallmatrix}}\ar[ru]
&&
{\begin{smallmatrix} 5 \end{smallmatrix}}\ar[ru]
&&
{\begin{smallmatrix} 3\\1^{\kp}1^{\km}\\\end{smallmatrix}}\ar[ru]
&&
{\begin{smallmatrix} 3\\5 \end{smallmatrix}}
\\
&{\begin{smallmatrix} 4\\3\\5\end{smallmatrix}}\ar[ru]
&&&&
{\begin{smallmatrix}5\\ 2^{\kp}2^{\km}\\3 \end{smallmatrix}}\ar[ru]
&&&&&
{\begin{smallmatrix} 3\\1^{\kp}1^{\km}5\\4 2^{\kp}2^{\km}\\3 \end{smallmatrix}}\ar[ruu]
&&
}
\]
}

\subsection{Example}\label{ex71a}
This example is less symmetric than the previous one. The left picture below shows the checkerboard polygon $\cals$ as well as the quiver of a dimer tree algebra $A$. The right picture shows the checkerboard punctured polygon $\calp$ and the quiver of the skew group algebra $B=\overline{e}\,AG\,\overline{e}$. 

\begin{center}
 \scalebox{0.57}{\Large
\begingroup%
  \makeatletter%
  \providecommand\color[2][]{%
    \errmessage{(Inkscape) Color is used for the text in Inkscape, but the package 'color.sty' is not loaded}%
    \renewcommand\color[2][]{}%
  }%
  \providecommand\transparent[1]{%
    \errmessage{(Inkscape) Transparency is used (non-zero) for the text in Inkscape, but the package 'transparent.sty' is not loaded}%
    \renewcommand\transparent[1]{}%
  }%
  \providecommand\rotatebox[2]{#2}%
  \newcommand*\fsize{\dimexpr\f@size pt\relax}%
  \newcommand*\lineheight[1]{\fontsize{\fsize}{#1\fsize}\selectfont}%
  \ifx\svgwidth\undefined%
    \setlength{\unitlength}{740.33842855bp}%
    \ifx\svgscale\undefined%
      \relax%
    \else%
      \setlength{\unitlength}{\unitlength * \real{\svgscale}}%
    \fi%
  \else%
    \setlength{\unitlength}{\svgwidth}%
  \fi%
  \global\let\svgwidth\undefined%
  \global\let\svgscale\undefined%
  \makeatother%
  \begin{picture}(1,0.47754408)%
    \lineheight{1}%
    \setlength\tabcolsep{0pt}%
    \put(0,0){\includegraphics[width=\unitlength,page=1]{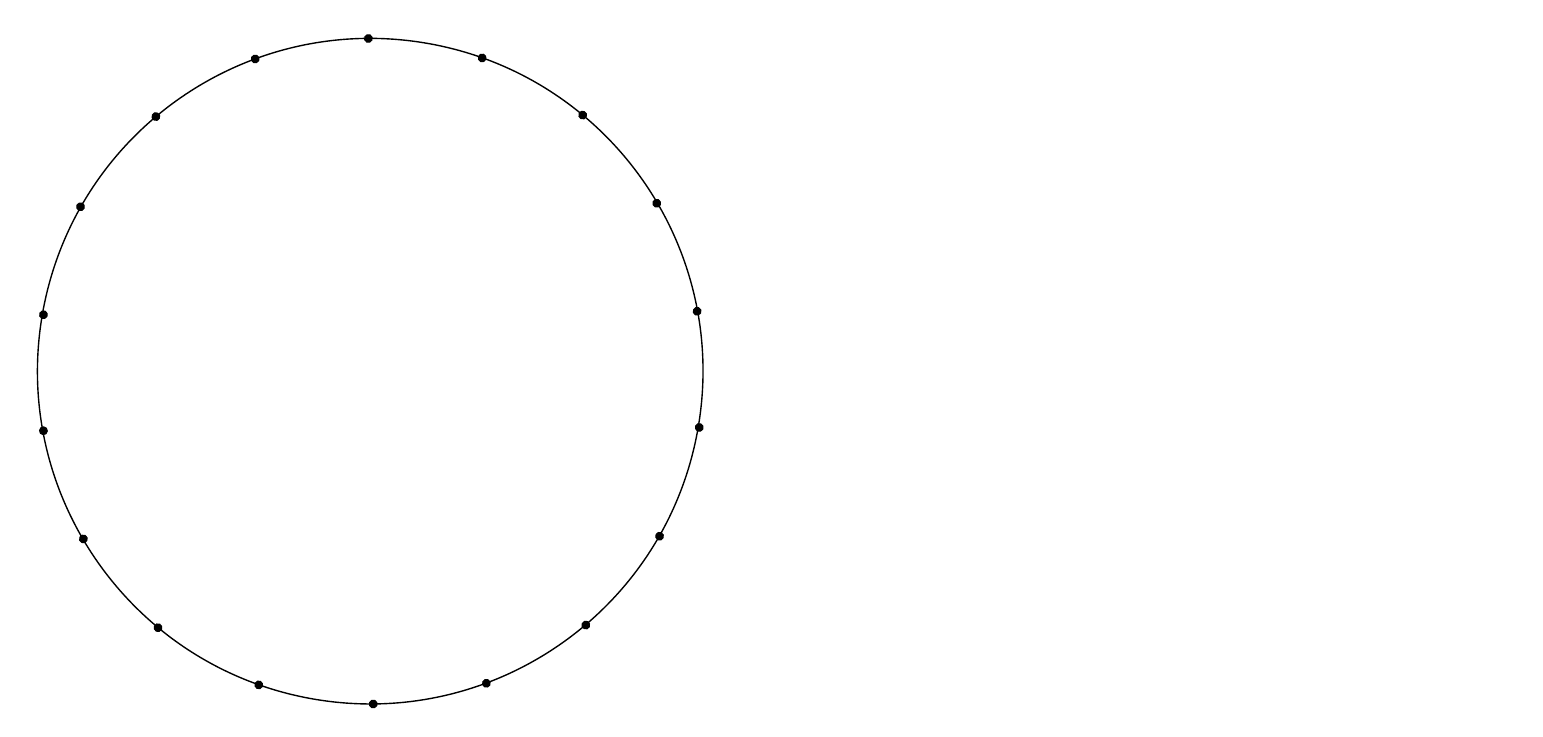}}%
    \put(0.0074037,0.272784){\makebox(0,0)[lt]{\lineheight{1.25}\smash{\begin{tabular}[t]{l}1\end{tabular}}}}%
    \put(0.03053587,0.34623272){\makebox(0,0)[lt]{\lineheight{1.25}\smash{\begin{tabular}[t]{l}2\end{tabular}}}}%
    \put(0.08530332,0.40996591){\makebox(0,0)[lt]{\lineheight{1.25}\smash{\begin{tabular}[t]{l}3\end{tabular}}}}%
    \put(0.15434876,0.44940219){\makebox(0,0)[lt]{\lineheight{1.25}\smash{\begin{tabular}[t]{l}4\end{tabular}}}}%
    \put(0.23299813,0.46200514){\makebox(0,0)[lt]{\lineheight{1.25}\smash{\begin{tabular}[t]{l}5\end{tabular}}}}%
    \put(0.31627394,0.45043906){\makebox(0,0)[lt]{\lineheight{1.25}\smash{\begin{tabular}[t]{l}6\end{tabular}}}}%
    \put(0.38584597,0.41175254){\makebox(0,0)[lt]{\lineheight{1.25}\smash{\begin{tabular}[t]{l}7\end{tabular}}}}%
    \put(0.43609856,0.34645588){\makebox(0,0)[lt]{\lineheight{1.25}\smash{\begin{tabular}[t]{l}8\end{tabular}}}}%
    \put(0.46293189,0.27370929){\makebox(0,0)[lt]{\lineheight{1.25}\smash{\begin{tabular}[t]{l}9\end{tabular}}}}%
    \put(0.4595179,0.19036952){\makebox(0,0)[lt]{\lineheight{1.25}\smash{\begin{tabular}[t]{l}10\end{tabular}}}}%
    \put(0.43343435,0.11641052){\makebox(0,0)[lt]{\lineheight{1.25}\smash{\begin{tabular}[t]{l}11\end{tabular}}}}%
    \put(0.38225642,0.05267741){\makebox(0,0)[lt]{\lineheight{1.25}\smash{\begin{tabular}[t]{l}12\end{tabular}}}}%
    \put(0.31037116,0.01156604){\makebox(0,0)[lt]{\lineheight{1.25}\smash{\begin{tabular}[t]{l}13\end{tabular}}}}%
    \put(0.23207285,0){\makebox(0,0)[lt]{\lineheight{1.25}\smash{\begin{tabular}[t]{l}14\end{tabular}}}}%
    \put(0.14972233,0.01324104){\makebox(0,0)[lt]{\lineheight{1.25}\smash{\begin{tabular}[t]{l}15\end{tabular}}}}%
    \put(0.0801503,0.05071516){\makebox(0,0)[lt]{\lineheight{1.25}\smash{\begin{tabular}[t]{l}16\end{tabular}}}}%
    \put(0.02469704,0.1176869){\makebox(0,0)[lt]{\lineheight{1.25}\smash{\begin{tabular}[t]{l}17\end{tabular}}}}%
    \put(-0.00220022,0.19135878){\makebox(0,0)[lt]{\lineheight{1.25}\smash{\begin{tabular}[t]{l}18\end{tabular}}}}%
    \put(0,0){\includegraphics[width=\unitlength,page=2]{example71a.pdf}}%
    \put(0.25630582,0.40712611){\makebox(0,0)[lt]{\lineheight{1.25}\smash{\begin{tabular}[t]{l}6\end{tabular}}}}%
    \put(0.33032876,0.3847437){\makebox(0,0)[lt]{\lineheight{1.25}\smash{\begin{tabular}[t]{l}6\end{tabular}}}}%
    \put(0.3709298,0.34883304){\makebox(0,0)[lt]{\lineheight{1.25}\smash{\begin{tabular}[t]{l}6\end{tabular}}}}%
    \put(0.09010528,0.12039881){\makebox(0,0)[lt]{\lineheight{1.25}\smash{\begin{tabular}[t]{l}6'\end{tabular}}}}%
    \put(0.14007077,0.07807503){\makebox(0,0)[lt]{\lineheight{1.25}\smash{\begin{tabular}[t]{l}6'\end{tabular}}}}%
    \put(0.20906857,0.05666569){\makebox(0,0)[lt]{\lineheight{1.25}\smash{\begin{tabular}[t]{l}6'\end{tabular}}}}%
    \put(0.3435699,0.35010939){\makebox(0,0)[lt]{\lineheight{1.25}\smash{\begin{tabular}[t]{l}2\end{tabular}}}}%
    \put(0.35438624,0.39018384){\makebox(0,0)[lt]{\lineheight{1.25}\smash{\begin{tabular}[t]{l}2\end{tabular}}}}%
    \put(0.31078615,0.33680433){\makebox(0,0)[lt]{\lineheight{1.25}\smash{\begin{tabular}[t]{l}2\end{tabular}}}}%
    \put(0.27961352,0.27220978){\makebox(0,0)[lt]{\lineheight{1.25}\smash{\begin{tabular}[t]{l}2\end{tabular}}}}%
    \put(0.19523691,0.20124981){\makebox(0,0)[lt]{\lineheight{1.25}\smash{\begin{tabular}[t]{l}2\end{tabular}}}}%
    \put(0.16915336,0.13943112){\makebox(0,0)[lt]{\lineheight{1.25}\smash{\begin{tabular}[t]{l}2\end{tabular}}}}%
    \put(0.13174318,0.11964906){\makebox(0,0)[lt]{\lineheight{1.25}\smash{\begin{tabular}[t]{l}2\end{tabular}}}}%
    \put(0.11630042,0.07651165){\makebox(0,0)[lt]{\lineheight{1.25}\smash{\begin{tabular}[t]{l}2\end{tabular}}}}%
    \put(0.31644947,0.41192808){\makebox(0,0)[lt]{\lineheight{1.25}\smash{\begin{tabular}[t]{l}5\end{tabular}}}}%
    \put(0.31228568,0.36473845){\makebox(0,0)[lt]{\lineheight{1.25}\smash{\begin{tabular}[t]{l}5\end{tabular}}}}%
    \put(0.34715947,0.32968913){\makebox(0,0)[lt]{\lineheight{1.25}\smash{\begin{tabular}[t]{l}5\end{tabular}}}}%
    \put(0.37601888,0.27776149){\makebox(0,0)[lt]{\lineheight{1.25}\smash{\begin{tabular}[t]{l}5\end{tabular}}}}%
    \put(0.38093239,0.32552534){\makebox(0,0)[lt]{\lineheight{1.25}\smash{\begin{tabular}[t]{l}3\end{tabular}}}}%
    \put(0.326165,0.30124471){\makebox(0,0)[lt]{\lineheight{1.25}\smash{\begin{tabular}[t]{l}3\end{tabular}}}}%
    \put(0.23363631,0.34218051){\makebox(0,0)[lt]{\lineheight{1.25}\smash{\begin{tabular}[t]{l}3\end{tabular}}}}%
    \put(0.12537777,0.36705167){\makebox(0,0)[lt]{\lineheight{1.25}\smash{\begin{tabular}[t]{l}3\end{tabular}}}}%
    \put(0.1805439,0.39989934){\makebox(0,0)[lt]{\lineheight{1.25}\smash{\begin{tabular}[t]{l}4\end{tabular}}}}%
    \put(0.14619672,0.33281607){\makebox(0,0)[lt]{\lineheight{1.25}\smash{\begin{tabular}[t]{l}4\end{tabular}}}}%
    \put(0.1018306,0.26862021){\makebox(0,0)[lt]{\lineheight{1.25}\smash{\begin{tabular}[t]{l}4\end{tabular}}}}%
    \put(0.0425646,0.27555987){\makebox(0,0)[lt]{\lineheight{1.25}\smash{\begin{tabular}[t]{l}4\end{tabular}}}}%
    \put(0.05690655,0.32298896){\makebox(0,0)[lt]{\lineheight{1.25}\smash{\begin{tabular}[t]{l}1\end{tabular}}}}%
    \put(0.10640938,0.3173733){\makebox(0,0)[lt]{\lineheight{1.25}\smash{\begin{tabular}[t]{l}1\end{tabular}}}}%
    \put(0.16412822,0.25879312){\makebox(0,0)[lt]{\lineheight{1.25}\smash{\begin{tabular}[t]{l}1\end{tabular}}}}%
    \put(0.30378258,0.20616336){\makebox(0,0)[lt]{\lineheight{1.25}\smash{\begin{tabular}[t]{l}1\end{tabular}}}}%
    \put(0.36306487,0.14978481){\makebox(0,0)[lt]{\lineheight{1.25}\smash{\begin{tabular}[t]{l}1\end{tabular}}}}%
    \put(0.40833997,0.14573261){\makebox(0,0)[lt]{\lineheight{1.25}\smash{\begin{tabular}[t]{l}1\end{tabular}}}}%
    \put(0.09946974,0.34900858){\makebox(0,0)[lt]{\lineheight{1.25}\smash{\begin{tabular}[t]{l}7\end{tabular}}}}%
    \put(0.06836104,0.2951664){\makebox(0,0)[lt]{\lineheight{1.25}\smash{\begin{tabular}[t]{l}7\end{tabular}}}}%
    \put(0.05921976,0.22825865){\makebox(0,0)[lt]{\lineheight{1.25}\smash{\begin{tabular}[t]{l}7\end{tabular}}}}%
    \put(0.241342,0.12733846){\makebox(0,0)[lt]{\lineheight{1.25}\smash{\begin{tabular}[t]{l}3'\end{tabular}}}}%
    \put(0.34455915,0.09599032){\makebox(0,0)[lt]{\lineheight{1.25}\smash{\begin{tabular}[t]{l}3'\end{tabular}}}}%
    \put(0.14793569,0.17025269){\makebox(0,0)[lt]{\lineheight{1.25}\smash{\begin{tabular}[t]{l}3'\end{tabular}}}}%
    \put(0.09669388,0.14498286){\makebox(0,0)[lt]{\lineheight{1.25}\smash{\begin{tabular}[t]{l}3'\end{tabular}}}}%
    \put(0.08426645,0.19783577){\makebox(0,0)[lt]{\lineheight{1.25}\smash{\begin{tabular}[t]{l}5'\end{tabular}}}}%
    \put(0.12740387,0.14156882){\makebox(0,0)[lt]{\lineheight{1.25}\smash{\begin{tabular}[t]{l}5'\end{tabular}}}}%
    \put(0.16776541,0.10235572){\makebox(0,0)[lt]{\lineheight{1.25}\smash{\begin{tabular}[t]{l}5'\end{tabular}}}}%
    \put(0.15765122,0.05210316){\makebox(0,0)[lt]{\lineheight{1.25}\smash{\begin{tabular}[t]{l}5'\end{tabular}}}}%
    \put(0.2964442,0.06679615){\makebox(0,0)[lt]{\lineheight{1.25}\smash{\begin{tabular}[t]{l}4'\end{tabular}}}}%
    \put(0.3335673,0.13249148){\makebox(0,0)[lt]{\lineheight{1.25}\smash{\begin{tabular}[t]{l}4'\end{tabular}}}}%
    \put(0.36381465,0.18395648){\makebox(0,0)[lt]{\lineheight{1.25}\smash{\begin{tabular}[t]{l}4'\end{tabular}}}}%
    \put(0.41869363,0.18505729){\makebox(0,0)[lt]{\lineheight{1.25}\smash{\begin{tabular}[t]{l}4'\end{tabular}}}}%
    \put(0.40146424,0.23334773){\makebox(0,0)[lt]{\lineheight{1.25}\smash{\begin{tabular}[t]{l}7'\end{tabular}}}}%
    \put(0.39538593,0.16441386){\makebox(0,0)[lt]{\lineheight{1.25}\smash{\begin{tabular}[t]{l}7'\end{tabular}}}}%
    \put(0.36722863,0.11722424){\makebox(0,0)[lt]{\lineheight{1.25}\smash{\begin{tabular}[t]{l}7'\end{tabular}}}}%
    \put(0,0){\includegraphics[width=\unitlength,page=3]{example71a.pdf}}%
    \put(0.53418973,0.272784){\makebox(0,0)[lt]{\lineheight{1.25}\smash{\begin{tabular}[t]{l}1\end{tabular}}}}%
    \put(0.61178105,0.41049437){\makebox(0,0)[lt]{\lineheight{1.25}\smash{\begin{tabular}[t]{l}2\end{tabular}}}}%
    \put(0.76021803,0.46551412){\makebox(0,0)[lt]{\lineheight{1.25}\smash{\begin{tabular}[t]{l}3\end{tabular}}}}%
    \put(0.90877367,0.41237005){\makebox(0,0)[lt]{\lineheight{1.25}\smash{\begin{tabular}[t]{l}4\end{tabular}}}}%
    \put(0.99084298,0.27264){\makebox(0,0)[lt]{\lineheight{1.25}\smash{\begin{tabular}[t]{l}5\end{tabular}}}}%
    \put(0.96178097,0.11497146){\makebox(0,0)[lt]{\lineheight{1.25}\smash{\begin{tabular}[t]{l}6\end{tabular}}}}%
    \put(0.8451033,0.00984491){\makebox(0,0)[lt]{\lineheight{1.25}\smash{\begin{tabular}[t]{l}7\end{tabular}}}}%
    \put(0.68296575,0.01207748){\makebox(0,0)[lt]{\lineheight{1.25}\smash{\begin{tabular}[t]{l}8\end{tabular}}}}%
    \put(0.558403,0.11359976){\makebox(0,0)[lt]{\lineheight{1.25}\smash{\begin{tabular}[t]{l}9\end{tabular}}}}%
    \put(0,0){\includegraphics[width=\unitlength,page=4]{example71a.pdf}}%
    \put(0.62836904,0.36455858){\makebox(0,0)[lt]{\lineheight{1.25}\smash{\begin{tabular}[t]{l}$1^\pm$\end{tabular}}}}%
    \put(0.69523033,0.32606267){\makebox(0,0)[lt]{\lineheight{1.25}\smash{\begin{tabular}[t]{l}$1^\pm$\end{tabular}}}}%
    \put(0.72156928,0.25312305){\makebox(0,0)[lt]{\lineheight{1.25}\smash{\begin{tabular}[t]{l}$1^\pm$\end{tabular}}}}%
    \put(0.75601252,0.19436608){\makebox(0,0)[lt]{\lineheight{1.25}\smash{\begin{tabular}[t]{l}$2^\pm$\end{tabular}}}}%
    \put(0.80261308,0.14573971){\makebox(0,0)[lt]{\lineheight{1.25}\smash{\begin{tabular}[t]{l}$2^\pm$\end{tabular}}}}%
    \put(0.79653486,0.09913941){\makebox(0,0)[lt]{\lineheight{1.25}\smash{\begin{tabular}[t]{l}$2^\pm$\end{tabular}}}}%
    \put(0.83705633,0.06064353){\makebox(0,0)[lt]{\lineheight{1.25}\smash{\begin{tabular}[t]{l}$2^\pm$\end{tabular}}}}%
    \put(0.586093,0.27376191){\makebox(0,0)[lt]{\lineheight{1.25}\smash{\begin{tabular}[t]{l}4\end{tabular}}}}%
    \put(0.66233952,0.26751756){\makebox(0,0)[lt]{\lineheight{1.25}\smash{\begin{tabular}[t]{l}4\end{tabular}}}}%
    \put(0.77766321,0.32276821){\makebox(0,0)[lt]{\lineheight{1.25}\smash{\begin{tabular}[t]{l}4\end{tabular}}}}%
    \put(0.87259507,0.35839725){\makebox(0,0)[lt]{\lineheight{1.25}\smash{\begin{tabular}[t]{l}4\end{tabular}}}}%
    \put(0.71471087,0.40306693){\makebox(0,0)[lt]{\lineheight{1.25}\smash{\begin{tabular}[t]{l}7\end{tabular}}}}%
    \put(0.63432907,0.3152679){\makebox(0,0)[lt]{\lineheight{1.25}\smash{\begin{tabular}[t]{l}7\end{tabular}}}}%
    \put(0.60486141,0.19743235){\makebox(0,0)[lt]{\lineheight{1.25}\smash{\begin{tabular}[t]{l}7\end{tabular}}}}%
    \put(0.79509262,0.39102866){\makebox(0,0)[lt]{\lineheight{1.25}\smash{\begin{tabular}[t]{l}3\end{tabular}}}}%
    \put(0.83408677,0.24672304){\makebox(0,0)[lt]{\lineheight{1.25}\smash{\begin{tabular}[t]{l}3\end{tabular}}}}%
    \put(0.76651341,0.14794029){\makebox(0,0)[lt]{\lineheight{1.25}\smash{\begin{tabular}[t]{l}3\end{tabular}}}}%
    \put(0.73319488,0.07728632){\makebox(0,0)[lt]{\lineheight{1.25}\smash{\begin{tabular}[t]{l}3\end{tabular}}}}%
    \put(0.91880513,0.20051303){\makebox(0,0)[lt]{\lineheight{1.25}\smash{\begin{tabular}[t]{l}6\end{tabular}}}}%
    \put(0.85362533,0.09422999){\makebox(0,0)[lt]{\lineheight{1.25}\smash{\begin{tabular}[t]{l}6\end{tabular}}}}%
    \put(0.7693927,0.05908674){\makebox(0,0)[lt]{\lineheight{1.25}\smash{\begin{tabular}[t]{l}6\end{tabular}}}}%
    \put(0.90879289,0.11579468){\makebox(0,0)[lt]{\lineheight{1.25}\smash{\begin{tabular}[t]{l}5\end{tabular}}}}%
    \put(0.83505818,0.1351319){\makebox(0,0)[lt]{\lineheight{1.25}\smash{\begin{tabular}[t]{l}5\end{tabular}}}}%
    \put(0.78171536,0.11579468){\makebox(0,0)[lt]{\lineheight{1.25}\smash{\begin{tabular}[t]{l}5\end{tabular}}}}%
    \put(0.66870214,0.13436171){\makebox(0,0)[lt]{\lineheight{1.25}\smash{\begin{tabular}[t]{l}5\end{tabular}}}}%
  \end{picture}%
\endgroup%
}
\end{center}
\[\xymatrix{
7\ar[r]&4\ar[d]&3\ar[l]\ar[r]&5\ar[ld]&6\ar[l]\\
&1\ar[lu]\ar[r]^\za\ar[ld]&2\ar[d]\ar[u]\ar[rru]\ar[rrd]\\
7'\ar[r]&4'\ar[u]&3'\ar[l]\ar[r]&5'\ar[lu]&6'\ar[l]\\
}
\qquad\qquad\qquad
\xymatrix{
7\ar[r]&4\ar[d]\ar@/_10pt/[dd]&3\ar[l]\ar[r]&5\ar@/_
0pt/[ddl]\ar[ld]&6\ar[l]\\
&1^+\ar[lu]\ar[r]^{\za^+}&2^+\ar[u]\ar[rru]\\
&1^-\ar[luu]\ar[r]^{\za^-}&2^-\ar@/^10pt/[uu]\ar[rruu]\\
}\]

\subsection{Algebra of syzygy type $\mathbb{D}$ that is not a skew group algebra}\label{ex72}
Let $B$ be the Jacobian algebra of the following quiver where the potential is given by the sum of all chordless cycles. Note that $B$ is not a skew group algebra of a dimer tree algebra. 
Nevertheless, by directly computing the syzygy category of $B$ we obtain that it is a 2-cluster category of type $\mathbb{D}_{6}$, so it can be modeled by punctured polygon of size 11.  In the polygon we can place arcs corresponding to radicals of projective modules and obtain a checkerboard pattern shown below, where shaded regions together with the arcs attached to the puncture with the same sign give chordless cycles in the quiver.  Moreover, as in the previous examples, the projective presentation of an indecomposable syzygy $M_{\tilde \zg}$ is given by intersection patterns of the 2-arc $\tilde\zg$ with the arcs of the checkerboard polygon.

\begin{minipage}{0.5\textwidth}
\[\xymatrix{
&5\ar[dl]\\
6\ar[dr]&&2\ar[dd]\ar[ul]\\
&1\ar[ur]\ar[dl]&&4\ar[ul]\\
7\ar[uu]\ar[d]&&3\ar[ur]\ar[ul]\\
8\ar[rr]&&9\ar[u]
}
\]
\end{minipage}
\begin{minipage}{0.5\textwidth}
\begin{center}
\scalebox{0.67}{\Large
\begingroup%
  \makeatletter%
  \providecommand\color[2][]{%
    \errmessage{(Inkscape) Color is used for the text in Inkscape, but the package 'color.sty' is not loaded}%
    \renewcommand\color[2][]{}%
  }%
  \providecommand\transparent[1]{%
    \errmessage{(Inkscape) Transparency is used (non-zero) for the text in Inkscape, but the package 'transparent.sty' is not loaded}%
    \renewcommand\transparent[1]{}%
  }%
  \providecommand\rotatebox[2]{#2}%
  \newcommand*\fsize{\dimexpr\f@size pt\relax}%
  \newcommand*\lineheight[1]{\fontsize{\fsize}{#1\fsize}\selectfont}%
  \ifx\svgwidth\undefined%
    \setlength{\unitlength}{244.50155914bp}%
    \ifx\svgscale\undefined%
      \relax%
    \else%
      \setlength{\unitlength}{\unitlength * \real{\svgscale}}%
    \fi%
  \else%
    \setlength{\unitlength}{\svgwidth}%
  \fi%
  \global\let\svgwidth\undefined%
  \global\let\svgscale\undefined%
  \makeatother%
  \begin{picture}(1,0.99351652)%
    \lineheight{1}%
    \setlength\tabcolsep{0pt}%
    \put(0,0){\includegraphics[width=\unitlength,page=1]{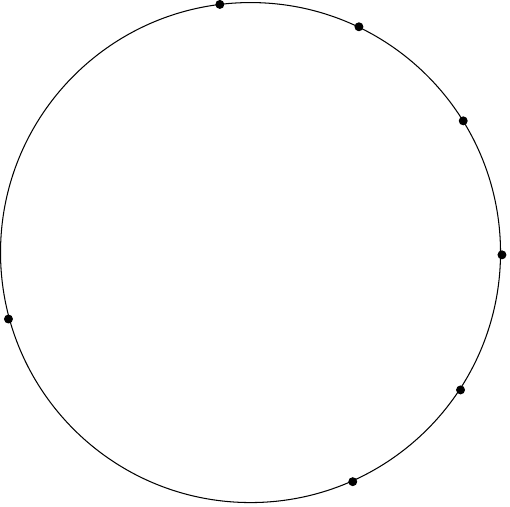}}%
    \put(0.86081332,0.68824158){\makebox(0,0)[lt]{\lineheight{1.25}\smash{\begin{tabular}[t]{l}2\end{tabular}}}}%
    \put(0.76734336,0.69902541){\makebox(0,0)[lt]{\lineheight{1.25}\smash{\begin{tabular}[t]{l}5\end{tabular}}}}%
    \put(0,0){\includegraphics[width=\unitlength,page=2]{example15.pdf}}%
    \put(0.89612863,0.44574399){\makebox(0,0)[lt]{\lineheight{1.25}\smash{\begin{tabular}[t]{l}3\end{tabular}}}}%
    \put(0.4093652,0.07806938){\makebox(0,0)[lt]{\lineheight{1.25}\smash{\begin{tabular}[t]{l}3\end{tabular}}}}%
    \put(0.56338705,0.25101277){\makebox(0,0)[lt]{\lineheight{1.25}\smash{\begin{tabular}[t]{l}3\end{tabular}}}}%
    \put(0.17919423,0.18454319){\makebox(0,0)[lt]{\lineheight{1.25}\smash{\begin{tabular}[t]{l}8\\\end{tabular}}}}%
    \put(0.57765487,0.11577608){\makebox(0,0)[lt]{\lineheight{1.25}\smash{\begin{tabular}[t]{l}9\\\end{tabular}}}}%
    \put(0.49535621,0.86146761){\makebox(0,0)[lt]{\lineheight{1.25}\smash{\begin{tabular}[t]{l}5\end{tabular}}}}%
    \put(0.28174148,0.57111807){\makebox(0,0)[lt]{\lineheight{1.25}\smash{\begin{tabular}[t]{l}$7^+$\end{tabular}}}}%
    \put(0.50148999,0.65092111){\makebox(0,0)[lt]{\lineheight{1.25}\smash{\begin{tabular}[t]{l}$6^+$\end{tabular}}}}%
    \put(0.65387388,0.3159317){\makebox(0,0)[lt]{\lineheight{1.25}\smash{\begin{tabular}[t]{l}$1^+$\end{tabular}}}}%
    \put(0,0){\includegraphics[width=\unitlength,page=3]{example15.pdf}}%
    \put(0.81264683,0.3845664){\makebox(0,0)[lt]{\lineheight{1.25}\smash{\begin{tabular}[t]{l}2\end{tabular}}}}%
    \put(0.69753888,0.17327168){\makebox(0,0)[lt]{\lineheight{1.25}\smash{\begin{tabular}[t]{l}2\end{tabular}}}}%
    \put(0.30698512,0.17830176){\makebox(0,0)[lt]{\lineheight{1.25}\smash{\begin{tabular}[t]{l}9\\\end{tabular}}}}%
    \put(0.09095991,0.3422916){\makebox(0,0)[lt]{\lineheight{1.25}\smash{\begin{tabular}[t]{l}9\\\end{tabular}}}}%
    \put(0.13173783,0.43375742){\makebox(0,0)[lt]{\lineheight{1.25}\smash{\begin{tabular}[t]{l}8\\\end{tabular}}}}%
    \put(0.18219628,0.70654838){\makebox(0,0)[lt]{\lineheight{1.25}\smash{\begin{tabular}[t]{l}8\\\end{tabular}}}}%
    \put(0.89134834,0.5176388){\makebox(0,0)[lt]{\lineheight{1.25}\smash{\begin{tabular}[t]{l}5\end{tabular}}}}%
    \put(0,0){\includegraphics[width=\unitlength,page=4]{example15.pdf}}%
    \put(0.82890394,0.29151534){\makebox(0,0)[lt]{\lineheight{1.25}\smash{\begin{tabular}[t]{l}$1^+$\end{tabular}}}}%
    \put(0,0){\includegraphics[width=\unitlength,page=5]{example15.pdf}}%
    \put(0.69178879,0.28038765){\makebox(0,0)[lt]{\lineheight{1.25}\smash{\begin{tabular}[t]{l}$4^-$\end{tabular}}}}%
    \put(0.87266853,0.25673534){\makebox(0,0)[lt]{\lineheight{1.25}\smash{\begin{tabular}[t]{l}$4^-$\end{tabular}}}}%
  \end{picture}%
\endgroup%
}
\end{center}
\end{minipage}

\subsection{Application to dimer models} \label{sect 74}
In  \cite[Figure 2]{SS4}, we explained in an example how to associate  a consistent dimer model to every dimer tree algebra. This is done by embedding the checkerboard polygon into a unique alternating strand diagram on a slightly larger disk. The mutable part of the quiver of the alternating strand diagram is equal to the twist of the quiver of the dimer tree algebra in the sense of Bocklandt \cite{Bocklandt}. It is natural to ask what happens to the dimer model under the action of $G$. A similar situation was studied recently by Baur, Pasquali and Velasco in \cite{BPV}, who construct orbifold diagrams using a rotation action on an alternating strand diagram on a polygon. The situation here is similar but not the same. The reason is that our rotation $\zs$ does not preserve the orientation of the diameters.

In the example in Figure~\ref{figdimer1}, we show how this correspondence behaves  under the $G$-action.  Let $A$ denote the dimer tree algebra from the example  in section~\ref{ex71a}. Then the top right picture in Figure~\ref{figdimer1} shows the alternating strand diagram that contains the checkerboard polygon of $A$. 
 The top left picture shows the quiver $\widetilde Q$ of the corresponding dimer model; its vertices correspond to the white regions in the polygon, two vertices are connected by an arrow if the corresponding regions share a vertex, and the direction of the arrow is determined by the orientation of the strands in the polygon. The vertices corresponding to the boundary regions are called frozen vertices and the remaining vertices 1,2,3,4,5,6,7,3',4',5',6',7' are called mutable vertices. The full subquiver of $\widetilde Q$ on the mutable vertices is called  the twist  of the quiver $Q$ of the dimer tree algebra $A$. 

Note that in this polygon  the action of $\zs$ is given by a rotation by angle $\pi$  as indicated by the labels of the white regions and the vertices. 
 This symmetry induces a symmetry on the twisted quiver as well, but it reverses the direction of the arrows.
 
 The bottom right picture in Figure~\ref{figdimer1}  shows the resulting  strand diagram on the punctured polygon obtained by the rotational symmetry.  In order to reconcile the problem that $\zs$ does not preserve the orientations, we add the so-called rift in the punctured polygon that goes from the puncture to the boundary segment between vertices 8 and 9.  Thus, the white regions with labels 1,2,8, 8' in the polygon give white regions with labels $1^\pm, 2^\pm, 8^\pm, 8'^\pm$ in the punctured polygon with rift.  
 The rift allows us to orient strands such that each strand that passes through the rift changes orientation, which ensures that the boundary of white regions is alternating and the boundary of shaded regions is oriented.

One may want to associate a quiver to this dimer model which is shown on the bottom left, where the rift behaves in a similar way as the boundary.
Note that this quiver is almost the one corresponding to the $G$-action on the quiver $\widetilde Q$, except that we also have arrows $8^\pm \to 8'^\pm$.  

\begin{remark}
There are strong connections between dimer models on a disk and Cohen-Macaulay categories, see \cite{BKM, Pr}.  In particular, a consistent dimer model on a disk yields a quiver $Q$ and the associated Jacobian algebra $A_Q$, with frozen vertices coming from the boundary regions of the dimer model.  Let $e$ be the sum of the primitive idempotents associated to the frozen vertices, then $eA_Qe$ is called the boundary algebra of $A_Q$.   It is known that the Cohen-Macaulay category of the boundary algebra is 2-Calabi-Yau.  Moreover, it is a 1-cluster category that provides an additive categorification of cluster algebras coming from coordinate rings of (open) positroid varieties.   On the other hand we are studying something different here, because for certain dimer models we are considering the Cohen-Macaulay category of the dimer tree algebra $A_Q/\langle e\rangle$ obtained by removing the frozen vertices, and show that it is a 2-cluster category of type $\mathbb{A}_r$.  The rank $r$ of the category is related to $2(r+2)$ marked points on the boundary of the disk of the dimer model.      
\end{remark}

\begin{figure}
 \begin{minipage}{0.5\textwidth}
\small\[\xymatrix@R8pt@C5pt{
&&&&13\ar[r]\ar[ld]&14\ar[d]\\
&&&12\ar[rd]&&6\ar[r]\ar[ld]&15\ar[lu]\ar[rd]\\
&&11\ar[ld]\ar[ru]&&3\ar[uu]\ar[rr]\ar[ld]&&5\ar[lu]\ar[dd]&16\ar[l]\\
&10\ar[rd]&&4\ar[lu]\ar[rdd]&&&&&8'\ar[lu]\ar[dr]\\
&&7\ar[lld]\ar[ru]&&&&2\ar[lluu]\ar[rru]\ar[dd]&&&9'\ar[lld]\\
9\ar[uur]\ar[dr]&&&&1\ar[llu]\ar[rru]\ar[ldd]&&&7'\ar[lu]\ar[rd]\\
&8\ar[rrru]&&&&&4'\ar[ru]\ar[ld]&&10'\ar[uur]\ar[ld]\\
&&16'\ar[lu]\ar[rd]&5'\ar[l]\ar[rr]&&3'\ar[luu]\ar[ld]\ar[rd]&&11'\ar[lu]\\
&&&15'\ar[r]&6'\ar[lu]\ar[d]&&12'\ar[ru]\ar[ld]\\
&&&&14'\ar[lu]\ar[r]&13'\ar[uu]
}
\]
\vspace{1in}
\[\xymatrix@R8pt@C5pt{
&&&&13\ar[r]\ar[ld]&14\ar[d]\\
&&&12\ar[rd]&&6\ar[r]\ar[ld]&15\ar[lu]\ar[rd]\\
&&11\ar[ld]\ar[ru]&&3\ar[uu]\ar[rr]\ar[ld]&&5\ar[lu]\ar[dd]\ar@/^10pt/[ddd]&16\ar[l]\\
&10\ar[rd]&&4\ar[lu]\ar[rdd]\ar@/_10pt/[rddd]&&&&&8'^+\ar[lu]\ar@/^17pt/[dddlllllll]\\
&&7\ar[lld]\ar[ru]&&&&2^+\ar[lluu]\ar[rru]&&8'^-\ar[luu]\ar@/^17pt/[dddlllllll]\\
9\ar[ddr]\ar[rd]\ar[uur]&&&&1^+\ar[llu]\ar[rru]&&2^-\ar@/^10pt/[lluuu]\ar[rru]\\
&8^+\ar[rrru]&&&1^-\ar@/^10pt/[lluu]\ar[rru]\\
&8^-\ar[rrru] &&\\
}
\]
\end{minipage}
\begin{minipage}{0.4\textwidth}
\begin{center}
\scalebox{0.6}{ 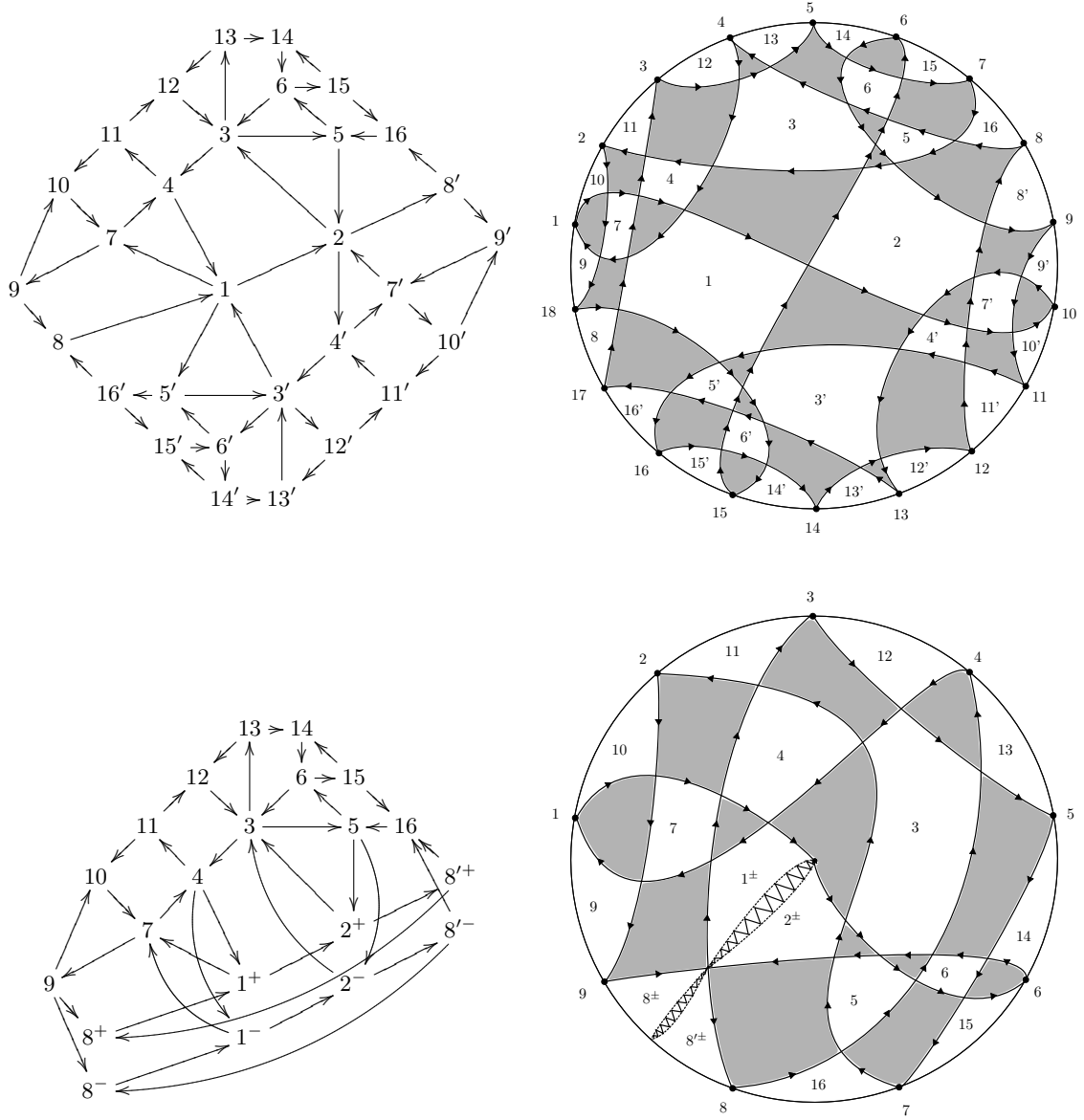}
\end{center}
\end{minipage}
\caption{The dimer models of the example in section \ref{ex71a}.}
\label{figdimer1}

\vspace{1in}
\end{figure}


\subsection{Algebras of syzygy type $\mathbb{E}$}\label{sect 75} 
It is natural to ask if there are 2-Calabi-Yau tilted algebras that are of syzygy type $\mathbb{E}_n$, with $n=6,7,8$. We give here three examples.   
 The Jacobian algebras of the following quivers with the potential given by the sum of the chordless cycles have syzygy categories of types $\mathbb{E}_6, \mathbb{E}_7, \mathbb{E}_8$ respectively.  Observe that moving from left to right we increase the length of the chordless cycle at the top right by adding a vertex.  Moreover, if this chordless where to have length three, that is the quiver would be a collection of six 3-cycles joined together at an interior vertex, then it would be the same quiver as shown in Example~\ref{ex71} on the right whose corresponding syzygy category is of type $\mathbb{D}_4$. These examples were checked by computer using the QPA package in \cite{GAP}.

\[\xymatrix@R=10pt@C=10pt{&6\ar[ld]\ar[rr]&&7\ar[ld]&8\ar[l]\\
4\ar[rr]&&3\ar[ld]\ar[ul]\ar[rr]&&5\ar[ld]\ar[u]\\
&1\ar[lu]\ar[rr]&&2\ar[lu]\\
}
\qquad
\xymatrix@R=10pt@C=10pt{&6\ar[ld]\ar[rr]&&7\ar[ld]&8\ar[l]&9\ar[l]\\
4\ar[rr]&&3\ar[ld]\ar[ul]\ar[rr]&&5\ar[ld]\ar[ur]\\
&1\ar[lu]\ar[rr]&&2\ar[lu]\\
}
\qquad
\xymatrix@R=10pt@C=10pt{&6\ar[ld]\ar[rr]&&7\ar[ld]&8\ar[l]&9\ar[l]&10\ar[l]\\
4\ar[rr]&&3\ar[ld]\ar[ul]\ar[rr]&&5\ar[ld]\ar[urr]\\
&1\ar[lu]\ar[rr]&&2\ar[lu]\\
}
\]

\end{document}